\documentclass[reqno]{amsart}

\usepackage{amsmath,amsthm,amssymb,amsfonts,enumerate,graphicx,color,pstricks}
\numberwithin{equation}{section} 
\theoremstyle{plain}
\newtheorem{theo+}           {Theorem}      [section]
\newtheorem{prop+}  [theo+]  {Proposition}
\newtheorem{coro+}  [theo+]  {Corollary}
\newtheorem{lemm+}  [theo+]  {Lemma}
\newtheorem{defi+}  [theo+]  {Definition}
\newtheorem{conj+}  [theo+]  {Conjecture}

\theoremstyle{definition}
\newtheorem{rema+}  [theo+]  {Remark}
\newtheorem{prob+}  [theo+]  {Problem}
\newtheorem{exam+}  [theo+]  {Example}

\newenvironment{theorem}{\begin{theo+}}{\end{theo+}}
\newenvironment{proposition}{\begin{prop+}}{\end{prop+}}
\newenvironment{corollary}{\begin{coro+}}{\end{coro+}}
\newenvironment{lemma}{\begin{lemm+}}{\end{lemm+}}

\newenvironment{remark}{\begin{rema+}}{\end{rema+}}
\newenvironment{definition}{\begin{defi+}}{\end{defi+}}
\newenvironment{conjecture}{\begin{conj+}}{\end{conj+}}

\newcommand{\al}{\alpha}
\newcommand{\be}{\beta}
\newcommand{\la}{\lambda}
\newcommand{\om}{\omega}
\newcommand{\tha}{\theta}

\newcommand{\iceR}{\psline{->}(0,0)(.7,0) \psline(1,0)}
\newcommand{\iceL}{\psline{->}(0,0)(-.7,0) \psline(-1,0)}
\newcommand{\iceU}{\psline{->}(0,0)(0,.7) \psline(0,1)}
\newcommand{\iceD}{\psline{->}(0,0)(0,-.7) \psline(0,-1)}

\begin{document}

\baselineskip 18pt
\larger[2]
\title
[Three-colour model with domain wall boundary conditions] 
{The three-colour model with\\ domain wall boundary conditions}
\author{Hjalmar Rosengren}
\address
{Department of Mathematical Sciences
\\ Chalmers University of Technology\\SE-412~96 G\"oteborg, Sweden}
\email{hjalmar@chalmers.se}
\urladdr{http://www.math.chalmers.se/{\textasciitilde}hjalmar}

\thanks{Research  supported by the Swedish Science Research
Council (Vetenskapsr\aa det)}

\dedicatory{\Large Dedicated to Dennis Stanton\\
on his 60th birthday}

\begin{abstract}
We study the partition function for the three-colour model with domain wall boundary conditions. We  express it in terms of certain special polynomials, which can be constructed recursively. Our method generalizes
 Kuperberg's proof of the alternating sign matrix theorem, replacing the   
 six-vertex model used by Kuperberg with the  eight-vertex-solid-on-solid model.
As applications, we obtain some combinatorial results
 on  three-colourings. We also conjecture an explicit formula for the free energy of the model.
\end{abstract}

\maketitle

\section{Introduction}

An \emph{alternating sign matrix} is a square matrix with entries $1$, $-1$  and $0$, such that the non-zero entries in each row and column alternate in sign and add up to $1$.  Mills, Robbins and Rumsey \cite{mrr} conjectured 
that the number of $n\times n$ alternating sign matrices is
\begin{equation}\label{an}A_n=\frac {1!4!7!\dotsm(3n-2)!}{n!(n+1)!\dotsm(2n-1)!}.\end{equation}
This  conjecture was proved by Zeilberger \cite{z1}. Soon afterwards, a much simpler proof was found by Kuperberg \cite{ku}, using the
\emph{six-vertex model} of statistical mechanics.

The six-vertex model is an example of an \emph{ice model}, whose states can
be identified with what we call ice graphs;
see \eqref{ig} below.
Alternating sign matrices can be identified with 
ice graphs satisfying \emph{domain wall boundary conditions}.
For the six-vertex model, there is a closed formula for the corresponding partition function, the \emph{Izergin--Korepin determinant} \cite{i,ik}.
Kuperberg observed that in a special case, when all parameters of the model are  cubic roots of unity, the partition function  simply counts the number of states. He could then prove \eqref{an} by computing the corresponding limit of the determinant.

There is a natural two-parameter extension of the six-vertex model known as the eight-vertex-solid-on-solid (8VSOS) model. This model was introduced by Baxter \cite{b} as a tool for solving the eight-vertex model. It is elliptic, that is, the Boltzmann weights are elliptic functions of the parameters. We stress that, in contrast to the eight-vertex model, the 8VSOS model is an ice model. 
In particular, for domain wall boundary conditions, its states can be identified with alternating sign matrices.

It is natural to ask what happens to the 8VSOS model under Kuperberg's specialization of the parameters. The answer turns out to be very satisfactory: it degenerates to the \emph{three-colour model}. It is an observation of Lenard   that ice graphs  are in bijection with three-colourings of a square lattice, such that adjacent squares have distinct colour \cite{li}. The three-colour model is defined by assigning independent weights to the three colours; the partition function is simply the corresponding generating function \cite{b2}.
Thus, one may hope that 
extending Kuperberg's work to the 8VSOS model would lead to new applications of statistical mechanics to combinatorics. That is precisely the object of the present study.

Apparently, the first step in this program is to generalize the Izergin--Korepin formula to the 8VSOS model.  In a recent paper \cite{r} we found such a generalization.  In the trigonometric limit (which is intermediate between the six-vertex model and the general elliptic 8VSOS model), we 
used it
 to obtain a closed formula for a special case of the three-colour partition function, see \eqref{ttf}. In the present paper, we consider the general case. Although there seems to be no  very simple  formula for the general three-colour partition function, we can express it in terms of certain special polynomials, which have remarkable properties and deserve further study. 
To obtain these results has not been straight-forward; in particular, we have not been able to work directly with the explicit formulas from \cite{r}. 
Rather, we  combine a simple consequence of those formulas
with several further ideas. 

The plan of the paper is as follows. In \S \ref{tccs}, we  describe the three-colour model with domain wall boundary conditions, and its relation to ice graphs and alternating sign matrices. 
We refer to the states of the model as \emph{three-coloured chessboards}.
In \S \ref{srs}, we state our main results in elementary form. \S \ref{ps} contains preliminaries on theta functions and the 8VSOS model. In \S \ref{qrs}, we prove our first main result, Theorem \ref{st}, which expresses the three-colour partition function $Z_n^{3C}$ in terms of special polynomials $q_n$ and $r_n$. This is a rather easy consequence of results in \cite{r}.

We then turn towards an alternative way of expressing $Z_n^{3C}$. In \S \ref{phs}, we introduce a function $\Phi_n$, which provides a one-parameter extension of $Z_n^{3C}$. For the six-vertex model, a similar function appears in  Zeilberger's proof of the \emph{refined} alternating sign matrix conjecture \cite{z2}.
In \S \ref{sfs}, $\Phi_n$ is generalized to a multivariable  theta function $\Psi_n$. These functions play a similar role as 
Schur polynomials do in Stroganov's  proof of the alternating sign matrix theorem \cite{s} (see also \cite{o}). However, while the Schur polynomials are instances of
the six-vertex partition function (with the crossing parameter a cubic root of unity), the function
 $\Psi_n$ is  \emph{not}  directly related to the 8VSOS partition function.
The function $\Psi_n$ has two important properties, the first being a determinant formula reminiscent of the Izergin--Korepin determinant. 
The second property 
  is a symmetry with respect to inversion of 
 each variable, meaning that it naturally lives on the Riemann sphere rather than the torus. As a consequence,  $\Psi_n$ can be identified with a  symmetric polynomial $S_n$. In \S \ref{pps}, we specialize the variables in $S_n$, obtaining certain two-variable polynomials $P_n$  and one-variable polynomials $p_n$. 
Using minor relations for the determinant defining $\Psi_n$, 
  we  obtain recursions for these polynomials, 
which can  be used to derive many further properties. 
In \S \ref{rtcs}, we return to the three-colour model, expressing
 $Z_n^{3C}$  in terms of the polynomials $p_n$. This result has  combinatorial consequences. For instance, 
for three-coloured chessboards of fixed size, 
we can compute the maximal and minimal possible number of squares of each 
colour,  see Corollary \ref{ecc}. 

Finally, in \S \ref{tds} we study the thermodynamic limit  $n\rightarrow\infty$. Using non-rigorous arguments, we are led to an explicit formula for the free energy of the three-colour model with domain wall boundary conditions (Conjecture \ref{tc}). From the viewpoint of physics, this is the main result of the paper.
 To prove Conjecture \ref{tc} rigorously is an interesting problem, which 
we expect to be  rather difficult.  For the six-vertex model with domain wall boundary conditions, a rigorous analysis has been done only  recently \cite{bf,bl1,bl2,bl3}.

{\bf Note added in proof:}
Immediately after seeing an earlier version of the present paper, 
Vladimir Bazhanov and Vladimir Mangazeev sent me an interesting conjectured recursion for the polynomials $p_n$. 
They  have also obtained a similar conjecture for $P_n$. These recursions give a much faster way of computing the three-colour partition function than those obtained in the present paper. Both conjectures can be found in \cite{bm4}, where the authors stress the  resemblance to polynomials occurring  in their analysis of the eight-vertex model  \cite{bm1,bm2,bm3}. It would be  interesting to investigate if that relation can be made precise; for instance, that may give a link between the present work and the Painlev\'e VI equation.

{\bf Acknowledgements:} 
This work was partly carried out while participating in the programme 
Discrete Integrable Systems at
 the Isaac Newton Institute for Mathematical Sciences, and I  thank the institute and the programme organizers for their support. 
 I would also like to thank Vladimir Bazhanov, Vladimir Mangazeev, 
 Jacques Perk and, in particular, Don Zagier for their interest and for valuable comments and suggestions. Finally, I thank the anonymous referees for useful comments.

\section{Three-coloured chessboards}
\label{tccs}

We will refer to a state of the three-colour model  with domain wall boundary conditions as a \emph{three-coloured chessboard}. 
Fixing $n$, consider a chessboard of size $(n+1)\times(n+1)$. 
The squares will be labelled with  three colours, which we identify with the three residue classes $0,1,2\ \operatorname{mod}\ 3$. We impose the following two rules. First, vertically or horizontally adjacent squares have distinct colours. Second, the north-west and south-east squares are labelled $0$ and, as one proceeds away from these squares along the boundary, the colours increase  with respect to the cyclic order
\begin{equation}\label{co}0<1<2<0.\end{equation}
In particular, the north-east and south-west square are labelled $n$ $\operatorname{mod}\ 3$. As an example, when $n=3$ there are seven three-coloured chessboards; these are displayed in Figure \ref{cbf}.

\begin{figure}[htb]
\caption{The seven three-coloured chessboards of size $n=3$, where $0$ is pictured as black, $1$ as red and $2$ as yellow.}
\begin{center}
\vspace*{1ex}
\includegraphics[height=2cm]{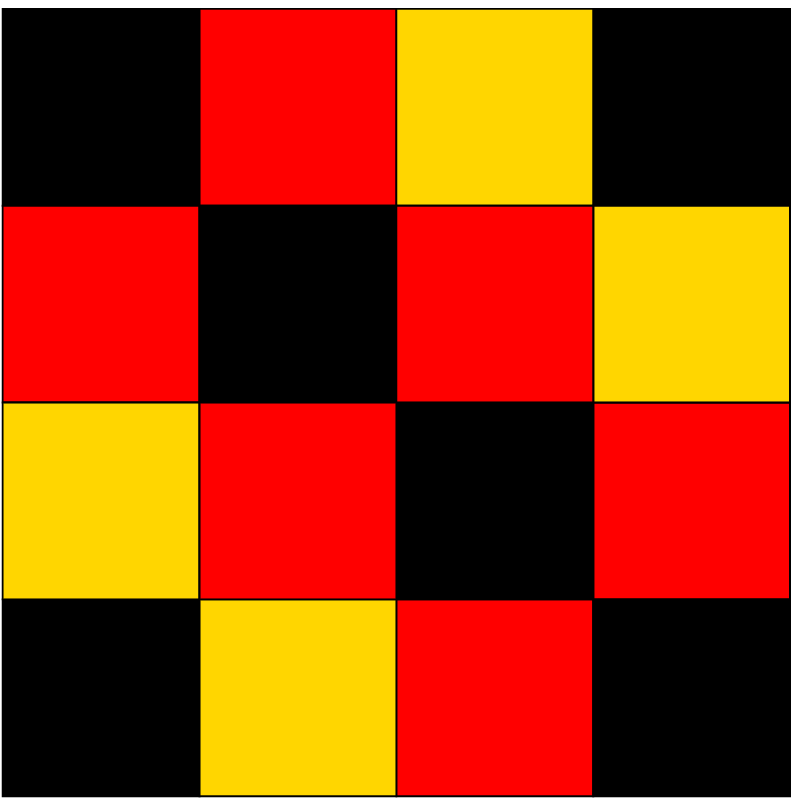}
\hspace{.5cm}
\includegraphics[height=2cm]{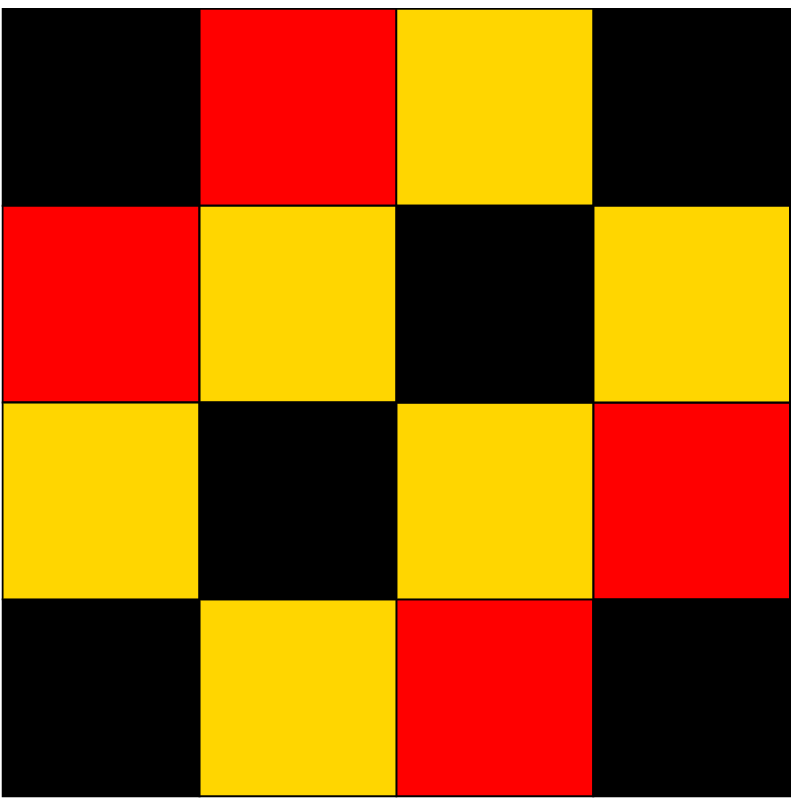}
\hspace{.5cm}
\includegraphics[height=2cm]{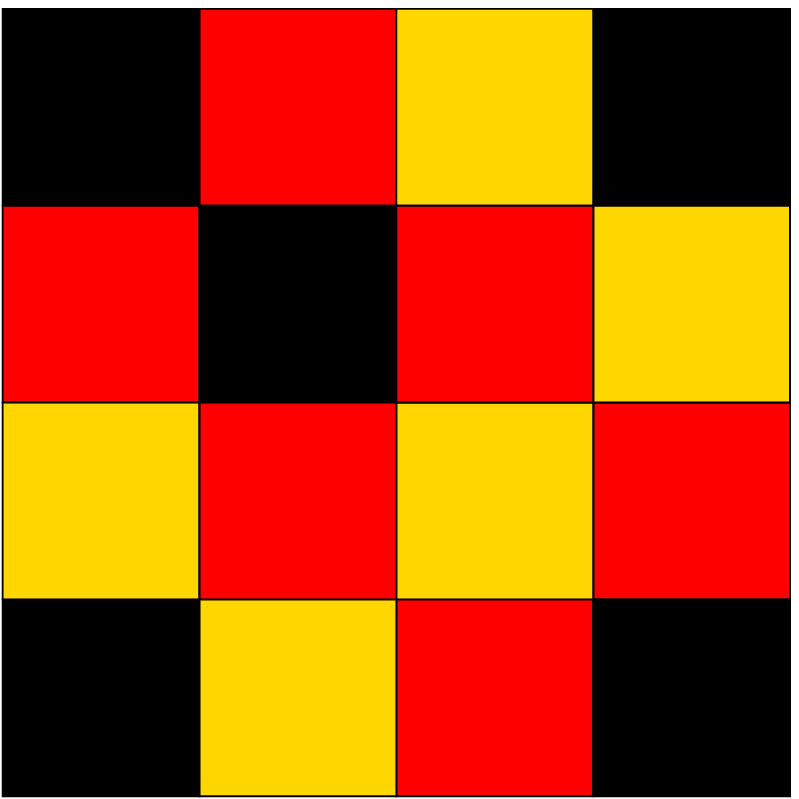}
\hspace{.5cm}
\includegraphics[height=2cm]{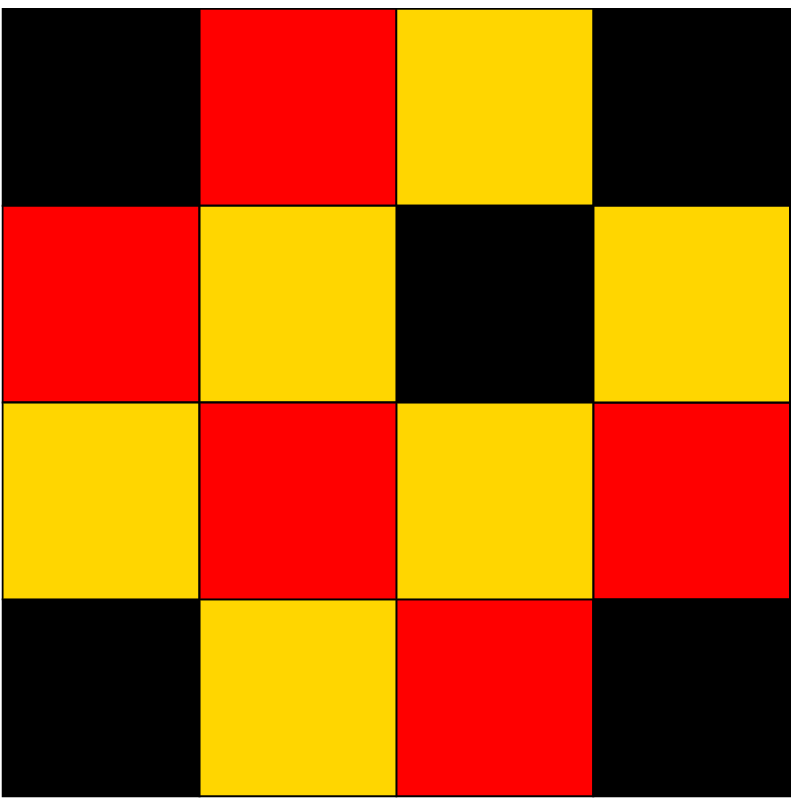}\\[.7cm]
\end{center}
\begin{center}
\includegraphics[height=2cm]{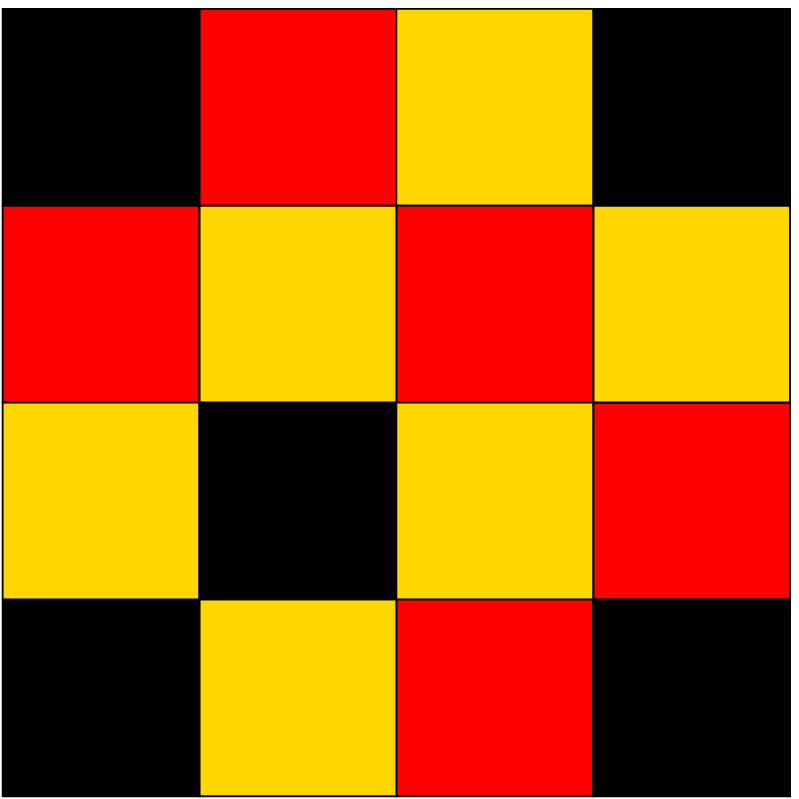}
\hspace{.5cm}
\includegraphics[height=2cm]{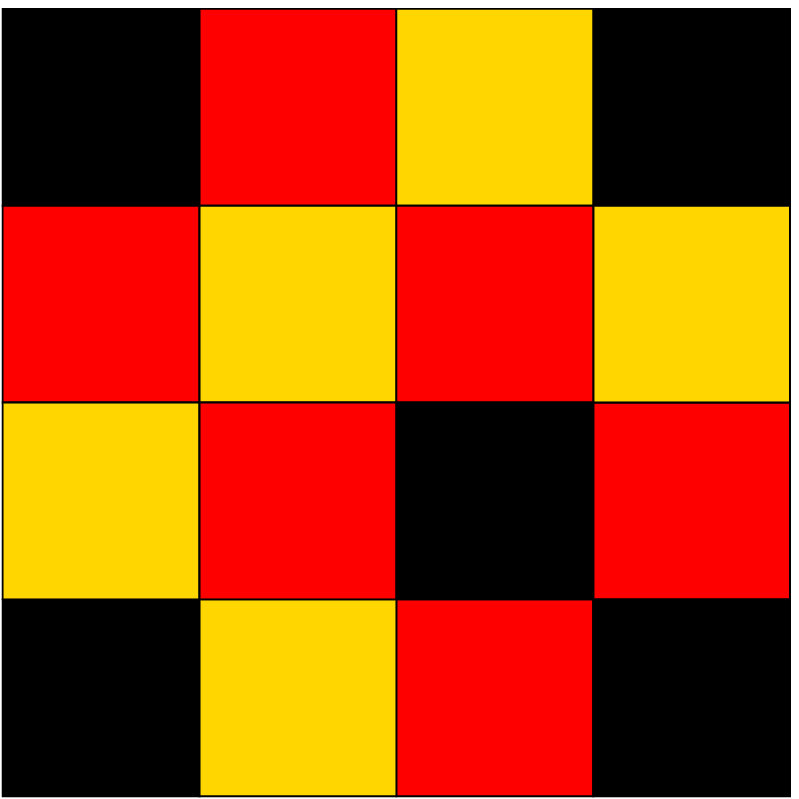}
\hspace{.5cm}
\includegraphics[height=2cm]{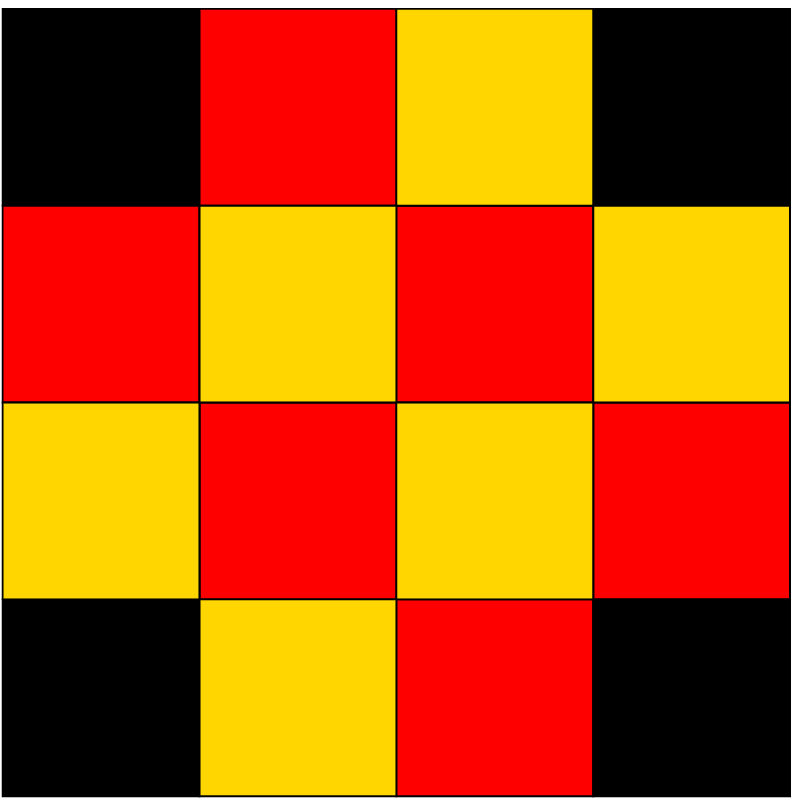}
\end{center}
\label{cbf}
\end{figure}

We will briefly explain the bijections to alternating sign matrices and ice graphs mentioned in the introduction. 
 Let $\left(\begin{smallmatrix}a&b\\ c&d\end{smallmatrix}\right)$
be a $2\times 2$-block of adjacent squares from a three-coloured chessboard, and choose representatives of the residue classes so that adjacent labels differ by exactly $1$. For instance, if the original block is 
$\left(\begin{smallmatrix}0&2\\ 2&1\end{smallmatrix}\right)$,
we may choose the representatives  
$\left(\begin{smallmatrix}3&2\\ 2&1\end{smallmatrix}\right)$. Having made such a choice, we contract each  block to the  number $(b+c-a-d)/2$, obtaining an $n\times n$-matrix with entries $-1,0,1$. 
For instance, from the last chessboard
 in \eqref{cbe}
we obtain
$$\left(\begin{matrix}0&1&0\\1&-1&1\\0&1&0\end{matrix}\right).$$
This  gives a bijection from three-coloured chessboards to  alternating sign matrices.

To obtain the bijection to ice graphs, we draw an arrow between any two adjacent squares in such a way that the larger label, with respect to the order \eqref{co}, is to the right. For instance, the last chessboard in \eqref{cbe} corresponds to the arrow configuration
\begin{equation}\label{ig}\begin{pspicture}
(0,0)(2.8,2.8)
\psset{xunit=.7,yunit=.7}
\psset{arrowsize=5pt}
\rput(0,3){\iceR}
\rput(0,2){\iceR}
\rput(0,1){\iceR}
\rput(1,3){\iceR}
\rput(1,1){\iceR}
\rput(2,2){\iceR}
\rput(2,2){\iceL}
\rput(3,3){\iceL}
\rput(3,1){\iceL}
\rput(4,3){\iceL}
\rput(4,2){\iceL}
\rput(4,1){\iceL}
\rput(1,3){\iceU}
\rput(2,3){\iceU}
\rput(3,3){\iceU}
\rput(1,2){\iceU}
\rput(3,2){\iceU}
\rput(2,1){\iceU}
\rput(2,3){\iceD}
\rput(1,2){\iceD}
\rput(3,2){\iceD}
\rput(1,1){\iceD}
\rput(2,1){\iceD}
\rput(3,1){\iceD}
\end{pspicture}.
\end{equation}
The result is a directed graph, where each internal vertex has two incoming and two outgoing edges. Considering vertices as oxygen atoms and incoming edges as hydrogen bonds, this can be viewed as a model for a two-dimensional sheet of ice.

We are interested in the generating function
\begin{equation}\label{zd}\begin{split}Z_n^{\text{3C}}(t_0,t_1,t_2)&=\sum_{\text{chessboards}}\,\prod_{\text{squares}}t_{\text{colour}}\\
&=\sum_{k_0+k_1+k_2=(n+1)^2}N(k_0,k_1,k_2)t_0^{k_0}t_1^{k_1}t_2^{k_2},\end{split} \end{equation}
where
$N(k_0,k_1,k_2)$ denotes the number of three-coloured chessboards with exactly $k_i$ squares of colour $i$. In physics terminology, $Z_n^{\text{3C}}$ is the \emph{partition function} of the three-colour model with domain wall boundary conditions. 

It is difficult to study  $Z_n^{\text{3C}}$  by direct methods. Indeed, to compute
$Z_n^{\text{3C}}(1,1,1)=A_n$
was an unsolved problem for more than a decade.
In \cite{r}, we generalized that enumeration using the 
trigonometric 8VSOS model. Namely, we found 
 a  closed expression for $Z_n^{\text{3C}}(t_0,t_1,t_2)$ when
\begin{equation}\label{tt}\frac{(t_0t_1+t_0t_2+t_1t_2)^3}{(t_0t_1t_2)^2}=27. 
\end {equation}
This surface can be parametrized by 
$t_i(\la,\mu)=\mu/(1-\la\om^i)^3$,
where, as throughout the paper, 
$$\om=e^{2\pi i/3}. $$
 By homogeneity, we may take $\mu=1$. 
Then \cite[Cor.\ 8.4]{r},
\begin{multline}\label{ttf}Z_n^{\text{3C}}\left(\frac 1{(1-\la)^3},\frac 1{(1-\la\om)^3},\frac 1{(1-\la\om^2)^3}\right)\\
=\frac{(1-\la\om^2)^2(1-\la\om^{n+1})^2}{(1-\la^3)^{n^2+2n+3}}\big(A_n(1+\om^{n}\la^2 )+(-1)^{n} C_n\om^{2n}\la\big),\end{multline}
where
$A_n$ is as in \eqref{an} and
\begin{equation}\label{cn}C_n=\frac{2\cdot 5\cdots (3n-1)}{1\cdot 4\cdots(3n-2)}\,A_n
\end{equation}
is  the number of cyclically symmetric plane partitions in a cube of size $n$ \cite{a}. 
Although \eqref{tt} is a strong restriction, it is sufficient for computing the moments
$$\sum_{k_0+k_1+k_2=(n+1)^2}N(k_0,k_1,k_2)k_i,\qquad i=0,1,2, $$
see \cite[Cor.\ 8.5]{r}.

\section{Statement of results}
\label{srs}

\subsection{Polynomials $q_n$ and $r_n$.}
In this Section, we state our main results. 
We begin with the following fact.
As we will see in \S \ref{qrs}, it is  a rather  straight-forward consequence of results in \cite{r}.

\begin{theorem}\label{st}
Let \begin{equation}\label{om}T=T(t_0,t_1,t_2)=\frac{(t_0t_1+t_0t_2+t_1t_2)^3}{(t_0t_1t_2)^2}. 
\end{equation}
Then,
there exist polynomials $q_n$ and $r_n$ such that, for 
 $n\equiv 0\ \operatorname{mod}\ 3$, 
 $Z_n^{3C}(t_0,t_1,t_2)$ equals
$$(-1)^{n+1}(t_0t_1t_2)^{\frac{n(n+2)}3}\left(\frac{(t_0t_1+t_0t_2+t_1t_2)^2}{t_0t_1t_2}\,q_{n}(T)-2^{\chi(n \text{\emph{ odd}})}t_0r_{n}(T)\right), $$
 while for $n\equiv 1\ \operatorname{mod}\ 3$ it equals 
$$(-1)^{n+1}(t_0t_1t_2)^{\frac{n(n+2)}3}\left(\frac{t_0t_1}{t_2}\,q_{n}(T)-2^{\chi(n \text{\emph{ odd}})}\frac{t_0t_1+t_0t_2+t_1t_2}{t_2}\,r_{n}(T)\right) $$
 and for $n\equiv 2\ \operatorname{mod}\ 3$ it equals
 $$(-1)^{n+1}(t_0t_1t_2)^{\frac{(n+1)^2}3}\left(q_{n}(T)-2^{\chi(n \text{\emph{ odd}})}\frac{t_0t_1+t_0t_2+t_1t_2}{t_0t_2}\,r_{n}(T)\right). $$
Here,  $\chi(\text{\emph{true}})=1$,  $\chi(\text{\emph{false}})=0$.
\end{theorem}

The first few instances of the polynomials $q_n$ and $r_n$ are given in Table \ref{pqt}.

Note that, in each case, $Z_n^{3C}$ is symmetric in the two variables
$t_{-n\pm 1\, \operatorname{mod}\, 3}$. This  is explained by the fact that reflection in the vertical (say) axis, followed by  interchanging the colours $-n\pm 1\ \operatorname{mod}\ 3$, defines an involution on three-coloured chessboards.  Theorem \ref{st} shows that $Z_n^{3C}$ is  very nearly symmetric in all three variables, being a linear combination of two symmetric polynomials,  where the coefficients are  polynomials in $t_{-n\, \operatorname{mod}\, 3}$ of  low order. 
Moreover, the symmetric polynomials depend only on the second and third 
elementary symmetric polynomial, being independent of $t_0+t_1+t_2$.

\begin{table}[htb]
\caption{The polynomials $q_n$ and $r_n$.}
\begin{tabular}{lll}
$n$ & $q_n(x)$ & $r_n(x)$\\
\hline
$0$ & $0$ & $1$ \\
$1$ & $1$ & $0$ \\
$2$ & $1$ & $1$ \\
$3$ & $1$ & $1$ \\
$4$ & $x+3$ & $x-3$\\
$5$ & $x^2-4x+6$ & $x+6$\\
$6$ & $x^2-2x+40$ & $x^3-8x^2+20$ \\   
$7$ & $x^4 -10x^3+ 15x^2+ 100x+ 50$ & $x^3+75x-50$\\
\hline
\end{tabular}
\label{pqt}
\end{table}

The following result seems  much deeper than Theorem \ref{st}. We  need considerable preparation for its proof, which is given in
\S \ref{qrss}.

\begin{theorem}\label{qrdp}
The polynomials $q_n$ and $r_n$ are monic. Moreover, their degrees  are  given by
\begin{align*}
\deg(q_n)+1=\deg(r_n)&=\frac{n^2}{12},& n&\equiv 0\ \operatorname{mod}\ 6,\\
\deg(q_n)=\deg(r_n)+1&=\frac{n^2-1}{12},& n&\equiv \pm 1\ \operatorname{mod}\ 6, \\
\deg(q_n)=\deg(r_n)&=\frac{n^2-4}{12},& n&\equiv \pm 2\ \operatorname{mod}\ 6,\\
\deg(q_n)=\deg(r_n)&=\frac{n^2-9}{12},& n&\equiv 3\ \operatorname{mod}\ 6.
\end{align*}
\end{theorem}

As an application, we can determine the 
maximal and minimal number of squares of each colour.
These bounds restrict the 
counting function $N$ introduced in \eqref{zd} to an 
equilateral triangle. We can also 
explicitly evaluate  the restriction of $N$ to the boundary. To formulate  the result, 
we introduce some notation. Fixing $n$, let
$$\bar N(x,y,z)
=\begin{cases}
N(x,y,z), & n\equiv 0\ \operatorname{mod}\ 3,\\
N(y,z,x),  & n\equiv 1\ \operatorname{mod}\ 3,\\
N(z,x,y),  & n\equiv 2\ \operatorname{mod}\ 3,
\end{cases} $$
so that
$\bar N(x,y,z)=\bar N(x,z,y)$.
Moreover, let
$$M=\left[\frac{5n^2+8n+11}{12}\right]=\begin{cases}
(5n^2+8n)/12, & n\equiv 0,\,2\ \operatorname{mod}\ 6,\\
(5n^2+8n+11)/12, & n\equiv 1\ \operatorname{mod}\ 6,\\
(5n^2+8n+3)/12, & n\equiv 3,\,5\ \operatorname{mod}\ 6,\\
(5n^2+8n+8)/12, & n\equiv 4\ \operatorname{mod}\ 6,
\end{cases} $$
$$m=\begin{cases}
(n^2+4n)/6, & n\equiv 0,\,2\ \operatorname{mod}\ 6,\\
(n^2+4n+7)/6, & n\equiv 1\ \operatorname{mod}\ 6,\\
(n^2+4n+3)/6, & n\equiv 3,\,5\ \operatorname{mod}\ 6,\\
(n^2+4n+4)/6, & n\equiv 4\ \operatorname{mod}\ 6,
\end{cases} $$
$$\varepsilon=\begin{cases}
1, & n\equiv 0,\,2\ \operatorname{mod}\ 6,\\
-2, & n\equiv 1\ \operatorname{mod}\ 6,\\
0, & n\equiv 3,\,5\ \operatorname{mod}\ 6,\\
-1, & n\equiv 4\ \operatorname{mod}\ 6,
\end{cases} $$
\begin{equation}\label{de}\delta=\left[\frac{n^2}4\right]=\begin{cases}
(n^2-1)/4, & n\text{ odd},\\
n^2/4, & n\text{ even}.
\end{cases} \end{equation}
Note that $2M+m+\varepsilon=(n+1)^2$.
Let
$$\Delta=\{(x,y,z)\in \mathbb Z^3;\, x+y+z=(n+1)^2,\ x\leq M+\varepsilon,\
y,\,z\leq M \}. $$
Then, $\Delta$ is an equilateral triangle, with $\delta$ lattice points on each side. We denote its corners by
$$P=(m+\varepsilon,M,M),\quad Q=(M+\varepsilon,m,M),\quad R=(M+\varepsilon, M,m). $$

\begin{corollary}\label{ecc}
The convex hull of the support of $\bar N$ is equal to $\Delta$ when $n$ is odd and $\Delta\setminus\{P\}$ when $n$ is even. In particular, the maximal number of squares of each of the colours $-n\pm 1\ \operatorname{mod}\ 3$ is  equal to $M$, and the minimal number of such squares is $m$.  The maximal number of squares of colour $-n\ \operatorname{mod}\ 3$ is  $M+\varepsilon$, and the minimal number  is $m+\varepsilon$ if $n$ is odd and  $m+\varepsilon+1$ if $n$ is even. Moreover, the restriction of $\bar N$ to $\partial \Delta$ is given by
 $$\bar N\left(\frac{kP+(\delta-k)Q}{\delta}\right)=\bar N\left(\frac{kP+(\delta-k)R}{\delta}\right)=\begin{cases}\displaystyle\binom{\delta-1}k, & n \text{ \emph{even}},\\[5mm]
\displaystyle\binom{\delta}k, & n \text{ \emph{odd}},
\end{cases} $$
$$\bar N\left(\frac{kQ+(\delta-k)R}{\delta}\right)=\begin{cases}\displaystyle\binom{\delta}k, & n \text{ \emph{even}},\\[5mm]
\displaystyle\binom{\delta-2}k+\binom{\delta-2}{k-2}, & n \text{ \emph{odd}},
\end{cases} $$
where $0\leq k\leq \delta$.
\end{corollary}

It is  straight-forward to derive Corollary \ref{ecc} from Theorem \ref{qrdp}; some details are given in \S \ref{qrss}.

As we explain at the end of \S \ref{qrs}, it follows easily from Theorem \ref{st} that the polynomials $q_n$, $r_{2n}$ and $2r_{2n+1}$ have integer coefficients. However, the following fact is not obvious, see \S \ref{ics}.

\begin{proposition}\label{rip} The polynomial $r_{2n+1}$ has integer coefficients.
\end{proposition}

This result has a simple combinatorial meaning. 
If  $n\equiv 3$ or $5$ $\operatorname{mod}\ 6$, Theorem~\ref{st} expresses $Z_n^{3C}$ as a sum of a symmetric polynomial and a polynomial with even coefficients. Thus, the function 
$$(k_0,k_1,k_2)\mapsto N(k_0,k_1,k_2)\ \operatorname{mod}\ 2$$
 is symmetric. Similarly, 
when $n\equiv 1\ \operatorname{mod}\ 6$, 
$$(k_0,k_1,k_2)\mapsto N(k_0,k_1,k_2-2)\ \operatorname{mod}\ 2$$ 
is symmetric.
In the notation of Corollary \ref{ecc}, these facts can be stated as follows.

\begin{corollary}\label{psc} When $n$ is odd,  $\bar N\ \operatorname{mod}\ 2$ is invariant under the action of $S_3$ as the symmetry group of $\Delta$.
\end{corollary}

To illustrate  Corollaries  \ref{ecc} and \ref{psc}, we give
two examples. When $n=4$, the non-zero values of  $N$ are
$$\big(N(6+i,6+j,13-i-j)\big)_{i,j=0}^{4}=
\left(\begin{matrix}&&&&1\\
&&&4&3\\
&&6&6&3\\
&4&6&&1\\
1&3&3&1&
\end{matrix}\right). $$
Since $n$ is even, the lower right corner is missing from the support of $N$, and all boundary entries are  binomial coefficients. When $n=5$, the non-zero values are 
$$\big(N(8+i,20-i-j,8+j)\big)_{i,j=0}^{6}
=\left(\begin{matrix}&&&&&&1\\ 
&&&&&4&6\\
&&&&7&18&15\\
&&&8&12&36&20\\
&&7&12&36&40&15\\
&4&18&36&40&24&6\\
1&6&15&20&15&6&1
\end{matrix}\right).$$
In this case, the diagonal entries  are numbers of the form $\binom 4k+\binom 4{k-2}$. Note  the symmetric distribution of the odd entries, which 
is peculiar to the case of odd $n$.

\subsection{Polynomials $p_n$ and $P_n$}
The polynomials $q_n$ and $r_n$ are closely related to a third class of polynomials, which we  denote $p_n$. 
In the following result, which is  proved in \S \ref{qrss}, we use the notation
\begin{equation}\label{tn}\tilde p(x)=x^{\deg p}p(1/x);\end{equation}
that is, $\tilde p$ denotes the polynomial obtained from $p$ by reversing the coefficients. 

\begin{theorem}\label{pqrt}
There exist  polynomials  $p_n$ of degree $n(n+1)/2$ such that, for $n$  odd,
$$p_{n-1}(\zeta)-\zeta^{\frac{n+1}2}\tilde p_{n-1}(\zeta)=(1-\zeta)(\zeta^2+4\zeta+1)^{\frac{n^2-1}{4}}\tilde q_{n}\left(\frac{\zeta(\zeta+1)^4}{2(\zeta^2+4\zeta+1)^3}\right),$$
$$p_{n-1}(\zeta)-\zeta^{\frac{n-1}2}\tilde p_{n-1}(\zeta)=(1-\zeta)(1+\zeta)^3(\zeta^2+4\zeta+1)^{\frac{n^2-9}{4}}\tilde r_{n}\left(\frac{\zeta(\zeta+1)^4}{2(\zeta^2+4\zeta+1)^3}\right),$$
whereas  for $n$ even,
$$p_{n-1}(\zeta)-\zeta^{\frac{n+2}2}\tilde p_{n-1}(\zeta)=(1-\zeta)(\zeta^2+4\zeta+1)^{\frac{n^2}{4}}\tilde r_{n}\left(\frac{\zeta(\zeta+1)^4}{2(\zeta^2+4\zeta+1)^3}\right),$$
$$p_{n-1}(\zeta)-\zeta^{\frac{n}2}\tilde p_{n-1}(\zeta)=(1-\zeta^2)(\zeta^2+4\zeta+1)^{\frac{n^2-4}{4}}\tilde q_{n}\left(\frac{\zeta(\zeta+1)^4}{2(\zeta^2+4\zeta+1)^3}\right).$$
\end{theorem}

The shift in $n$ is introduced for  convenience. 
See Table \ref{pt} for the first few instances of the polynomials $p_n$.

\begin{table}[htb]
\caption{The polynomials $p_n$.}
\begin{tabular}{ll}
$n$ & $p_n(\zeta)$ \\
\hline
$-1$ & $1$  \\
$0$ & $1$  \\
$1$ & $3\zeta+1$  \\
$2$ & $5\zeta^3+15\zeta^2+7\zeta+1$  \\
$3$ & $\frac 12(35\zeta^6+231\zeta^5+504\zeta^4+398\zeta^3+147\zeta^2+27\zeta+2)$\\[1mm]
$4$ & $\frac 12(63\zeta^{10}+798\zeta^9+4122\zeta^8+11052\zeta^7+16310\zeta^6+13464\zeta^5+6636\zeta^4$ \\
& $\quad+2036\zeta^3+387\zeta^2+42\zeta+2)$\\
$5$ & $\frac 12(231\zeta^{15}+4554\zeta^{14}+39468\zeta^{13}+196922\zeta^{12}+622677\zeta^{11}+1298446\zeta^{10}$\\
&$\quad +1816006\zeta^9+1726302\zeta^8+1140593\zeta^7+535478\zeta^6+181104\zeta^5$\\
&$\quad +44134\zeta^4+7603\zeta^3+882\zeta^2+62\zeta+2)$ \\   
$6$ & $\frac 18(1716\zeta^{21}+50193\zeta^{20}+673530\zeta^{19}+5484050\zeta^{18}+30199260\zeta^{17}$\\
&$\quad+118703208\zeta^{16}+342834244\zeta^{15}+738954900\zeta^{14}+1198556100\zeta^{13}$\\
&$\quad+1470762970\zeta^{12}+1373623128\zeta^{11}+984509064\zeta^{10}+546520100\zeta^{9}$\\
&$\quad+236837400\zeta^{8}+80476380\zeta^{7}+21422188\zeta^6+4430904\zeta^5+699405\zeta^4$\\
&$\quad+81550\zeta^3+6630\zeta^2+336\zeta+8)$\\
\hline 
\end{tabular}\label{pt}
\end{table}

To indicate the meaning of Theorem \ref {pqrt}, we solve for $p_n$, obtaining for $n$ even
\begin{multline}\label{pqre}p_n(\zeta)=\Bigg((\zeta^2+4\zeta+1)^{\frac{n(n+2)}{4}}\tilde q_{n+1}\left(\frac{\zeta(\zeta+1)^4}{2(\zeta^2+4\zeta+1)^3}\right)\\
-\zeta(1+\zeta)^3(\zeta^2+4\zeta+1)^{\frac{(n-2)(n+4)}{4}}\tilde r_{n+1}\left(\frac{\zeta(\zeta+1)^4}{2(\zeta^2+4\zeta+1)^3}\right)\Bigg) \end{multline}
and for $n$ odd
\begin{multline*}p_n(\zeta)=\Bigg((\zeta^2+4\zeta+1)^{\frac{(n+1)^2}{4}}\tilde r_{n+1}\left(\frac{\zeta(\zeta+1)^4}{2(\zeta^2+4\zeta+1)^3}\right)\\
-\zeta(1+\zeta)(\zeta^2+4\zeta+1)^{\frac{(n-1)(n+3)}{4}}\tilde q_{n+1}\left(\frac{\zeta(\zeta+1)^4}{2(\zeta^2+4\zeta+1)^3}\right)\Bigg). \end{multline*}
A priori, the right-hand sides are polynomials of degree 
$n(n+2)/2$, $(n+1)^2/2$, respectively. The degree bound  $n(n+1)/2$
 imposes  relations between high coefficients of the polynomials $q_n$ and $r_n$.
 In the notation of Corollary \ref{ecc}, this implies relations between values of
 $\bar N$ close to the boundary $\partial \Delta$; the details will not be worked out here.

An important property of
the polynomials  $p_n$  is that they appear as solutions to certain linear  equations. 

\begin{theorem}\label{prt}
Consider the 
polynomial equation
\begin{equation}\label{le}AX-BY=C,\end{equation}
where
$$A=(\zeta+1)^2p_{n-1}(\zeta),\qquad B=\zeta^{n+1} \tilde p_n(\zeta),$$
$$C=(1+2\zeta)^{1+\chi(n \text{\emph{ even}})}\left(1+\frac \zeta 2\right)^{1+\chi(n \text{\emph{ odd}})}p_n(\zeta)^2.$$
Then, this equation has a solution $(X,Y)$ such that $X=p_{n+1}$ and $\deg Y= \deg A$.
\end{theorem}

Theorem \ref{prt} follows from Proposition \ref{pyr} and Corollary \ref{ydc}.
Although we have not been able to prove it, we believe that
the polynomials $A$ and $B$ are always relatively prime. 
In fact, we believe that $p_n$ is irreducible over $\mathbb Q$
for $n\geq 2$. 
Assuming that this is the case, \eqref{le} has a unique solution $(X_0,Y_0)$ with $\deg Y_0<\deg A$. The solutions with
 $\deg Y\leq \deg A$  then have the form $(X,Y)=(X_0+\la B,Y_0+\la A)$, $\lambda\in\mathbb C$. We can specify the solution $X=p_{n+1}$ in this space by 
declaring that the leading term  is
$$2^{-\left[\frac {n+2}2\right]}\binom{2n+2}{n+1}\zeta^{n(n+1)/2},$$
see Proposition \ref{psv}.
Then, Theorem \ref{prt} gives a recursive procedure for constructing the polynomials $p_n$. This gives a comparatively fast
method for computing the partition function $Z_n^{3C}$. 
(As was mentioned at the end of the Introduction, a much faster method has been suggested by Bazhanov and Mangazeev \cite{bm4}.)

The polynomials $p_n$ have coefficients in $\mathbb Z/2^{\mathbb Z}$; see Corollary \ref{pic} for a more precise statement. Moreover, the following facts seem to be true.

\begin{conjecture}\label{pcc}
The polynomials  $p_n$ have positive coefficients.
\end{conjecture}

\begin{conjecture}\label{umc}
The polynomials  $p_n$ are unimodal in the sense that, if 
$$p_n(\zeta)=\sum_{k=0}^{n(n+1)/2}a_k\zeta^k, $$
then, for some $N$,
$$a_0<a_1<\dots<a_{N-1}<a_N>a_{N+1}>\dots>a_{n(n+1)/2}. $$
\end{conjecture}

Using Theorem \ref{prt}, 
we have verified Conjectures \ref{pcc} and \ref{umc} up to $n=16$. For these values, the maximal coefficient $a_N$ occurs at
$$N=\begin{cases}
n(n+2)/4, & n \text{ even},\quad n\leq 16\\
(n+1)^2/4, & n \text{ odd},\quad  n\leq 7\\
(n-1)(n+3)/4, & n \text{ odd},\quad  9\leq n\leq 15.
\end{cases} $$

Numerical experiments suggest that
the zeroes of $p_n$  form quite remarkable patterns, see Figure \ref{zf}.

\begin{figure}[htb]\caption{The $105$ complex zeroes of $p_{14}$.}\label{zf}
\begin{center}
\vspace*{2ex}
\includegraphics[width=7cm]{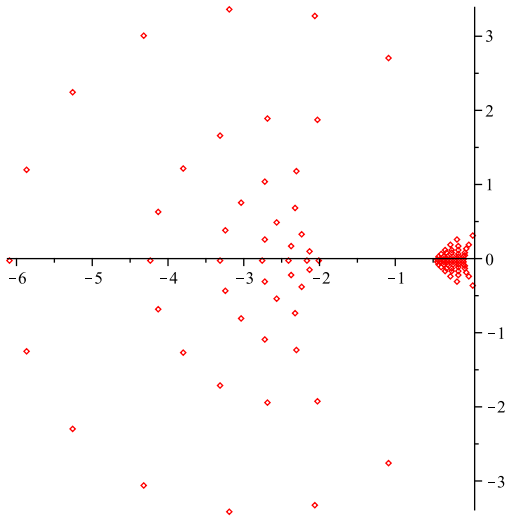}
\end{center}
\end{figure}

We have some partial results on the real zeroes, see  \S \ref{zs} for proofs.

 \begin{conjecture}\label{zc}
The polynomial $p_n$ has exactly  $[(n+1)/2]$ real zeroes.
\end{conjecture}

\begin{conjecture}\label{nvc}
The polynomial $p_n$ does not vanish in the interval $-2<\zeta<-1/2$.
\end{conjecture}

\begin{proposition}\label{zp}
Assume \emph{Conjecture \ref{zc}}. Then,  all real zeroes of  $p_n(\zeta)$ are simple, and contained in the interval $-1/2<\zeta<0$ if $n$ is odd and in  $\zeta<-2$ if $n$ is even. Moreover, the real zeroes of $ p_{2n}$ interlace the real zeroes of $\tilde p_{2n+1}$, which in turn alternate left of the real  zeroes of $ p_{2n+2}$. In particular, \emph{Conjecture  \ref{zc}} implies 
\emph{Conjecture \ref{nvc}}.
\end{proposition}

The polynomials $p_n$ can be embedded in a more general family of polynomials  $P_n(x,\zeta)$, which are of degree $n$ in
 $x$  and degree $[n(n+2)/2]$ in $\zeta$. To be precise,
$$P_n(1+2\zeta,\zeta)=(1+2\zeta)^{\left[\frac n2\right]}p_n(\zeta).$$
We give the first few instances of $P_n$ in  Table \ref{ppt}.

\begin{table}[htb]
\caption{The polynomials $P_n$.}
\begin{tabular}{ll}
$n$ & $P_n(x,\zeta)$ \\
\hline
$0$ & $1$\\
$1$ & $x+\zeta$ \\
$2$ & $\frac12\Big((\zeta+2)(3\zeta+1)x^2+\zeta(\zeta+3)(3\zeta+1)x+\zeta^2(\zeta+3)(2\zeta+1)\Big)$\\
$3$ & $\frac12\Big((\zeta+2)(5\zeta^3+15\zeta^2+7\zeta+1)x^3+\zeta(\zeta+4)(5\zeta^3+15\zeta^2+7\zeta+1)x^2$\\
&$\quad+\zeta^2(4\zeta+1)(\zeta^3+7\zeta^2+15\zeta+5)x
+\zeta^3(2\zeta+1)(\zeta^3+7\zeta^2+15\zeta+5)\Big)$  \\
$4$ & $\frac 18\Big((\zeta+2)^2(35\zeta^6+231\zeta^5+504\zeta^4+398\zeta^3+147\zeta^2+27\zeta+2)x^4$\\
&$\quad+\zeta(\zeta+5)(\zeta+2)(35\zeta^6+231\zeta^5+504\zeta^4+398\zeta^3+147\zeta^2+27\zeta+2)x^3$ \\
&$\quad + 3\zeta^2(10\zeta^8+139\zeta^7+790\zeta^6+2245\zeta^5$\\
&\hfill$+3232\zeta^4+2245\zeta^3+790\zeta^2+139\zeta+10)x^2 $\\
&$\quad+\zeta^3(2\zeta+1)(5\zeta+1)(2\zeta^6+27\zeta^5+147\zeta^4+398\zeta^3+504\zeta^2+231\zeta+35)x$\\
& $\quad +\zeta^4(2\zeta+1)^2(2\zeta^6+27\zeta^5+147\zeta^4+398\zeta^3+504\zeta^2+231\zeta+35)\Big)$\\
\hline
\end{tabular}\label{ppt}
\end{table}

The polynomials $P_n$ 
 satisfy a three-term recursion of the form
$$A_nP_{n+1}=B_nP_n+C_nP_{n-1}, $$
where $A_n=A_n(x,\zeta)$ and $C_n=C_n(x,\zeta)$ are quadratic polynomials in $x$, while $B_n=B_n(x,\zeta)$ is  cubic in $x$, see Proposition \ref{prp}.
This is   reminiscent of the recursion 
satisfied by ortho\-gonal polynomials (where $A_n$ and $C_n$ are 
constants and $B_n$ is linear). It also seems that $P_n$ resemble orthogonal polynomials with respect to their zeroes.
Indeed, the following fact is proved in \S \ref{zs}.

\begin{proposition}\label{zcp} Assume \emph{Conjecture \ref{nvc}} (or, in view of \emph{Proposition \ref{zp}}, \emph{Conjecture \ref{zc}}). Then, if $-2<\zeta<-1/2$ and $\zeta\neq -1$,  all zeroes of $P_n(x,\zeta)$ are simple
and positive. 
Moreover,  the zeroes of $P_{n+1}$
interlace those of  $P_n$.
\end{proposition}

Computer calculations suggest that, if $-2<\zeta<-1$, the zeroes are in fact contained in the interval $x>1$, while if  $-1<\zeta<-1/2$, they are contained
in $0<x<1$.

\subsection{Symmetric polynomials $S_n$}

The key to the proof of most our results is that the
 polynomials $P_n$ and  $p_n$ can be obtained as specializations of  certain  symmetric polynomials $S_n$ of $2n+1$ variables.
To define them, we introduce the elementary polynomials
\begin{equation}\label{fgh}f(x)=(\zeta+2)x-\zeta, \qquad g(x)=(x-1)(x-\zeta), \qquad h(x)=x^2(x-(2\zeta+1)),\end{equation}
 and then let
\begin{equation}\label{p}\begin{split}F(x,y,z)&=\frac 1{(y-x)(z-x)(z-y)}\,\det\left(\begin{matrix}f(x) & f(y) &f(z)\\
g(x)&g(y)&g(z)\\
h(x)&h(y)&h(z)
\end{matrix}\right)\\
&=(\zeta+2)xyz-\zeta(xy+yz+xz+x+y+z)+\zeta(2\zeta+1),\end{split}\end{equation}
\begin{multline}\label{q}G(x,y)=\frac 1{y-x}\,\det\left(\begin{matrix}f(x) & f(y)\\
h(x)&h(y)
\end{matrix}\right)\\
=(\zeta+2)xy(x+y)-\zeta(x^2+y^2)-2(\zeta^2+3\zeta+1)xy+\zeta(2\zeta+1)(x+y).
\end{multline}
In this notation, 
\begin{multline}\label{sn}S_n(x_1,\dots,x_n,y_1,\dots,y_n,z)\\
=\frac{\prod_{i,j=1}^nG(x_i,y_j)}{\prod_{1\leq i<j\leq n}(x_j-x_i)(y_j-y_i)}\,\det_{1\leq i,j\leq n}\left(\frac{F(x_i,y_j,z)}{G(x_i,y_j)}\right),
 \end{multline}
the dependence on  $\zeta$ being suppressed from the notation. 
Though it is not  apparent from this definition, 
 $S_n$ is a symmetric polynomial of all $2n+1$ variables, see
 \S \ref{us}. 
 Moreover,
 \begin{multline*}S_n\Big(x,\underbrace{2\zeta+1,\dots,2\zeta+1}_{n},\underbrace{\frac\zeta
{\zeta+2},\dots,\frac\zeta{\zeta+2}}_{n}\Big)\\
=(-1)^{\binom{n+1}2}\zeta^{n^2}(1+\zeta)^{n^2}(1+2\zeta)^{\big[\frac{(n-1)^2}4\big]}\left(1+\frac\zeta2\right)^{\big[\frac{(n-1)^2}4\big]-n^2}P_n(x,\zeta),
\end{multline*}
 see Proposition \ref{efp}.

We mention that the special case $\zeta=-2$  of  $S_n$ is an orthogonal
character: 
$$S_n(x_1,\dots,x_{2n+1})\Big|_{\zeta=-2}=2^{n^2}\chi_{(n,n,n-1,n-1,\dots,1,1,0)}^{\mathfrak{so}(4n+3)}(t_1,\dots,t_{2n+1}),$$
where $x_i=-(1+t_i+t_i^{-1})$; 
see Theorem \ref{tst}  for a more general result. For general $\zeta$,  $S_n$  
can probably be interpreted as  an affine Lie algebra character; see Remark~\ref{alr}.

\subsection{Thermodynamic limit}
All  results mentioned so far concern properties of the partition function for fixed $n$. However, from the viewpoint of statistical mechanics,
 the main problem is to investigate the asymptotics  as $n\rightarrow\infty$. A particularly important quantity is the 
 free energy per volume, which we identify with the limit
\begin{equation}\label{fe}f(t_0,t_1,t_2)=\lim_{n\rightarrow\infty}\frac{\log Z_n^{3C}(t_0,t_1,t_2)}{n^2}. \end{equation}
 We propose the following explicit expression for $f$. 
In \S \ref{tds} we present a formal derivation of this result; to give a rigorous proof is presumably quite difficult.

\begin{conjecture}\label{tc}
Assume that the parameters $t_i$ are all positive, and let
 $\zeta$ be any positive solution of
\begin{equation}\label{tz}T=\frac{(t_0t_1+t_0t_2+t_1t_2)^3}{(t_0t_1t_2)^2}=\frac{2(\zeta^2+4\zeta+1)^3}{\zeta(\zeta+1)^4}.\end{equation}
Then, 
the free energy per volume is given by
\begin{equation}\label{fee}f(t_0,t_1,t_2)=\frac 13\log(t_0t_1t_2)+\log
\left(\frac{(\zeta+2)^{\frac 34}(2\zeta+1)^{\frac34}}{2^\frac 23\zeta^{\frac{1}{12}}(\zeta+1)^{\frac43}}\right). \end{equation}
\end{conjecture}

Note that, by the  arithmetic-geometric inequality,
$T\geq 27$.
Using that
\begin{equation}\label{tr}\frac{2(\zeta^2+4\zeta+1)^3}{\zeta(\zeta+1)^4}-27=\frac{(\zeta-1)^4(\zeta+2)(2\zeta+1)}{\zeta(\zeta+1)^4}, \end{equation}
it is easy to check that \eqref{tz} always has two positive solutions, except in the case $T=27$, when $\zeta=1$.
Moreover, the two solutions are related by $\zeta\mapsto \zeta^{-1}$, which does not change the right-hand side of \eqref{fee}.

As an example, when  $t_0=t_1=t_2=\zeta=1$, \eqref{fee} gives
$$f(1,1,1)=\log\left(\frac{3\sqrt 3}4\right), $$
which agrees with the known asymptotics \cite{bf}
\begin{equation}\label{aa}Z_n^{3C}(1,1,1)=A_n\sim  Cn^{-\frac 5{36}}\left(\frac{3\sqrt 3}4\right)^{n^2},\qquad n\rightarrow\infty. \end{equation}

One should compare Conjecture \ref{tc} with Baxter's result for \emph{periodic} boundary conditions \cite{b2}. In that case, one still has
$$f(t_0,t_1,t_2)=\frac 13\log(t_0t_1t_2)+\log W,$$
with $W$  a function only of $T$.  Baxter gives the formula
$$W^2=\frac{64(1-9t^2)^{\frac 23}}{27(1+t)^3(1-3t)}, $$
where $t$ is the unique solution of
$$\frac T{27}=\frac{(1-3t^2)^3}{1-9t^2} $$
such that $0\leq t<1/3$. It is straight-forward to check that
$t=(\zeta-1)/3(\zeta+1)$, where  $\zeta$ is the unique solution of \eqref{tz} such that $\zeta\geq 1$. This gives
$$W=W_{\text{per}}=\frac{2^{\frac 53}\zeta^\frac 13 (\zeta+1)^\frac 43}{(2\zeta+1)^\frac 32}, $$
which is manifestly different from our conjectured formula 
$$W_{\text{DWBC}}=\frac{(\zeta+2)^{\frac 34}(2\zeta+1)^{\frac 34}}{2^\frac 23\zeta^{\frac{1}{12}}(\zeta+1)^{\frac 43}} $$
for domain wall boundary conditions. Note
 the intriguing relation
$$W_{\text{DWBC}}(\zeta)=\frac 2{\sqrt{W_{\text{per}}(\zeta)W_{\text{per}}(1/\zeta)}}.$$

\section{Preliminaries}
\label{ps}

\subsection{Theta functions}\label{tfs}

We will work on the multiplicative  torus $\mathbb C^\ast/\{z=pz\}$, where 
$\mathbb C^\ast =\mathbb C\setminus\{0\}$ and 
$0<|p|<1$. We introduce the  theta function
$$\theta(x;p)=\prod_{j=0}^\infty(1-p^j x)(1-p^{j+1}/ x). $$
 We often employ condensed notation such as
$$\tha(a_1,\dots,a_n;p)=\tha(a_1;p)\dotsm\tha(a_n;p), $$
$$\tha(xy^{\pm};p)=\tha(xy;p)\tha(xy^{-1};p). $$

The most fundamental relations for the theta function are
$$\tha(px;p)= \tha(x^{-1};p)=-x^{-1}\tha(x;p), $$
together with the addition formula
\begin{equation}\label{tad}\tha(xz^{\pm},yw^{\pm};p)
-\tha(xw^{\pm},yz^{\pm};p)
=\frac yz\,\tha(xy^{\pm},zw^{\pm};p).\end{equation}
Other elementary identities that we will use include
$$\tha(x^2;p^2)=\tha(\pm x;p)
=\theta(\pm x,\pm px;p^2),$$
from which one can derive
\begin{subequations}\label{tq}
\begin{equation}\label{tqa}\theta(-1,\pm p ;p^2)=2, \end{equation}
\begin{equation}\label{tqb}\theta(-\om,\pm p\om;p^2)=-\om^2, \end{equation}
\end{subequations}
where $\om=e^{2\pi i/3}$.

We will use the following terminology  from \cite{rs}.

\begin{definition}
A holomorphic function on $\mathbb C^\ast$ is called
 an $A_{n-1}$ theta function of   nome $p$ and norm $t$ if it satisfies 
$$f(px)=\frac{(-1)^n}{tx^n}\,f(x). $$
It is called a $BC_n$ theta function of nome $p$ if
$$f(px)=\frac{1}{p^nx^{2n+1}}f(x),\qquad f(x^{-1})=-x^{-1}f(x). $$
Finally, it is called a $D_n$ theta function of nome $p$ if
$$f(px)=\frac{1}{p^{n-1}x^{2n-2}},\qquad f(x^{-1})=f(x). $$
\end{definition}

\begin{lemma}[{\cite[Lemma 3.2]{rs}}]\label{tfl}
A function $f$ is an  $A_{n-1}$ theta function of nome $p$ and norm $t$  if and only if it can be factored as 
$$f(x)=C\theta(a_1x,\dots,a_nx;p), $$
where $C\in\mathbb C$ and $a_i\in \mathbb C^\ast$ with  $a_1\dotsm a_n=t$.
It  is a $BC_n$ theta function of nome $p$ if and only if 
$$f(x)=C\theta(x,\pm\sqrt px;p)\theta(a_1x^{\pm},\dots,a_{n-1}x^{\pm};p), $$
where  $C\in\mathbb C$ and $a_i\in \mathbb C^\ast$. Finally, it is a $D_n$ theta function of nome $p$ if and only if 
$$f(x)=C\theta(a_1x^{\pm},\dots,a_{n-1}x^{\pm};p), $$
where  $C\in\mathbb C$ and $a_i\in \mathbb C^\ast$.
\end{lemma}

The terminology  is motivated by
Macdonald's theory of affine root systems \cite{m}.
In each case, the Macdonald identity for the affine root system $R$ 
is equivalent to evaluating a determinant $\det(f_i(x_j))$, where $f_i$ runs through a basis in the space of $R$ theta functions, see 
\cite[Proposition 6.1]{rs}. We will need two  special cases, which are both classical theta function identities. For $R=A_1$, the
 Macdonald identity  can be written
\begin{equation}\label{et}\tha(ax,bx;p)=\frac 1{\tha(p;p^2)}\left(\tha(-pa/b,-abx^2;p^2)-bx\tha(-a/b,-pabx^2;p^2)\right); \end{equation}
this can also be obtained from \eqref{tad}.
The Macdonald identity for  $BC_1$ is  Watson's quintuple product \cite{w}
\begin{equation}\label{qp}\tha(x,\pm\sqrt px;p)=\frac{(p^3;p^3)_\infty}{(p;p)_\infty}\left(\tha(-px^3;p^3)-x\tha(-px^{-3};p^3)\right), \end{equation}
where
\begin{equation}\label{ep}(p;p)_\infty=\prod_{j=1}^\infty(1-p^{j}).\end{equation}

Finally, we recall some classical facts on uniformization.
The Riemann surface $S=\mathbb C^\ast/\{z=pz=z^{-1}\}$ is a sphere, the analytic automorphisms 
$\xi:\,S\rightarrow\mathbb C\cup\{\infty\}$ being given by
$$\xi(z)=C\frac{\tha(az^{\pm};p)}{\tha(bz^{\pm};p)}, $$
where $C\in\mathbb C^\ast$ and $a,b$ are distinct in $\mathbb C^\ast/\{z=pz=z^{-1}\}$. 
Accordingly,
any $D_{n+1}$ theta function $f$ can be written
$$f(z)=\tha(bz^{\pm};p)^{n}P(\xi(z)), $$
where
$P$ is a polynomial of degree at most $n$. 
We will refer to the passage from $f$ to $P$ as \emph{uniformization}. When  $f$ is factored as in Lemma~\ref{tfl},
 the explicit uniformization is
\begin{multline}\label{eu}\tha(a_1z^{\pm},\dots,a_nz^{\pm};p)=
\frac{a_1\dotsm a_n\tha(ba_1^{\pm},\dots,ba_n^{\pm};p)}{C^na^n\tha(ba^{\pm};p)^n}\\
\times\tha(bz^{\pm};p)^n(\xi(z)-\xi(a_1))\dotsm(\xi(z)-\xi(a_n)).
 \end{multline}
This follows immediately from the case $n=1$, which is equivalent
to  \eqref{tad}.

\subsection{The 8VSOS model}\label{tcs}

The 8VSOS model is an ice model, so for domain wall boundary conditions  states can be identified with three-coloured chessboards. 
Fixing the size of the chessboards to  $(n+1)\times (n+1)$, the model
 depends on parameters
 $$x_1,\dots,x_n,y_1,\dots,y_n,\lambda,p\in\mathbb C^\ast,\qquad |p|<1.$$
In general, there is also a crossing parameter $q$, but for our purposes it can be fixed to   a cubic root of unity.

We will assign a weight 
to each $2\times 2$-block of adjacent squares. These blocks can be viewed as
entries of an $n\times n$ matrix,  and are given coordinates $1\leq i,j\leq n$ in a standard way. A block  $\left(\begin{smallmatrix}a&b\\ c&d\end{smallmatrix}\right)$ with coordinates $(i,j)$ is then given the weight
$R^{b-a,d-b}_{d-c,c-a}(\lambda \om^a,x_i/y_j),$
where 
$$R^{++}_{++}(\lambda,u)=R^{--}_{--}(\lambda,u)=\frac{\theta(\om u;p)}{\theta(u;p)}, $$
$$R^{+-}_{+-}(\lambda,u)=\frac{\theta(u,\om\lambda;p)}{\theta(\om,\lambda;p)},\qquad R^{-+}_{-+}(\lambda,u)=\om\frac{\theta(u,\om^2\lambda;p)}{\theta(\om,\lambda;p)},$$
$$R^{-+}_{+-}(\lambda,u)=\frac{\theta(\lambda u;p)}{\theta(\lambda;p)},\qquad R^{+-}_{-+}(\lambda,u)=u\frac{\theta(\lambda/u;p)}{\theta(\lambda;p)}.$$
Here, $\pm$ is a short-hand for $\pm 1$ $\operatorname{mod}\ 3$.

The partition function is now defined as
$$Z_n^{\text{8VSOS}}(x_1,\dots,x_n;y_1,\dots,y_n;\lambda,p)=\sum_{\text{chessboards}}\,\prod_{\text{blocks}}\,\text{weight}(\text{block}).$$
This differs slightly from the normalization used in  \cite{r}. We have
$$Z_n^{\text{8VSOS}}(x;y;\lambda,p)=(y_1\dots y_n)^{-n}\tilde Z_n(x;y;\lambda), $$
where $\tilde Z_n$ is as in \cite[\S 7]{r}.

In the specialization  $x_i\equiv \om$, $y_i\equiv 1$, the 8VSOS 
model  reduces to  the three-colour model with parameters
\begin{equation}\label{ti}t_i=\frac{1}{\theta(\lambda\omega^i;p)^3}.\end{equation}
More precisely, in the case of domain wall boundary conditions,
\begin{multline}\label{ter}Z_n^{3C}\left(\frac{1}{\theta(\lambda;p)^3},\frac{1}{\theta(\lambda\omega;p)^3},\frac{1}{\theta(\lambda\omega^2;p)^3}\right)\\
=\omega^{n(n+1)}\frac{\theta(\lambda\omega^2,\lambda\omega^{n+1};p)^2}{\theta(\lambda\omega^{n};p)\theta(\lambda^3;p^3)^{n^2+2n+2}}\, Z_n^{\text{8VSOS}}(\omega,\dots,\omega;1,\dots,1;\lambda,p).\end{multline}
See \cite[\S 8]{r} for the trigonometric case $p=0$; the discussion there 
carries over \emph{verbatim} to general $p$.

We need to mention the recursion
\begin{multline}\label{znr}
Z_n^{\text{8VSOS}}(x_1,\dots,x_n;y_1,\dots,y_n;\lambda,p)\Big|_{\om x_1=y_1}\\
=\om^{n+1}\frac{\tha(\la \om^n;p)\prod_{k=2}^n\tha(y_1\om^2/y_k,x_k/y_1;p)}{\tha(\la \om^{n-1};p)\tha(\om;p)^{2n-2}}\,Z_{n-1}^{\text{8VSOS}}(x_2,\dots,x_n;y_2,\dots,y_n;\la,p), 
\end{multline}
see e.g.\ \cite[Lemma 3.3]{r}, and the crossing symmetry 
\begin{multline}\label{cs}
Z_n^{\text{8VSOS}}(\om^2/x_1,\dots,\om^2/x_n;1/y_1,\dots,1/y_n;\om^{2n}/\lambda,p)\\
=\frac{\om^{n(n-1)}\theta(\lambda;p)Y^n}{\theta(\lambda \om^n;p)X^n}\,
Z_n^{\text{8VSOS}}(x;y;\lambda,p),
\end{multline}
where
$$X=x_1\dotsm x_n, \qquad Y=y_1\dotsm y_n.$$
The equation \eqref{cs} can be derived from a corresponding symmetry of the Boltzmann weights, and is also apparent from the explicit formulas
for the partition function given in \cite{r}.

The main result of \cite {r} is the explicit expression
\begin{multline}\label{mmi} Z_n^{\text{8VSOS}}(x;y;\lambda,p)=
\frac{(-1)^{\binom n2}\theta(\lambda \om^n;p)}{\theta(\om;p)^{n^2}\theta(\gamma;p)^n Y^{n+1}\theta(X\lambda\gamma \om^n/Y;p)}\\
\begin{split}&\times\frac{\prod_{i,j=1}^ny_j^2 \theta(x_i/y_j,\om x_i/y_j;p)}{\prod_{1\leq i<j\leq n}x_jy_j\theta(x_i/x_j,y_i/y_j;p)}\\
&\times\sum_{S\subseteq \{1,\dots,n\}}(-1)^{|S|}
\frac{\theta(\lambda\gamma \om^{n-|S|};p)}{\theta(\lambda \om^{n-|S|};p)}\det_{1\leq i,j\leq n}\left(\frac{\theta(\gamma x_i^S/y_j;p)}{\theta(x_i^S/y_j;p)}\right),
\end{split}\end{multline}
where
$$x_i^S=\begin{cases}x_i\om,& i\in S,\\ x_i,& i\notin S\end{cases} $$
and $\gamma$ is arbitrary.
This can be viewed as an analogue of the Izergin--Korepin formula for the six-vertex model.
We will not work with this formula directly, though we need the following immediate consequence.

\begin{proposition}[{\cite[Corollary 5.4]{r}}]\label{lap}
As a function of $\lambda$,
$$\theta(\lambda\omega^{n+1},\lambda\omega^{n+2};p)Z_n^{\text{\emph{8VSOS}}}(x;y;\lambda,p) $$
is analytic on $\mathbb C^\ast$. More precisely, it is an $A_1$ theta function of nome $p$ and norm $\om^{2n}Y/X$.
\end {proposition}

By \eqref{ter}, it follows that
\begin{equation}\label{tct}\frac{\theta(\lambda^3;p^3)^{n^2+2n+3}}{\theta(\lambda\omega^2,\lambda\omega^{n+1};p)^2}\,Z_n^{3C}\left(\frac{1}{\theta(\lambda;p)^3},\frac{1}{\theta(\lambda\omega;p)^3},\frac{1}{\theta(\lambda\omega^2;p)^3}\right) 
\end{equation}
is analytic in $\lambda\in\mathbb C^\ast$.
This remarkable fact is a key result for the present work. We do not know how to prove it except as a consequence of \eqref{mmi}.

We  also need the following functional equation, which was recently obtained by Razumov and Stroganov.

\begin{proposition}[{\cite{ras}}]\label{ras}
Let 
$$F_n(x;y;\lambda,p)=\theta(\lambda\omega^{n+1},\lambda\omega^{n+2};p)
\Delta(x,y;p)
Z_n^{\text{\emph{8VSOS}}}(\omega x;y;\lambda;p), $$
where 
\begin{equation}\label{evdm}\Delta(x_1,\dots,x_N;p)=\prod_{1\leq i<j\leq N}x_j\theta(x_i/x_j;p). \end{equation}
Then,
$$\sum_{k=0}^2F_n(\omega^kx_1,x_2,\dots,x_n;y;\omega^{-k}\lambda,p)=0. $$
\end{proposition}

Consider $F_n $ as a function of $\lambda$. 
By Proposition \ref{lap}, it is an 
$A_1$ theta function of nome $p$ and norm $\om^{n}Y/X$.
By \eqref{et}, it can be decomposed as
$$F_n(x;y;\lambda,p)=F_n^{(0)}(x;y;p)\theta(-\omega^n\lambda^2Y/X;p^2)
+
F_n^{(1)}(x;y;p)\lambda\theta(-p\omega^n\lambda^2Y/X;p^2).$$
The following result is then  immediate from Proposition \ref{ras}.

\begin{corollary}\label{fec} For $i=0,1$,
$$\sum_{k=0}^2F_n^{(i)}(\omega^kx_1,x_2,\dots,x_n;y)=0.$$
\end{corollary}

\section{Proof of Theorem \ref{st}}\label{qrs}

Consider \eqref{tct} as a function of $\lambda$. 
By Proposition \ref{lap},
it is  an $A_1$ theta function of nome $p$ and norm $\om^{n}$. Since the space of such functions is  spanned by
$\tha(\la\om^{2n};p)^2$ and $\tha(\la\om^{2n+1},\la\om^{2n+2};p)$, there exist functions $X_n$ and $Y_n$  such that
\begin{multline}\label{zfd}
\frac{\theta(\lambda^3;p^3)^{n^2+2n+3}}{\theta(\lambda\omega^2,\lambda\omega^{n+1};p)^2}\,Z_n^{3C}(t_0,t_1,t_2)\\
=X_n(p)\tha(\la\om^{2n};p)^2+Y_n(p)\tha(\la\om^{2n+1},\la\om^{2n+2};p),
\end{multline}
where $t_i$ are as in \eqref{ti}.
We will see that this decomposition corresponds to Theorem~\ref{st}, the $p$-dependence being encoded in the polynomials $q_n$ and $r_n$.

We first give some preliminary results.

\begin{lemma}\label{fil}
There exists a function $p\mapsto  \tau(p)$ such that,
under the parametrization \eqref{ti},
\begin{equation}\label{ssf}\frac 1{t_0}+\frac 1{t_1}+\frac 1{t_2}=\tau(p)\tha( \la^3;p^3). \end{equation}
Moreover, the quantity $T(t_0,t_1,t_2)$, defined in \eqref{om}, equals
$\tau(p)^3$.
 \end{lemma}

\begin{proof}
Each of the functions $1/t_i=\tha(\la\om^i;p)^3$, as well as $\tha( \la^3;p^3)$, are $A_2$ theta functions  of  nome $p$ and norm $1$. Thus,  the first statement 
can be reduced to the trivial verification that 
 $\sum_i 1/t_i$ vanishes at the  points $\lambda=1,\om,\om^2$. 
The second statement follows using
\begin{equation}\label{tsf}t_0t_1t_2=\frac 1{\tha(\la^3;p^3)^3}.\end{equation}
\end{proof}

It was pointed out to us by Don Zagier that
\begin{equation}\label{cp}
\tau(p)=3\left(1+9p\frac{(p^9;p^9)_\infty^3}{(p;p)_\infty^3}\right),
\end{equation}
where we use the notation \eqref{ep}. We also have
(see the remark after Lemma \ref {ztl})
\begin{equation}\label{omp}\tau(p)^3=27\left(1+27p\frac{(p^3;p^3)_\infty^{12}}{(p;p)_\infty^{12}}\right). \end{equation}
This is consistent with \eqref{cp} in view of the identity
$$ \left(1+9p\frac{(p^9;p^9)_\infty^3}{(p;p)_\infty^3}\right)^3=1+27p\frac{(p^3;p^3)_\infty^{12}}{(p;p)_\infty^{12}},$$
which can be found in Ramanujan's notebooks \cite[p.\ 345]{br}. 
We mention that
the function 
$$p\frac{(p^3;p^3)_\infty^{12}}{(p;p)_\infty^{12}}$$
is well-known in the theory of modular forms. It 
is automorphic under the  group
$$\Gamma_0(3)=\left\{\left(\begin{matrix}a&b\\c&d\end{matrix}\right)\in\mathrm{SL}(2,\mathbb Z);\,c\equiv 0\ \operatorname{mod}\ 3\right\} $$
(acting by
 $\tau\mapsto(a\tau+b)/(c\tau+d)$,  
where $p=e^{2\pi i\tau}$)
and in fact generates the field of all such functions \cite[Thm.\ 21]{sc}.

\begin{lemma}\label{cpl}
Let $f$ be a homogeneous rational function in three variables. Suppose that, 
 under the parametrization \eqref{ti}, $f(t_0,t_1,t_2)\equiv 0$. Then, $f\equiv  0$.
\end{lemma}

\begin{proof}  It is enough to show that the Jacobian of the map 
 $(\la,\mu,p)\mapsto(\mu t_0,\mu t_1,\mu t_2)$
does not vanish identically. Clearing denominators, the Jacobian is proportional to
$$\det\left(\begin{matrix}\tha(\la;p) & \tha_\la'(\la;p) & \tha_p'(\la;p) \\
\tha(\la\om;p) & \om\tha_\la'(\la\om;p) & \tha_p'(\la\om;p) \\
\tha(\la\om^2;p) & \om^2\tha_\la'(\la\om^2;p) & \tha_p'(\la\om^2;p)
\end{matrix}
\right). $$
Using that 
$$\tha(x;p)=(1-x)\left(1-(x+x^{-1})p+(1-x-x^{-1})p^2\right)+\mathcal{O}(p^3), $$
one can check that the determinant is 
$$\frac{3\sqrt 3 i(1+2\la^3)(1-\la^3)}{\la^2}\,p+\mathcal{O}(p^2). $$
\end{proof}

\begin{lemma}\label{lil}
Let $f$ be a Laurent polynomial in three variables, homogeneous of degree $0$. Suppose that, under the parametrization \eqref{ti}, $f(t_0,t_1,t_2)$ is independent of $\lambda$. Then, $f$ is a polynomial in $T$.
\end{lemma}

\begin{proof}
By the change of variables $\lambda\mapsto \om\lambda$, it is seen that 
$f(t_0,t_1,t_2)=f(t_1,t_2,t_0)$. A priori, this holds when $t_i$ are as in \eqref{ti}, but by Lemma \ref{cpl} it is valid in general.
Similarly, $\lambda\mapsto \lambda^{-1}$
gives $f(t_0,t_1,t_2)=f(t_0,t_2,t_1)$. Since these two transformations generate
$S_3$, $f$ is symmetric, and can thus be expressed as
$$f(t_0,t_1,t_2)=\sum_{k+2l+3m=0,\,k,l\geq 0}C_{klm}(t_0+t_1+t_2)^k(t_0t_1+t_0t_2+t_1t_2)^l(t_0t_1t_2)^m. $$
Introducing the function
$$\phi(\la,p)=(t_0+t_1+t_2)\tha(\la^3;p^3)^3 $$
and using \eqref{ssf} and \eqref{tsf} gives
$$f(t_0,t_1,t_2)=\sum_{k\geq 0}\frac{\phi(\la,p)^k}{\tha(\la^3;p^3)^k}\sum_{2l+3m=-k,\ l\geq 0}C_{klm}\tau(p)^l.$$
We observe that, since
$$\phi(1,p)=\tha(\om,\om^2;p)^3\neq 0, $$
the $k$th term has a pole at $\lambda=1$ of order exactly $k$. In particular, if $f$ is independent of $\lambda$, then all terms with $k\neq 0$ vanish. Thus,
$$f(t_0,t_1,t_2)=\sum_{2l+3m=0,\,l\geq 0}C_{0lm}(t_0t_1+t_0t_2+t_1t_2)^l(t_0t_1t_2)^m,  $$
which  is indeed a polynomial in $T$.
\end{proof}

We are now ready to prove  Theorem \ref{st}. Assume that $n\equiv 0\ \operatorname{mod}\ 3$. Then, \eqref{zfd} can be written
\begin{equation*}\begin {split}Z_n^{\text{3C}}(t_0,t_1,t_2)&=\frac{1}{\tha(\la^3;p^3)^{n(n+2)}}\left(\frac{X_n(p)}{\tha(\la^3;p^3)}+\frac{Y_n(p)}{\tha(\la;p)^3}\right)
\\
&=(t_0t_1t_2)^{\frac{n(n+2)}3}\left(\frac{X_n(p)\tau(p)t_0t_1t_2}{t_0t_1+t_0t_2+t_1t_2}+Y_n(p)t_0\right),
\end{split}\end{equation*}
where $\tau(p)$ is as in Lemma \ref{fil}. 
This implies
\begin{equation}\label{xz}X_n(p)\tau(p)=\frac{t_0t_1+t_0t_2+t_1t_2}{(t_0t_1t_2)^{\frac{n(n+2)}3+1}}\cdot\frac{t_1Z_n(t_0,t_1,t_2)-t_0Z_n(t_1,t_0,t_2)}{t_1-t_0}.
\end{equation}
Since the right-hand side vanishes when $T=0$, it follows from Lemma  \ref{lil} that $X_n(p)\tau(p)$ is a polynomial in $T$ divisible by $T$. 
 Similarly,
$$Y_n(p)=\frac 1{(t_0t_1t_2)^{\frac{n(n+2)}3}}\cdot\frac{Z_n(t_0,t_1,t_2)-Z_n(t_1,t_0,t_2)}{t_1-t_0}, $$
so $Y_n(p)$ is a polynomial in $T$. Writing
\begin{subequations}\label{aba}
\begin{equation}\label{abaa}X_n(p)\tau(p)=(-1)^{n+1}Tq_n(T), \end{equation}
\begin{equation}Y_n(p)=(-1)^n2^{\chi (n \text{ odd})}r_n(T), \end{equation}
we obtain Theorem \ref{st} for $n\equiv 0\ \operatorname{mod}\ 3$. The cases 
$n\equiv\pm 1\ \operatorname{mod}\ 3$ are similar and we do not give the details. For later use we note that in both cases
\begin{equation}X_n(p)=(-1)^{n+1}q_n(T), \end{equation}
\begin{equation}Y_n(p)\tau(p)^2=(-1)^n2^{\chi (n \text{ odd})}Tr_n(T). \end{equation}
\end{subequations}

Finally, we comment on the claim made before stating Proposition
\ref{rip}, that  $q_n$, $r_n$ and $2r_{2n+1}$ have integer coefficients. Consider the case of $q_n$ when $n\equiv 0\ \operatorname{mod}\ 3$; all other cases follow similarly. By \eqref{xz} and \eqref{abaa}, 
\begin{multline*}\frac{t_1Z_n(t_0,t_1,t_2)-t_0Z_n(t_1,t_0,t_2)}{t_1-t_0}\\
=(-1)^{n+1}
(t_0t_1t_2)^{\frac{n(n+2)}3-1}(t_0t_1+t_0t_2+t_1t_2)^2q_n\left(\frac{(t_0t_1+t_0t_2+t_1t_2)^3}{(t_0t_1t_2)^2}\right).
 \end{multline*}
The left-hand side is a symmetric polynomial in $t_i$ with integer coefficients. By \cite[I.2.4]{mb}, it can be expanded as an integer linear combination of 
elementary symmetric polynomials. By the identity above, the non-zero coefficients in that expansion are coefficients of the polynomial $q_n$.

\section{The function $\Phi_n$}\label{phs}

In \S  \ref{qrs}, we  expressed the function \eqref{tct} 
in terms of  the basis
$$\left(\tha(\la\om^{2n};p)^2,\tha(\la\om^{2n+1},\la\om^{2n+2};p)\right).$$
The first main idea for analyzing the partition function further is suggested by   the trigonometric case $p=0$. The relevant space is then the polynomials $A+B\lambda+C\lambda^2$ such that $C=\om^nA$. 
From  \eqref{ttf}, it appears that the most natural basis to use is not
 $\left((1-\lambda\om^{2n})^2,(1-\lambda\om^{2n+1})(1-\lambda\om^{2n+2})\right)$
but rather $(1+\om^n\la^2,\la)$. An elliptic analogue of the latter basis is
$\left(\tha(-\om^{n}\la^2;p^2),\la\tha(-p\om^{n}\la^2;p^2)\right).$  
As we will eventually see,  this change of basis corresponds to expressing
$Z_n^{\text{3C}}$ in terms of the polynomials $p_n$ and $\tilde p_n$ rather than $q_n$ and $r_n$.

The next main idea is to consider a one-parameter extension of 
$Z_n^{\text{3C}}$, which is given by
 $Z_n^{\text{8VSOS}}$ with $\lambda$,  $p$ and $x_1$ free but all other parameters fixed. This corresponds to incorporating $p_n(\zeta)$ in the two-variable polynomial $P_n(x,\zeta)$. 

Combining these two ideas, we consider the function
\begin{equation}\label{nsp}\tha(\la\om^{n+1},\la\om^{n+2};p)Z_n^{\text{8VSOS}}(\om t,\om,\dots,\om;1,\dots,1;\lambda,p).\end{equation}
By  Proposition \ref{lap} and \eqref{et}, it is a linear combination of
$\tha(-\om^n\la^2/t;p^2)$ and $\la\tha(-p\om^n\la^2/t;p^2)$, 
with coefficients independent of $\lambda$. It will be convenient to
write the corresponding decomposition as
\begin{multline*}Z_n^{\text{8VSOS}}(\om t,\om,\dots,\om;1,\dots,1;\lambda,p)
=\frac 1{\tha(\la\om^{n+1},\la\om^{n+2};p)\tha(t;p)^{2n-1}}\\
\times\left(p^{\frac{3n-4}2}t^{\frac{9n-4}2}\tha(-\om^n\la^2/t;p^2)\tilde\Phi_n(t)-p^{-1}t^{\frac{3n-2}2}\om^{2n}\la\tha(-p\om^n\la^2/t;p^2)\Phi_n(t)\right).
 \end{multline*}
for $n$ even,
while for odd $n$  we write
\begin{multline*}Z_n^{\text{8VSOS}}(\om t,\om,\dots,\om;1,\dots,1;\lambda,p)
=\frac 1{\tha(\la\om^{n+1},\la\om^{n+2};p)\tha(t;p)^{2n-1}}\\
\times\left(t^{\frac{3n-1}2}\tha(-\om^n\la^2/t;p^2)\Phi_n(t)+p^{\frac{3n-3}2}t^{\frac{9n-5}2}\om^{2n}\la\tha(-p\om^n\la^2/t;p^2)\tilde\Phi_n(t)\right).
 \end{multline*}
The functions
$\Phi_n$ and $\tilde \Phi_n$ depend implicitly on $p$ but are independent of $\lambda$.

\begin{lemma}
The functions $\Phi_n$ and $\tilde \Phi_n$ are related by
$\tilde \Phi_n(t)=\Phi_n(pt)$. Moreover,
\begin{equation}\label{pqp}\Phi_n(p^2t)=\frac{1}{p^{6n-4}t^{6n-3}}\,\Phi_n(t).
\end{equation}
\end{lemma}

\begin{proof}
This follows easily from the fact that,
as  a function of $t$, \eqref{nsp} is an $A_{n-1}$ theta function of nome $p$ and norm $\om^n/\la$, see \cite[Lemma 3.2]{r}.
\end{proof}

As we will see,  $\Phi_n$ is uniquely determined by the following properties.

\begin{proposition}\label{php}
The function $\Phi_n$ 
 has the following properties:
\begin{enumerate}[(i)]
\item $\Phi_n$ is a $BC_{3n-2}$ theta function of nome $p^2$;
\item 
$\displaystyle\sum_{k=0}^2\omega^k\Phi_n(\omega^kt)=0$, that is,
$\Phi_n(t)=f(t^3)+tg(t^3)$, with  $f$ and $g$ analytic on $\mathbb C^\ast$;
\item $t=1$ and $t=p$ are zeroes of $\Phi_n(t)$ of multiplicity at least
 $2n-1$;
\item $\displaystyle \Phi_n(\om)=-p^{1+3\left[\frac{n-2}2\right]}\om\tha(\om;p)^{2n-1}\lim_{t\rightarrow p}\frac{\Phi_{n-1}(t)}{\tha(t;p)^{2n-3}}$;
\item $\displaystyle \Phi_1(t)= \frac{\om \tha(t,\pm pt;p^2)}{\tha(p;p^2)}$.
\end{enumerate}
\end{proposition}

\begin{proof}
We have already noted the quasi-periodicity \eqref{pqp}. For (i), it remains to  show that $\Phi_n(t^{-1})=-t^{-1}\Phi_n(t)$. This is a special case of \eqref{cs}.

Property 
(ii) is a special case of Corollary \ref{fec}, where we should note 
that
$$\Phi_n(t)=\begin{cases}
\displaystyle t^{(3n-2)/2}\frac{F_n^{(0)}(x;y)}{\Delta(\hat x,y;p)}\Bigg|_{x_1=pt,\,x_2=\dots=x_n=y_1=\dots=y_n=1}, & n \text{ even},\\[5mm]
\displaystyle t^{(1-3n)/2}
\frac{F_n^{(0)}(x;y)}{\Delta(\hat x,y;p)}\Bigg|_{x_1=t,\,x_2=\dots=x_n=y_1=\dots=y_n=1}
, & n \text{ odd},
\end{cases} $$
where $\Delta$ is as in \eqref{evdm} and the hats indicate omission of $x_1$.

Property (iii) is obvious, and
 property (iv) is a special case of \eqref{znr}.  

Finally, to check (v), we note that
$$F_1(x;y;\la)=y\tha(\om^2\la,x/y;p)R^{+-}_{-+}(\la,\om x/y)=
\om x\tha(x/y,\om^2\la,\om^2\la y/ x;p). $$
By \eqref{et}, this gives
$$F_1^{(0)}(x;y)=\om x\tha(x/y;p)\frac{\tha(-px/y;p^2)}{\tha(p;p^2)}, $$
$$\Phi_1(t)=\frac 1t\,F_1^{(0)}(t;1)=\frac{\om \tha(t;p)\tha(-pt;p^2)}{\tha(p;p^2)}=\frac{\om \tha(t,\pm pt;p^2)}{\tha(p;p^2)}. $$
\end{proof}

Finally, we use \eqref{ter} to express $Z_n^{3C}$ in terms of $\Phi_n$.

\begin{corollary}\label{tcpc}
When $t_i$ are as in \eqref{ti} and $n$ is even, then
\begin{multline*}Z_n^{3C}(t_0,t_1,t_2)=-\omega^{n(n+1)}\frac{\theta(\lambda\omega^2,\lambda\omega^{n+1};p)^2}{\theta(\lambda^3;p^3)^{n^2+2n+3}}\\
\times\left(p^{\frac{3n-4}2}\tha(-\om^n\la^2;p^2)\lim_{t\rightarrow p}\frac{\Phi_n(t)}{\tha(t;p)^{2n-1}}+p^{-1}\om^{2n}\la\tha(-p\om^n\la^2;p^2)\lim_{t\rightarrow 1}\frac{\Phi_n(t)}{\tha(t;p)^{2n-1}}\right),\end{multline*}
while if  $n$ is odd
\begin{multline*}Z_n^{3C}(t_0,t_1,t_2)=\omega^{n(n+1)}\frac{\theta(\lambda\omega^2,\lambda\omega^{n+1};p)^2}{\theta(\lambda^3;p^3)^{n^2+2n+3}}\\
\times\left(\tha(-\om^n\la^2;p^2)\lim_{t\rightarrow 1}\frac{\Phi_n(t)}{\tha(t;p)^{2n-1}}-p^{\frac{3n-3}2}\om^{2n}\la\tha(-p\om^n\la^2;p^2)\lim_{t\rightarrow p}\frac{\Phi_n(t)}{\tha(t;p)^{2n-1}}\right).\end{multline*}
\end{corollary}

\section{Symmetric functions}
\label{sfs}

We are now faced with two problems: to 
construct a function satisfying all properties of Proposition \ref{php}, and to show that this function is unique. The key for solving both problems is to replace property (iii) by a generic vanishing condition. This leads to certain symmetric multivariable theta functions, which after uniformization become the symmetric polynomials  \eqref{sn}.

\subsection{Symmetric theta functions}\label{stfs}

We  denote by $V_n$ the space of $BC_{3n+1}$ theta functions   of nome $p^2$ satisfying property (ii) of Proposition \ref{php}.
For $t_i\in\mathbb C^\ast/\{t=p^2t=t^{-1}\}$, we  denote 
by $V_n(t_1,\dots,t_{k})$
the subspace of $V_n$ consisting of functions $\psi$ such that,  
 apart from the trivial zeroes at $1$ and $\pm p$, $\psi$ vanishes at 
$t_1,\dots, t_{k}$ (counted with multiplicity). 
Then, properties (i)--(iii) of Proposition \ref{php} can be summarized as
\begin{equation}\label{piv}\Phi_n\in V_{n-1}(\underbrace{1,\dots,1}_{n-1},\underbrace{p,\dots,p}_{n-1}). \end{equation}

\begin{lemma}\label{dl} 
The space $V_n$ 
has dimension $2n+1$. The space  $V_n(t_1,\dots,t_k)$ has dimension at least $2n+1-k$, with equality for $k\leq n$.
\end{lemma}

\begin{proof}
A basis for the  $BC_{3n+1}$ theta functions
 of nome $p^2$ is given by
$$t^{j-3n-1}\theta(-p^{2j}t^{6n+3};p^{12n+6})-t^{3n+2-j}\theta(-p^{2j}t^{-6n-3};p^{12n+6}),\qquad 1\leq j\leq 3n+1, $$
  see 
\cite[Proposition 6.1]{rs}. It is clear that the subspace $V_n$ is spanned by the $2n+1$ basis vectors with $j\not\equiv 0\ \operatorname{mod}\ 3$.

Since $V_n(t_1,\dots,t_k)$ is obtained by imposing $k$ linear conditions on $V_n$, it  has dimension at least $2n+1-k$. For the final statement, note that by the quintuple product identity \eqref{qp}, $V_n$ contains all functions of the form 
$$\tha(t,\pm pt;p^2)\tha(b_1t^{\pm 3},\dotsm b_{n}t^{\pm 3};p^6).$$
When $b_i=t_i$ for $1\leq i\leq k$ and the remaining $b_i$ are generic, this function is in $V_n(t_1,\dots,t_k)\setminus V_n(t_1,\dots,t_{k+1})$. Thus, as long as we impose at most $n$ vanishing conditions, each additional condition decreases the dimension by one.
\end{proof}

Although we do not need it in full generality, the reader may find the  following characterization of the space $V_n$  helpful.

\begin{lemma}\label{rl}
Fix $\al\in\mathbb C^\ast/\{t=p^2t=t^{-1}\}$ with $\al\neq-\om$. Then,  $V_n$ is the space of all functions $\psi$ that can be written
\begin{multline*}\psi(t)=\tha(t,\pm pt;p^2)\left(A\,\tha(-t^\pm,-\om t^\pm,\al t^\pm;p^2)
\tha(a_1t^{\pm 3},\dotsm a_{n-1}t^{\pm 3};p^6)\right.\\
\left.+B\,\tha(b_1t^{\pm 3},\dotsm b_{n}t^{\pm 3};p^6)
\right), \end{multline*}
with $A,B\in\mathbb C$ and $a_1,\dots,a_{n-1},b_1,\dots,b_n\in\mathbb C^\ast$.
\end{lemma}

\begin{proof}
It is easy to check that any $BC_{3n+1}$ theta function 
$\psi$ of nome $p^2$
can be written
$$\psi(t)=\tha(t,\pm pt;p^2)\left(f(t^3)+tg(t^{3})+t^{-1}g(t^{-3})\right), $$where $f$ is a $D_{n+1}$ theta function of nome $p^6$, and $g$ satisfies
\begin{equation}\label{gq}g(p^6t)=\frac{1}{p^{6n+2}t^{2n}}\,g(t). \end{equation}
Since the term involving $f$ has the desired form, we  restrict to the case $f\equiv 0$.
By the quintuple product identity \eqref{qp}, $\psi\in V_n$ if and only if
$$\tha(-p^2t;p^6)g(t^{-1})=t\tha(-p^2t^{-1};p^6)g(t). $$
Together with \eqref{gq}, this implies that $g$ vanishes on all zeroes of
$\tha(-t,-p^2t;p^6)$. Factoring $g(t)=t^{-1}\tha(-t,-p^2t;p^6)h(t)$ gives
$$\psi(t)=t^{-2}\tha(t,\pm pt;p^2)\tha(-t^3;p^6)\left(\tha(-p^2t^3;p^6)+t\tha(-p^2t^{-3};p^6)\right)h(t^3), $$
where
$h$ is a  $D_{n}$ theta function of nome $p^6$.

Next, we observe that  $\tha(-p^2t^3;p^6)+t\tha(-p^2t^{-3};p^6)$ vanishes on the zeroes of $\tha(-t;p^2)$. Thus
$$\tha(-p^2t^3;p^6)+\tha(-p^2t^{-3};p^6)=t\tha(-t;p^2)k(t),
$$
where $k$ is a $D_2$ theta function of nome $p^2$, that is, $k(t)=C\tha(\be t^{\pm};p^2)$ for some $C$ and $\beta$ (depending on $p$).
After simplification, we conclude that
$$\psi(t)=C\tha(t,\pm pt;p^2)\tha(-t^\pm,-\om t^{\pm},\be t^{\pm};p^2)h(t^3). $$Using \eqref{tad} to write 
$$\tha(\be t^{\pm};p^2)=A\tha(\al t^{\pm};p^2)+B\tha(-\om^2 t^{\pm};p^2), $$
we arrive at an expression of the desired form. This shows that any $\psi\in V_n$ can be expressed as indicated. The converse follows by similar arguments.
\end{proof}

In view of \eqref{piv},
we are mainly interested in the space  $V_n(t_1,\dots,t_{2n})$. 
Generically, one expects it to be one-dimensional and spanned by
 the alternant
$$\psi(t)=\det_{1\leq i,j\leq 2n+1}(\psi_j(t_i)), $$
with $(\psi_j)_{j=1}^{2n+1}$ a basis of $V_n$ and
$t_{2n+1}=t$. 
However, we have not found such  expressions 
 useful for our purposes. Instead, we will work with
 the following less symmetric 
determinants.

\begin{theorem}\label{met}
Fixing a basis $\psi_1$, $\psi_2$, $\psi_3$ of $V_1$, define
$$\Psi_1(t_1,t_2,t_3)=\det_{1\leq i,j\leq 3}(\psi_j(t_i)) $$
and, more generally, 
$$\Psi_n(t_1,\dots,t_n,u_1,\dots,u_n,v)=\frac{\prod_{i,j=1}^nu_j^{-3}\theta(u_j^3t_i^{\pm 3};p^6)}{\tha(v,\pm pv;p^2)^{n-1}}\det_{1\leq i,j\leq n}\left(\frac{\Psi_1(t_i,u_j,v)}{u_j^{-3}\theta(u_j^3t_i^{\pm 3};p^6)}\right).
 $$
Let $(\psi_j^{(n)})_{j=1}^{2n+1}$ be a basis of $V_n$. Then, 
$$\Psi_n(t_1,\dots,t_{2n+1})=C\det_{1\leq i,j\leq 2n+1}(\psi_j^{(n)}(t_i)), $$
with $C$ independent of each $t_i$. In particular,
$\Psi_n$ is anti-symmetric in all $2n+1$ variables.
\end{theorem}

Note that $\Psi_n$ is only defined up to a multiplicative constant, which we do not specify. 

\begin{remark}\label{alr}
If we choose the basis $\psi_j^{(n)}$  as in the proof of Lemma~\ref{dl}, then the alternant is a minor of the determinant corresponding to the $BC_{3n+1}$ Macdonald identity. This should mean that $\Psi_n$ can be interpreted as an affine Lie algebra character.
Although this observation may have interesting consequences, 
 we will  not explore it here.
 In \S \ref{tes}, we will see that the trigonometric limit case $p=0$ corresponds to characters of the orthogonal and symplectic groups.
The corresponding limit of  Theorem  \ref{met}
is closely related to some determinant identities of
 Okada \cite{o}, see \eqref{oia}.
\end{remark}

We  divide the proof of Theorem \ref{met} into a few lemmas.

\begin{lemma}\label{fl}
For any $i=1,\dots,2n$, the map $t_i\mapsto\Psi_n(t_1,\dots,t_{2n+1})$ is an element of $V_n(t_1,\dots,\hat t_i,\dots,t_{2n+1})$. 
\end{lemma}

The proof of Lemma \ref{fl} is straight-forward. It will follow from Theorem \ref{met} that the statement holds also for $i=2n+1$.

\begin{lemma}\label{dcl}
If the map $t\mapsto\Psi_n(t,t_1,\dots,t_{2n})$ is not identically zero, then 
$$\dim V_n(t_1,\dots,t_{2n})=1.$$
\end{lemma}

\begin{proof}
Take $t_0$ with $\Psi_n(t_0,t_1,\dots,t_{2n})\neq 0$. Consider the functions 
$$\psi_k(t)=\Psi_n(t_0,t_1,\dots,t_{k-1},t,t_{k+1},\dots,t_{2n}),
\qquad k=0,1,\dots,2n-1.$$
By Lemma \ref{fl},  these functions all belong to the space
$V_n(t_{2n})$. Moreover, they are linearly independent since
$$\psi_k(t_j)\neq 0\quad \Longleftrightarrow \quad j=k, \qquad j,k=0,\dots,2n-1.$$ 
Since, by Lemma \ref{dl},  $\dim V_n(t_{2n})=2n$, we conclude that $\psi_0,\dots,\psi_{2n-1}$ span $V_n(t_{2n})$.

Suppose now that $f\in V_n(t_1,\dots,t_{2n})$. In particular,
$f\in V_n(t_{2n})$, so we can expand
$$f(t)=\sum_{k=0}^{2n-1}A_k\psi_k(t). $$
The remaining vanishing conditions for $f$ give $A_1=\dots=A_{2n-1}=0$.
Thus, $f$ is proportional to $\psi_0$, so $\dim  V_n(t_1,\dots,t_{2n})=1$.
\end{proof}

\begin{proof}[Proof of \emph{Theorem \ref{met}}]
Let
$$C(t_1,\dots,t_{2n+1})=\frac{\Psi_n(t_1,\dots,t_{2n+1})}{\det_{1\leq i,j\leq 2n+1}(\psi_j^{(n)}(t_i))}.$$
We first prove that $C$ is independent of  $t_1$. 
 By Lemma \ref{fl},  the denominator and numerator are both in 
$V_n(t_2,\dots,t_{2n+1})$. If that space is 
one-dimensional, $C$ is  independent of $t_1$. 
Otherwise, Lemma 
\ref{dcl} gives $C\equiv 0$.
The same argument applies to the variables $t_2,\dots,t_{2n}$.
 Thus,
$$\Psi_n(t_1,\dots,t_{2n+1})=C(t_{2n+1})\det_{1\leq i,j\leq 2n+1}(\psi_j^{(n)}(t_i)),$$
where it remains to show that $C$ is independent of $t_{2n+1}$. 

To complete the proof,  we 
write
\begin{equation}\label{cc}\frac{C(v)}{C(t_1)}=-\frac{\Psi_n(t_1,\dots,t_n,u_1,\dots,u_n,v)}{\Psi_n(v,t_2,\dots,t_n,u_1,\dots,u_n,t_1)}\end{equation}
and specialize   $u_n=\omega t_n$. It follows from the definition of $\Psi_n$ that 
\begin{multline*}\Psi_n(t_1,\dots,t_n,u_1,\dots,u_{n-1},\om t_n,v)
=\frac{\prod_{j=1}^nt_n^{-3}u_j^{-3}\tha(t_n^3t_j^{\pm 3},u_j^3t_n^{\pm 3};p^6)}{\tha(v,\pm pv;p^2)}\\
\times\Psi_1(t_n,\om t_n,v)\Psi_{n-1}(t_1,\dots,t_{n-1},u_1,\dots,u_{n-1},v).\end{multline*}
Using this in \eqref{cc} gives
\begin{equation*}\begin{split}\frac{C(v)}{C(t_1)}
&=-
\frac{\tha(t_1,\pm pt_1;p^2)\tha(t_n^3t_1^{\pm 3};p^6)\Psi_1(t_n,\om t_n,v)}{\tha(v,\pm pv;p^2)\tha(t_n^3v^{\pm 3};p^6)\Psi_1(t_n,\om t_n,t_1)}
\\
&\quad\times\frac{\Psi_{n-1}(t_1,\dots,t_{n-1},u_1,\dots,u_{n-1},v)}{\Psi_{n-1}(v,t_2,\dots,t_{n-1},u_1,\dots,u_{n-1},t_1)}. \end{split}\end{equation*}
We  claim that 
$$\frac{\tha(t_1,\pm pt_1;p^2)\tha(t_n^3t_1^{\pm 3};p^6)\Psi_1(t_n,\om t_n,v)}{\tha(v,\pm pv;p^2)\tha(t_n^3v^{\pm 3};p^6)\Psi_1(t_n,\om t_n,t_1)}=1. $$
To see this, consider the denominator and numerator as functions of $v$. They both belong to the space $V_1(t_n,\om t_n)$, which is one-dimensional for generic $t_n$. Thus, the quotient is independent of $v$ and can be computed by letting $v=t_1$. 

We have now reduced \eqref{cc} to
$$\frac{C(v)}{C(t_1)}=-
\frac{\Psi_{n-1}(t_1,\dots,t_{n-1},u_1,\dots,u_{n-1},v)}{\Psi_{n-1}(v,t_2,\dots,t_{n-1},u_1,\dots,u_{n-1},t_1)}. $$
By iteration, we  conclude that
 $$\frac{C(v)}{C(t_1)}=-
\frac{\Psi_{1}(t_1,u_1,v)}{\Psi_{1}(v,u_1,t_1)}=1. $$
This completes the proof of Theorem \ref{met}.
\end{proof}

\subsection{Uniformization}\label{us}

The next step is uniformization, see \S \ref{tfs}. 
We will work with the 
 uniformizing map
\begin{equation}\label{xi}\xi(t)=\frac{\theta(- p\omega;p^2)^2\theta(\omega t^\pm;p^2)}{\theta(-\omega;p^2)^2\theta( p\omega t^\pm;p^2)} \end{equation}
 and write
\begin{equation}\label{z}\zeta=\frac{\omega^2\theta(-1,- p\omega;p^2)}{\theta(- p,-\omega;p^2)}.\end{equation}

\begin{lemma}\label{xvl} We have
$$\xi(-1)=1, \qquad \xi(-\omega)=\zeta, $$
$$\xi(1)=2\zeta+1,\qquad \xi(p)=\frac{\zeta}{\zeta+2}. $$
\end{lemma}

\begin{proof}
The first two identities are obvious. By \eqref{eu}, 
\begin{equation}\label{xid}\xi(s)-\xi(t)=-\frac{\om\tha( p, p\om;p^2)\tha(- p\om;p^2)^2\tha(st^{\pm};p^2)}{s\tha(-\om;p^2)^2\tha( p\om s^\pm, p\om t^\pm;p^2)}. 
\end{equation}
Plugging in $s=1$ and $t=-1$, and using
\eqref{tq}, we obtain $\xi(1)-1=2\zeta$. The case $s= p$ and $t=-1$ similarly gives $\xi(p)-1=-2\xi(p)/\zeta$.
\end{proof}

We denote by
$W_n$ the space of polynomials 
$q$ of degree at most $3n$ such that
$$\psi(t)=\theta(t,\pm p t;p^2)\theta( p\omega t^\pm;p^2)^{3n} q(\xi(t)), $$
is an element of   $V_n$. Then, the correspondence between $\psi$ and $q$ is a bijection. We denote by $W_n(x_1,\dots,x_{k})$  the subspace of $W_n$ such that $\psi\in V_n(t_1,\dots,t_k)$, where $x_i=\xi(t_i)$.

\begin{lemma}\label{wl}
The space $W_1$ is spanned by the three polynomials $f$, $g$ and $h$ defined in \eqref{fgh}.
\end{lemma}

\begin{proof} By Lemma \ref{rl}, $V_1$ is spanned by the functions 
$$\tha(p^{3}t^{\pm 3};p^6), \qquad \tha(-t^{\pm},-\om t^{\pm},p\om t^\pm;p^2),\qquad \tha(t^{\pm 3};p^6).$$
 By Lemma \ref{xvl}, the uniformization of these functions are indeed, up to multiplicative constants, respectively
$$(\zeta+2)x-\zeta, \qquad (x-1)(x-\zeta), \qquad x^2(x-(2\zeta+1)).$$
\end{proof}

Taking determinants of the basis elements \eqref{fgh}, we get the following results.

\begin{corollary} For any elements $q_1$, $q_2$, $q_3$ of $W_1$,
$$\det_{1\leq i,j\leq 3}(q_j(x_i))=C(x_2-x_1)(x_3-x_1)(x_3-x_2)F(x_1,x_2,x_3),
 $$
where $C$ is a constant and  $F$ is as in \eqref{p}.

Moreover, writing $x=\xi(t)$ and $y=\xi(u)$, 
$$u^{-3}\theta(u^3t^{\pm 3};p^6)\tha(p\om t^\pm,p\om u^{\pm};p^2)
=C(y-x)G(x,y), $$
where $C$ is a constant and $G$ is as in \eqref{q}.
\end{corollary}

It follows that the function $\Psi_n$
from Theorem \ref{met} is given by
\begin{multline}\label{pss}\Psi_n(t_1,\dots,t_{2n+1})=C\prod_{i=1}^{2n+1}\tha(t_i,\pm pt_i;p^2)\tha(p\om t_i^{\pm};p^2)^n\prod_{1\leq i<j\leq 2n+1}t_j^{-1}\tha(t_jt_i^\pm;p^2)\\
\times S_n(x_1,\dots,x_{2n+1}), \end{multline}
where $C$ is a constant, $x_i=\xi(t_i)$, and $S_n$ is 
the polynomial introduced in \eqref{sn}.
In particular, $S_n$ is symmetric in all variables.

Finally, we note the following symmetry of $S_n$.

\begin{lemma}\label{il} Indicating also the dependence on $\zeta$, the
 polynomial $S_n$ satisfies
$$S_n(x_1^{-1},\dots,x_{2n+1}^{-1};\zeta^{-1})=\frac 1{\zeta^{2n^2}x_1^n\dotsm x_{2n+1}^n}\,S_n(x_1,\dots,x_{2n+1};\zeta) $$
\end{lemma}

\begin{proof}
This follows from an elementary computation, using
$$F(x^{-1},y^{-1},z^{-1};\zeta^{-1})=\frac{1}{\zeta^2xyz}\,F(x,y,z;\zeta), $$
$$G(x^{-1},y^{-1};\zeta^{-1})=\frac{1}{\zeta^2x^2y^2}\,G(x,y;\zeta). $$
\end{proof}

\subsection{Recursions}\label{rs}

Note that any minor  of the determinant in \eqref{sn} is  a determinant of the same type. Thus, any algebraic relation between minors (see e.g.\ \cite[Chapter 6]{n})  implies a relation involving the polynomials $S_n$. 
It is not our purpose  to give an exhaustive list of such identities; we only mention a few examples that will be used below or otherwise seem of particular interest.

First, we apply Jacobi's  identity (sometimes called the Lewis Carroll formula after one of its  proponents)
$$X_{n-1,n}^{n-1,n}X=X_n^nX_{n-1}^{n-1}-X_{n-1}^nX_n^{n-1}, $$
where
$X=\det_{1\leq i,j\leq n}(x_{ij})$
and the other quantities are minors, upper and lower indices signifying omitted 
rows and columns. Choosing
$$ x_{ij}=\frac{F(x_i,y_j,z)}{G(x_i,y_j)},$$
one obtains after  relabelling
 the following recursion for the polynomials $S_n$.

\begin{lemma}\label{tp}
For  $\mathbf x=(x_1,\dots,x_{2n-1})$,
\begin{multline*}(a-b)(c-d)S_{n-1}(\mathbf x)S_{n+1}(a,b,c,d,\mathbf x)\\
=G(a,d)G(b,c)S_{n}(a,c,\mathbf x)S_{n}(b,d,\mathbf x)-G(a,c)G(b,d)S_{n}(a,d,\mathbf x)S_{n}(b,c,\mathbf x).\end{multline*}
\end{lemma}

One of the Pl\"ucker relations is
$$X_{n,n+1}^nX_{n-1}-X_{n-1,n+1}^nX_{n}+X_{n-1,n}^nX_{n+1}=0, $$
valid for minors of an $n\times(n+1)$ matrix. 
Specializing the matrix entries as above yields the following result.

\begin{lemma}\label{pl}
For $\mathbf x=(x_1,\dots,x_{2n-2})$,
\begin{multline*}(c-d)G(a,b)S_{n-1}(b,\mathbf x)S_{n}(a,c,d,\mathbf x)+(d-b)G(a,c)S_{n-1}(c,\mathbf x)S_{n}(a,b,d,\mathbf x)\\
+(b-c)G(a,d)S_{n-1}(d,\mathbf x)S_{n}(a,b,c,\mathbf x)
=0.
 \end{multline*}
\end{lemma}

Combining Lemmas \ref{tp} and \ref{pl}, we obtain the following  recursion.

\begin{corollary}\label{grc} For $\mathbf x=(x_1,\dots,x_{2n-2})$,
\begin{multline*}
(b-d)(y-a)(y-c)S_{n-1}(a,\mathbf x)S_{n-1}(c,\mathbf x)S_{n+1}(y,a,b,c,d,\mathbf x)\\
\begin{split}&+(a-y)G(b,c)G(y,d)S_{n-1}(c,\mathbf x)S_n(a,c,d,\mathbf x)S_n(y,a,b,\mathbf x)\\
&+(y-c)G(a,d)G(y,b)S_{n-1}(a,\mathbf x)S_n(a,b,c,\mathbf x)S_n(y,c,d,\mathbf x)\\
&+(c-a)G(y,b)G(y,d)S_n(a,b,c,\mathbf x)S_n(a,c,d,\mathbf x)S_{n-1}(y,\mathbf x)
=0.
\end{split} \end{multline*}
\end{corollary}

\begin{proof}
The sum of the first two terms can be written
\begin{multline*}(y-a)S_{n-1}(c,\mathbf x)\big\{(b-d)(y-c)S_{n-1}(a,\mathbf x)S_{n+1}(y,a,b,c,d,\mathbf x)\\
-G(b,c)G(y,d)S_n(a,c,d,\mathbf x)S_n(y,a,b,\mathbf x)
\big\}. \end{multline*}
By Lemma \ref{tp}, with $(\mathbf x,a,b,c,d)\mapsto((\mathbf x,a),d,b,c,y)$, the factor in brackets equals 
$$-G(c,d)G(y,b)S_n(a,b,c,\mathbf x)S_n(y,a,d,\mathbf x). $$
The sum of the last two terms is 
\begin{multline*}G(y,b)S_n(a,b,c,\mathbf x)\big\{(y-c)G(a,d)S_{n-1}(a,\mathbf x)S_n(y,c,d,\mathbf x)\\
+(c-a)G(y,d)S_n(a,c,d,\mathbf x)S_{n-1}(y,\mathbf x)\big\}.\end{multline*}
By Lemma \ref{pl}, with $(\mathbf x,a,b,c,d)\mapsto(\mathbf x,d,a,y,c)$, the factor in brackets is
$$(y-a)G(c,d)S_{n-1}(c,\mathbf x)S_n(y,a,d,\mathbf x). $$
It is now clear that the sum of all four terms vanishes.
\end{proof}

Corollary \ref{grc} is particularly interesting in the case when 
$$\mathbf x=(\underbrace{a,\dots,a}_{n-1},\underbrace{b,\dots,b}_{n-1}),$$
$c=b$ and $d=a$. As we will see, one special case is the three-term recursion for the polynomials 
$P_n$.

\begin{corollary}\label{syr} For fixed $a$ and $b$, the polynomials
$$S_n(y)=S_n(y,\underbrace{a,\dots,a}_{n},\underbrace{b,\dots,b}_{n})$$
satisfy
\begin{multline*}
\begin{split}&(b-a)(y-a)(y-b)S_{n-1}(a)S_{n-1}(b)S_{n+1}(y)\\
&+\Big((a-y)G(b,b)G(y,a)S_{n-1}(b)S_n(a) \end{split}\\
\begin{split}&+(y-b)G(a,a)G(y,b)S_{n-1}(a)S_n(b)\Big)S_n(y)\\
&+(b-a)G(y,a)G(y,b)S_n(a)S_n(b)S_{n-1}(y)=0
.
 \end{split}\end{multline*}
\end{corollary}

We also consider a similar specialization of Lemma \ref{tp}.

\begin{corollary}\label{pc} For fixed $a$, $b$, and $z$, the
 polynomials
$$S_n(x,y)=S_n(x,y,z,\underbrace{a,\dots,a}_{n-1},\underbrace{b,\dots,b}_{n-1}) $$
satisfy 
\begin{multline*}(x-a)(y-b)S_{n-1}(a,b)S_{n+1}(x,y)\\
=G(a,y)G(b,x)S_{n}(a,b)S_{n}(x,y)-G(a,b)G(x,y)S_{n}(x,b)S_{n}(a,y).
\end{multline*}
\end{corollary}

Two cases of Corollary \ref{pc} are of special interest. 
The first one is   $x=b$ and $y=a$, when it reads
\begin{multline}\label{abr}(a-b)^2S_{n-1}(a,b)S_{n+1}(a,b)\\
=G(a,b)^2S_{n}(a,a)S_{n}(b,b)-G(a,a)G(b,b)S_{n}(a,b)^2.
\end{multline}
This will yield 
Theorem \ref{prt}. The second one is the limit case 
$x\rightarrow a$, $y\rightarrow b$, where we first divide through with
$(x-a)(y-b)$ and then use l'H\^{o}pital's rule on the right-hand side. Let
$$T_n(x,y)=S_n(\underbrace{x,\dots,x}_{n},\underbrace{y,\dots,y}_{n},z). $$
Since
$$\frac{\partial T_n}{\partial x}\Bigg|_{x=a,y=b}=n\frac{\partial S_n}{\partial x}\Bigg|_{x=a,y=b} $$
and similarly for  $y$, the result can be expressed in terms of  $T_n$ and its derivatives.

\begin{corollary}\label{tl} The polynomials $T_n=T_n(x,y)$ satisfy
$$T_{n-1}T_{n+1}=\left(\frac{\partial G}{\partial x}\frac{\partial G}{\partial y}-G\frac{\partial^2 G}{\partial x\partial y}\right)T_n^2+\frac 1{n^2}\,G^2\left(T_n\frac{\partial^2 T_n}{\partial x\partial y}-\frac{\partial T_n}{\partial x}\frac{\partial T_n}{\partial y}\right).$$
\end{corollary}

Equivalently, if we let
$$\tau_n=\frac{1}{\prod_{k=0}^{n-1}(k!)^2}\cdot\frac{T_n}{G^{n^2}}, $$
then $\tau_n$ satisfies   the two-dimensional Toda molecule equation
$$\frac{\partial ^2}{\partial x\partial y}\log\tau_n=\frac{\tau_{n-1}\tau_{n+1}}{\tau_n^2}. $$
We have recovered an instance of the well-known fact that this equation 
can be solved by determinants, which can be found already in Darboux's classic treatise \cite[\S 378]{d} (we owe this reference to Jacques Perk), see \cite{h,l} for more recent accounts.

The  one-dimensional Toda equation
$$\frac{\partial ^2}{\partial x^2}\log\tau_n=\frac{\tau_{n-1}\tau_{n+1}}{\tau_n^2} $$
plays an important role in the analysis of the six-vertex model with domain wall boundary conditions \cite{bf,bl1,bl2,bl3,kz,z}. This suggests that 
Corollary~\ref{tl} might be useful for proving Conjecture~\ref{tc}.

\subsection{Trigonometric limit}
\label{tes}

In the trigonometric limit $p\rightarrow 0$, \eqref{z} reduces to $\zeta=-2$.
We will show that the corresponding limit of the function $S_n$ is related to symplectic and odd orthogonal characters.

Let $\lambda=(\lambda_1,\dots,\lambda_n)$ be a partition. 
Recall \cite[\S 24.2]{fh} that the characters of the Lie algebras 
$\mathfrak{sp}(2n)$ and $\mathfrak{so}(2n+1)$ are given by
$$\chi_\lambda^{\mathfrak{sp}(2n)}(t_1,\dots,t_n)=\frac{\det_{1\leq i,j\leq n}\left(t_j^{-(\lambda_i+n-i+1)}-t_j^{\lambda_i+n-i+1}\right)}{\prod_{i=1}^nt_i^{-n}(1-t_i^2)\prod_{1\leq i<j\leq n}(t_j-t_i)(1-t_it_j)},$$
$$\chi_\lambda^{\mathfrak{so}(2n+1)}(t_1,\dots,t_n)=\frac{\det_{1\leq i,j\leq n}\left(t_j^{-(\lambda_i+n-i+\frac 12)}-t_j^{\lambda_i+n-i+\frac 12}\right)}{\prod_{i=1}^nt_i^{\frac 12-n}(1-t_i)\prod_{1\leq i<j\leq n}(t_j-t_i)(1-t_it_j)}.$$
In this notation, we have the following result.

\begin{theorem}\label{tst} Let the variables $t_i$, $u_i$, $x_i$ and $y_i$ be related by
$$x_i=-(1+t_i+t_i^{-1}),\qquad y_i=\frac 3{1+u_i+u_i^{-1}}.$$
 Then,
\begin{multline}\label{tsta}\lim_{\zeta\rightarrow -2}(\zeta+2)^{2mn-m(m-1)}S_n\left(x_1,\dots,x_{2n+1-2m},\frac{\zeta y_1}{\zeta+2},\dots,\frac{\zeta y_{2m}}{\zeta+2}\right)\\
\begin{split}&=2^{(2m+n)n-m(m-1)}3^{2mn-m(m-1)}
\prod_{i=1}^{2m}\frac{1}{(1+u_i+u_i^{-1})^{n}}\\
&\quad\times\chi_{(n-m,n-m,n-m-1,n-m-1,\dots,1,1,0)}^{\mathfrak{so}(4n-4m+3)}(t_1,\dots,t_{2n-2m+1})\\
&\quad\times\chi_{(m-1,m-1,\dots,1,1,0,0)}^{\mathfrak{sp}(4m)}(u_1,\dots,u_{2m})
,\end{split}
\end {multline} 
\begin{multline}\label{tstb}\lim_{\zeta\rightarrow -2}(\zeta+2)^{(2m+1)n-m^2}S_n\left(x_1,\dots,x_{2n-2m},\frac{\zeta y_1}{\zeta+2},\dots,\frac{\zeta y_{2m+1}}{\zeta+2}\right)\\
\begin{split}&=2^{(2m+n+1)n-m^2}3^{(2m+1)n-m^2}
\prod_{i=1}^{2m+1}\frac{1}{(1+u_i+u_i^{-1})^{n}}\\
&\quad\times\chi_{(n-m,n-m-1,n-m-1,\dots,1,1,0)}^{\mathfrak{so}(4n-4m+1)}(t_1,\dots,t_{2n-2m})\\
&\quad\times\chi_{(m,m-1,m-1,\dots,1,1,0,0)}^{\mathfrak{sp}(4m+2)}(u_1,\dots,u_{2m+1}).
\end{split}\end {multline} 
\end{theorem}

Note that the correspondence between $x_i$ and $t_i$
  is  the trigonometric limit of the uniformization \eqref{xi}:
$$\lim_{p\rightarrow 0}\xi(t_i)=-(1+t_i+t_i^{-1})=x_i. $$
For the variables  $y_i$ and $u_i$, we have instead
$$\left(\frac{\zeta \xi(u_i)}{\zeta+2}\right)^{-1}\Bigg|_{\zeta\mapsto \zeta^{-1}}=\frac{1+2\zeta}{\xi(u_i)}\rightarrow\frac 3{1+u_i+u_i^{-1}}=y_i, \qquad p\rightarrow 0,  $$
which is a natural limit  in view of Lemma \ref{il}.

We will use some  determinant identities due to  Okada \cite{o}. Let
$$W^n(x_1,\dots,x_n;a_1,\dots,a_n)=\det_{1\leq i,j\leq n}(x_i^{j-1}+a_ix_i^{n-j}).$$
(In \cite{o}, $W^n$ denotes the matrix rather than its determinant.)
Then, 
\begin{subequations}\label{oia}
\begin{multline}\label{oi2}\det_{1\leq i,j\leq n}\left(\frac{ W^2(x_i,y_j;a_i,b_j)}{(1-x_iy_j)(y_j-x_i)}\right)
=\frac{1}{\prod_{i,j=1}^n(1-x_iy_j)(y_j-x_i)}\\
\times W^{2n}(x_1,\dots,x_n,y_1,\dots,y_n;a_1,\dots,a_n,b_1,\dots,b_n), \end{multline}
\begin{multline}\label{oi}\det_{1\leq i,j\leq n}\left(\frac{W^3(x_i,y_j,z;a_i,b_j,c)}{(1-x_iy_j)(y_j-x_i)}\right)
=\frac{(1+c)^{n-1}}{\prod_{i,j=1}^n(1-x_iy_j)(y_j-x_i)}\\
\times W^{2n+1}(x_1,\dots,x_n,y_1,\dots,y_n,z;a_1,\dots,a_n,b_1,\dots,b_n,c). \end{multline}
\end{subequations}
Okada used these identities to enumerate certain symmetry classes of
 alternating sign matrices. 

The characters appearing in Theorem \ref{tst} can be obtained as special cases of Okada's determinants.

\begin{lemma}\label{cwl}
The following identities hold:
$$
\frac{W^{2n+1}(t_1^3,\dots,t_{2n+1}^3;-t_1,\dots,-t_{2n+1})}{\prod_{i=1}^{2n+1}t_i^{n}(1-t_i)\prod_{1\leq i<j\leq 2n+1}(t_j-t_i)(1-t_it_j)
}=
\chi_{(n,n,n-1,n-1,\dots,1,1,0)}^{\mathfrak{so}(4n+3)}(t_1,\dots,t_{2n+1}),
$$
$$
\frac{(-1)^n W^{2n}(t_1^3,\dots,t_{2n}^3;-t_1,\dots,-t_{2n})}{\prod_{i=1}^{2n}t_i^{n-1}(1-t_i^2)\prod_{1\leq i<j\leq 2n}(t_j-t_i)(1-t_it_j)
}=
\chi_{(n-1,n-1,\dots,1,1,0,0)}^{\mathfrak{sp}(4n)}(t_1,\dots,t_{2n}),
$$
$$
\frac{W^{2n+1}(t_1^3,\dots,t_{2n+1}^3;-t_1^2,\dots,-t_{2n+1}^2)}{\prod_{i=1}^{2n+1}t_i^{n}(1-t_i^2)\prod_{1\leq i<j\leq 2n+1}(t_j-t_i)(1-t_it_j)
}=
\chi_{(n,n-1,n-1,\dots,1,1,0,0)}^{\mathfrak{sp}(4n+2)}(t_1,\dots,t_{2n+1}),
$$
$$
\frac{(-1)^n W^{2n}(t_1^3,\dots,t_{2n}^3;-t_1^2,\dots,-t_{2n}^2)}{\prod_{i=1}^{2n}t_i^{n}(1-t_i)\prod_{1\leq i<j\leq 2n}(t_j-t_i)(1-t_it_j)
}=
\chi_{(n,n-1,n-1,\dots,1,1,0)}^{\mathfrak{so}(4n+1)}(t_1,\dots,t_{2n}).
$$
\end{lemma}

\begin{proof}
This is straight-forward, and we only provide some details for the first identity. By definition,
\begin{equation*}\begin{split}W^{2n+1}(t_1^3,\dots,t_{2n+1}^3;-t_1,\dots,-t_{2n+1})&=\det_{1\leq i,j\leq 2n+1}\left(t_i^{3(j-1)}-t_i^{3(2n+1-j)+1}\right)\\
&=\prod_{i=1}^{2n+1}t_i^{\frac{6n+1}2}\det_{1\leq i,j\leq 2n+1}\left(t_i^{\mu_j}-t_i^{-\mu_j}\right), \end{split}\end{equation*}
where $\mu_j=3n-3j+7/2$.
In order to make all $\mu_j$ positive, we multiply the last $n$ rows by $-1$. To sort the $\mu_j$ in descending order, each of those rows must be moved an odd number of positions upwards. Taken together, these two operations do not change the determinant, and we arrive at the desired result.
\end{proof}

We will need the following elementary identities.

\begin{lemma}\label{ell}
 In the notation above, the following identities
hold:
\begin{align*}x_2-x_1&=\frac{(1-t_1t_2)(t_2-t_1)}{t_1t_2}, \\
y_2-y_1&=\frac{3(1-u_1)(1-u_2)(1-u_1u_2)(u_2-u_1)}{(1-u_1^3)(1-u_2^3)}.\\
\lim_{\zeta\rightarrow -2}G(x_1,x_2)&=\frac{2(t_2^3-t_1^3)(1-t_1^3t_2^3)}{t_1^2t_2^2(t_2-t_1)(1-t_1t_2)}, \\
\lim_{\zeta\rightarrow -2}(\zeta+2)^2G\left(x_1,\frac{\zeta y_1}{\zeta+2}\right)&=\frac{72\,u_1^2(1-u_1)^2}{(1-u_1^3)^2}, \\
\lim_{\zeta\rightarrow -2}(\zeta+2)^2G\left(\frac{\zeta y_1}{\zeta+2},\frac{\zeta y_2}{\zeta+2}\right)&=\frac{72(1-u_1)^2(1-u_2)^2(u_2^3-u_1^3)(1-u_1^3u_2^3)}{(1-u_1^3)^2(1-u_2^3)^2(u_2-u_1)(1-u_1u_2)}, \\
\lim_{\zeta\rightarrow -2}F(x_1,x_2,x_3)&=\frac{2\, W^3(t_1^3,t_2^3,t_3^3;-t_1,-t_2,-t_3)}{\prod_{i=1}^3t_i(1-t_i)\prod_{1\leq i<j\leq 3}(t_j-t_i)(1-t_it_j)}, \\
\lim_{\zeta\rightarrow -2}(\zeta+2)F\left(x_1,x_2,\frac{\zeta y_1}{\zeta+2}\right)&=-\frac{12\,u_1(1-u_1)\, W^2(t_1^3,t_2^3;-t_1^2,-t_2^2)}{(1-u_1^3)t_1(1-t_1)t_2(1-t_2)(t_2-t_1)(1-t_1t_2)}, \end{align*}
\begin{multline*}\lim_{\zeta\rightarrow -2}(\zeta+2)^2F\left(x_1,\frac{\zeta y_1}{\zeta+2},\frac{\zeta y_2}{\zeta+2}\right)\\
=-\frac{72\,u_1u_2  W^2(u_1^3,u_2^3;-u_1,-u_2)}{(1+u_1)(1-u_1^3)(1+u_2)(1-u_2^3)(u_2-u_1)(1-u_1u_2)}, \end{multline*}
\begin{multline*}\lim_{\zeta\rightarrow -2}(\zeta+2)^2F\left(\frac{\zeta y_1}{\zeta+2},\frac{\zeta y_2}{\zeta+2},\frac{\zeta y_3}{\zeta+2}\right)\\
=\frac{72\, W^3(u_1^3,u_2^3,u_3^3;-u_1^2,-u_2^2,-u_3^2)}{\prod_{i=1}^3(1+u_i)(1-u_i^3)\prod_{1\leq i<j\leq 3}(u_j-u_i)(1-u_iu_j)}. \end{multline*}
\end{lemma}

\begin{proof}[Proof of \emph{Theorem \ref{tst}}]
To prove \eqref{tsta}, we relabel the variables, considering instead
\begin{multline}\label{rpl}\lim_{\zeta\rightarrow -2}(\zeta+2)^{2mn-m(m-1)}\\
\times S_n\left(\frac{\zeta x_1}{\zeta+2},\dots,\frac{\zeta x_m}{\zeta+2},x_{m+1},\dots,x_{n},\frac{\zeta y_1}{\zeta+2},\dots,\frac{\zeta y_{m}}{\zeta+2},y_{m+1},\dots,y_{n},z\right).\end{multline}
Applying \eqref{sn} leads to
 the  block determinant
$$\det\left(\begin{matrix}
\displaystyle\frac{F\left(\frac{\zeta x_i}{\zeta+2},\frac{\zeta y_j}{\zeta+2},z\right)}
{G\left(\frac{\zeta x_i}{\zeta+2},\frac{\zeta y_j}{\zeta+2}\right)}
&\displaystyle \frac{F\left(\frac{\zeta x_i}{\zeta+2},y_j,z\right)}
{G\left(\frac{\zeta x_i}{\zeta+2},y_j\right)}\\
\displaystyle\frac{F\left(x_i,\frac{\zeta y_j}{\zeta+2},z\right)}
{G\left(x_i,\frac{\zeta y_j}{\zeta+2}\right)}
&\displaystyle \frac{F\left(x_i,y_j,z\right)}
{G\left(x_i,y_j\right)}
\end{matrix}\right). $$
By Lemma \ref{ell}, the off-diagonal  blocks vanish in the limit $\zeta\rightarrow -2$, so the determinant splits as the product of the diagonal blocks. It follows that
\eqref{rpl} can be written
\begin{multline*}\lim_{\zeta\rightarrow -2}\prod_{\substack{1\leq i\leq m,\\ m+1\leq j\leq n}}
\frac{(\zeta+2)^4G\left(x_i,\frac{\zeta y_j}{\zeta+2}\right)G\left(\frac{\zeta x_i}{\zeta+2},y_j\right)}{(\zeta+2)^2\left(x_j-\frac{\zeta x_i}{\zeta+2}\right)\left(y_j-\frac{\zeta y_i}{\zeta+2}\right)}\\ 
\begin{split}&\times(\zeta+2)^{m(m+1)}S_m\left(\frac{\zeta x_1}{\zeta+2},\dots,\frac{\zeta x_m}{\zeta+2},\frac{\zeta y_1}{\zeta+2},\dots,\frac{\zeta y_{m}}{\zeta+2},z\right)\\
&\times S_{n-m}(x_{m+1},\dots,x_{n},y_{m+1},\dots,y_{n},z).
\end{split}\end{multline*}
Computing the prefactor using Lemma  \ref{ell}, we are reduced to proving the two cases $m=0$ and $m=n$ of \eqref{tsta}, that is,
$$\lim_{\zeta\rightarrow -2}S_n(x_1,\dots,x_{2n+1})=2^{n^2}\chi_{(n,n,n-1,n-1,\dots,1,1,0)}^{\mathfrak{so}(4n+3)}(t_1,\dots,t_{2n+1}),$$ 
\begin{multline*}\lim_{\zeta\rightarrow -2}(\zeta+2)^{n(n+1)}S_n\left(\frac{\zeta y_1}{\zeta+2},\dots,\frac{\zeta y_{2n}}{\zeta+2},z\right)\\
=2^{n(2n+1)}3^{n(n+1)}\prod_{i=1}^{2n}\frac{1}{(1+u_i+u_i^{-1})^n}\,\chi_{(n-1,n-1,\dots,1,1,0,0)}^{\mathfrak{sp}(4n)}(u_1,\dots,u_{2n}).\end{multline*}
This is now  straight-forward, using Lemma \ref{ell}, Okada's identities \eqref{oia} and finally Lemma \ref{cwl}.

The proof of  \eqref{tstb} is similar. We consider  the limit
\begin{multline*}\lim_{\zeta\rightarrow -2}(\zeta+2)^{(2m+1)n-m^2}\\
\times S_n\left(\frac{\zeta x_1}{\zeta+2},\dots,\frac{\zeta x_m}{\zeta+2},x_{m+1},\dots,x_{n},\frac{\zeta y_1}{\zeta+2},\dots,\frac{\zeta y_{m}}{\zeta+2},y_{m+1},\dots,y_{n},\frac{\zeta z}{\zeta+2}\right).\end{multline*}
Incorporating a factor $(\zeta+2)^m$ into the determinant, we  consider
$$\lim_{\zeta\rightarrow -2}\det\left(\begin{matrix}
\displaystyle\frac{(\zeta+2)F\left(\frac{\zeta x_i}{\zeta+2},\frac{\zeta y_j}{\zeta+2},\frac{\zeta z}{\zeta+2}\right)}
{G\left(\frac{\zeta x_i}{\zeta+2},\frac{\zeta y_j}{\zeta+2}\right)}
&\displaystyle \frac{(\zeta+2)F\left(\frac{\zeta x_i}{\zeta+2},y_j,\frac{\zeta z}{\zeta+2}\right)}
{G\left(\frac{\zeta x_i}{\zeta+2},y_j\right)}\\
\displaystyle\frac{F\left(x_i,\frac{\zeta y_j}{\zeta+2},\frac{\zeta z}{\zeta+2}\right)}
{G\left(x_i,\frac{\zeta y_j}{\zeta+2}\right)}
&\displaystyle \frac{F\left(x_i,y_j,\frac{\zeta z}{\zeta+2}\right)}
{G\left(x_i,y_j\right)}
\end{matrix}\right). $$
In the limit,  the upper right block  vanishes. Similarly as before, this reduces the proof of \eqref{tstb} to a straight-forward verification of 
the special cases
\begin{multline*}\lim_{\zeta\rightarrow -2}(\zeta+2)^nS_n\left(x_1,\dots,x_{2n},\frac{\zeta y_1}{\zeta+2}\right)\\
=\frac{2^{n(n+1)}3^n}{(1+u_1+u_1^{-1})^n}\chi_{(n,n-1,n-1,\dots,1,1,0)}^{\mathfrak{so}(4n+1)}(t_1,\dots,t_{2n}),\end{multline*} 
\begin{multline*}\lim_{\zeta\rightarrow -2}(\zeta+2)^{n(n+1)}S_n\left(\frac{\zeta y_1}{\zeta+2},\dots,\frac{\zeta y_{2n+1}}{\zeta+2}\right)\\
=2^{n(2n+1)}3^{n(n+1)}\prod_{i=1}^{2n+1}\frac{1}{(1+u_i+u_i^{-1})^n}\,\chi_{(n,n-1,n-1,\dots,1,1,0,0)}^{\mathfrak{sp}(4n+2)}(u_1,\dots,u_{2n+1}).\end{multline*}
\end{proof}

As a consequence of Theorem \ref{tst}, we have the following identities.

\begin{corollary}\label{acc}
One has
\begin{multline*}\lim_{\zeta\rightarrow -2}\left(1+\frac\zeta 2\right)^{n^2-\big[\frac{(n-1)^2}4\big]}
S_n\Big(\underbrace{2\zeta+1,\dots,2\zeta+1}_{n+1},\underbrace{\frac\zeta{\zeta+2},\dots,\frac\zeta{\zeta+2}}_{n}\Big)\\
=\begin{cases}
2^{n^2}3^{\frac{n^2}4}A_{n+1}, & n \text{ \emph{even}},\\
2^{n^2}3^{\frac{n^2-1}4}C_{n+1}, & n \text{ \emph{odd}},
\end{cases}
 \end{multline*}
\begin{multline*}\lim_{\zeta\rightarrow -2}\left(1+\frac\zeta 2\right)^{n(n+1)-\big[\frac{(n-1)^2}4\big]}
S_n\Big(\underbrace{2\zeta+1,\dots,2\zeta+1}_{n},\underbrace{\frac\zeta{\zeta+2},\dots,\frac\zeta{\zeta+2}}_{n+1}\Big)\\
=\begin{cases}
2^{n^2-1}3^{\frac{n(n-2)}4}C_{n+1}, & n \text{ \emph{even}},\\
2^{n^2-1}3^{\frac{(n-1)^2}4}A_{n+1}, & n \text{ \emph{odd}},
\end{cases}
\end{multline*}
where $A_n$ and $C_n$ are as in \eqref{an} and \eqref{cn}.
\end{corollary}

To verify these identities, we note that the
 limits correspond to special cases of Theorem \ref{tst} where $t_i=u_i=1$ for all $i$. Recall that, if $\mu_i=\lambda_i+n-i$, then \cite{fh}
$$\chi_\lambda^{\mathfrak{sp}(2n)}(1,\dots,1)=\frac{\prod_{1\leq i<j\leq n}(\mu_i-\mu_j)\prod_{1\leq i\leq j\leq n}(\mu_i+\mu_j+2)}{2^n 1! 3!\dotsm (2n-1)!}, $$
$$\chi_\lambda^{\mathfrak{so}(2n+1)}(1,\dots,1)=\frac{\prod_{1\leq i<j\leq n}(\mu_i-\mu_j)\prod_{1\leq i\leq j\leq n}(\mu_i+\mu_j+1)}{1! 3!\dotsm (2n-1)!}.$$
The relevant special cases can be simplified as
\begin{align*}\chi_{(n,n,n-1,n-1,\dots,1,1,0)}^{\mathfrak{so}(4n+3)}(1,\dots,1)&=3^{n^2}
\prod_{k=0}^{n}\frac{(6k+1)!}{(2n+2k+1)!}, \\
\chi_{(n-1,n-1,\dots,1,1,0,0)}^{\mathfrak{sp}(4n)}(1,\dots,1)&=3^{n(n-1)}
\prod_{k=0}^{n-1}\frac{(6k+4)!}{(2n+2k+2)!}, \\
\chi_{(n,n-1,n-1,\dots,1,1,0,0)}^{\mathfrak{sp}(4n+2)}(1,\dots,1)&=\frac{3^{n^2}}{2^{2n+1}}
\prod_{k=0}^{n}\frac{(6k+2)(6k)!}{(2n+2k+1)!}, \\
\chi_{(n,n-1,n-1,\dots,1,1,0)}^{\mathfrak{so}(4n+1)}(1,\dots,1)&={2^{2n}3^{n(n-1)}}\prod_{k=0}^{n-1}\frac{(6k+5)(6k+3)!}{(2n+2k+2)!}. \end{align*}
Using these identities, it is
 straight-forward to derive Corollary \ref{acc} from Theorem~\ref{tst}.

\section{The polynomials $P_n$ and $p_n$}
\label{pps}

In this Section, we show how specialization of the variables in $S_n$
lead to the  polynomials $P_n$ and $p_n$, and deduce a number of properties of the latter two systems.

\subsection{Elementary factors}
Our first task is to identify some elementary factors that appear 
when specializing  the polynomials $S_n$.
We will use the notation
$$\delta_n=\left[\frac {n^2}4\right];$$
cf.\ \eqref{de}. For later reference, we mention the identities
\begin{equation}\label{lac}\delta_{n+1}+\delta_{n-1}-2\delta_n=\chi(n\text{ odd}),\end{equation}
\begin{equation}\label{eli}\left[\frac{n(n+2)}2\right]=
n^2-2\delta_{n-1}=2\delta_n+n. \end{equation}

\begin{proposition}\label{efp}
There exist polynomials $P_n(x,\zeta)$, $p_n(\zeta)$ and $y_n(\zeta)$ such that
\begin{subequations}\label{ef}
 \begin{multline}\label{cpe}\left(1+\frac\zeta2\right)^{n^2}S_n\Big(x,\underbrace{2\zeta+1,\dots,2\zeta+1}_{n},\underbrace{\frac\zeta
{\zeta+2},\dots,\frac\zeta{\zeta+2}}_{n}\Big)\\
=(-1)^{\binom{n+1}2}\zeta^{n^2}(1+\zeta)^{n^2}(1+2\zeta)^{\delta_{n-1}}\left(1+\frac\zeta2\right)^{\delta_{n-1}}P_n(x,\zeta),
\end{multline}
\begin{multline}\label{pe}\left(1+\frac\zeta2\right)^{n^2}S_n\Big(\underbrace{2\zeta+1,\dots,2\zeta+1}_{n+1},\underbrace{\frac\zeta{\zeta+2},\dots,\frac\zeta{\zeta+2}}_{n}\Big)\\
=(-1)^{\binom{n+1}2}\zeta^{n^2}(1+\zeta)^{n^2}(1+2\zeta)^{\delta_n}\left(1+\frac\zeta 2\right)^{\delta_ {n-1}}p_n(\zeta), \end{multline}
and, for $n\geq 1$,
\begin{multline}\label{ye}\left(1+\frac\zeta2\right)^{n(n-1)}S_n\Big(\underbrace{2\zeta+1,\dots,2\zeta+1}_{n+2},\underbrace{\frac\zeta{\zeta+2},\dots,\frac\zeta{\zeta+2}}_{n-1}\Big)\\
=(-1)^{\binom{n-1}2}2^{\left[\frac{n+5}2\right]}\zeta^{n^2-1}(1+\zeta)^{n^2}(1+2\zeta)^{\delta_{n+1}}\left(1+\frac\zeta 2\right)^{\delta_ {n-2}}y_{n}(\zeta). \end{multline}
\end{subequations}
\end{proposition}

Applying  Lemma \ref{il} to  \eqref{cpe} gives, using also 
\eqref{eli},
\begin{equation}\label{pis}P_n(x,\zeta)=x^n\zeta^{\left[\frac{n(n+2)}{2}\right]}P_n(1/x,1/\zeta). \end{equation}
Similarly, applying  Lemma \ref{il} to
\eqref{pe}  gives
\begin{multline*}\left(1+\frac\zeta2\right)^{n(n+1)}S_n\Big(\underbrace{2\zeta+1,\dots,2\zeta+1}_{n},\underbrace{\frac\zeta{\zeta+2},\dots,\frac\zeta{\zeta+2}}_{n+1}\Big)\\
=(-1)^{\binom{n+1}2}2^{-\left[\frac {n+1}2\right]}\zeta^{n(n+1)}(1+\zeta)^{n^2}(1+2\zeta)^{\delta_{n-1}}\left(1+\frac\zeta 2\right)^{\delta_ {n}}\tilde p_n(\zeta), \end{multline*}
where
$$\tilde p_n(\zeta)=\zeta^{n(n+1)/2}p_n(1/\zeta). $$
We will see in Proposition \ref{psv} that $\deg p_n=n(n+1)/2$, so the notation agrees with \eqref{tn}.
Moreover, we clearly have
\begin{subequations}\label{ppc}
\begin{equation}\label{ppca}
P_n(2\zeta+1,\zeta)=(1+2\zeta)^{\left[\frac n2\right]}p_n(\zeta),\end{equation}
\begin{equation}\label{ppcb}P_n\left(\frac\zeta{\zeta+2},\zeta\right)=\zeta^n(\zeta+2)^{-\left[\frac {n+1}2\right]}\tilde p_n(\zeta). \end{equation}
\end{subequations}

To prove Proposition \ref{efp},
we first observe  that $S_n(x_1,\dots,x_{2n+1})$ 
is a polynomial of degree at most $n$ in each $x_i$. This is obvious for the variable $z$ in \eqref{sn}, and thus holds in general by symmetry. It follows 
that the  left-hand sides in \eqref{ef} are  polynomials  in $x$ and $\zeta$. 
We need to show that they vanish of appropriate degree at the points $\zeta=0,\,-1,\,-1/2$ and $-2$.

We first consider the function
\begin{equation}\label{sns}S_n(1+\zeta x_1,\dots,1+\zeta x_n,\zeta y_1,\dots,\zeta y_n,z). \end{equation}
Expressing it as in \eqref{sn}, we observe that
since  $\zeta^2\mid G(1+\zeta x,\zeta y)$, the prefactor $\prod_{i,j=1}^nG(1+\zeta x_i,y_j)$ is divisible by $\zeta^{2n^2}$. The product 
$$\prod_{1\leq i<j\leq n}\big((1+\zeta x_j)-(1+\zeta x_i)\big)(\zeta y_j-\zeta y_i\big)$$ 
contributes a factor $\zeta^{n(n-1)}$ to the denominator. 
Finally, since $\zeta\mid F(1+\zeta x,\zeta y,z)$,
each matrix entry has a single pole at $\zeta=0$, so the determinant has  a pole of degree at most $n$. In total, we conclude that \eqref{sns} is divisible by 
$\zeta^{2n^2-n(n-1)-n}=\zeta^{n^2}.$
This implies the same statement for the left-hand sides of 
 \eqref{cpe} and \eqref{pe}.

In the case of  \eqref{ye},
we consider instead 
\begin{equation}\label{ts}S_n\left(1+\zeta x_1,\dots,1+\zeta x_{n-2},x_{n-1},x_n,\zeta y_1,\dots,\zeta y_{n-2},y_{n-1},\frac\zeta{\zeta+2},z\right).
\end{equation}
Using that  $\zeta\mid G(x,\zeta y)$ and that  
\begin{equation}\label{gsp}G\left(x,\frac{\zeta}{\zeta+2}\right)=\frac{2\zeta^2(\zeta+1)^2}{(\zeta+2)^2}\end{equation}
for all $x$, obvious modifications of the previous argument shows that 
\eqref{ts}, and hence also  the left-hand side of  \eqref{ye}, is divisible by $\zeta^{n^2-1}$.

The case $\zeta=-1$ is  simple. Considering
\begin{equation}\label{sss}S_n(-1+(\zeta+1) x_1,\dots,-1+(\zeta+1) x_n,-1+(\zeta+1) y_1,\dots,-1+(\zeta+1) y_n,z),\end{equation}
and observing that
$$(\zeta+1)\mid F\big(-1+(\zeta+1)x,-1+(\zeta+1)y,z\big),$$ 
$$(\zeta+1)^2\mid G\big(-1+(\zeta+1)x,-1+(\zeta+1)y\big),$$
one checks that \eqref{sss} is divisible by $(\zeta+1)^{n^2}$. This implies the corresponding statement for the three left-hand sides of  \eqref{ef}.

The  vanishing conditions at $\zeta=-1/2$ and  $\zeta=-2$
follow from the following Lemma.

\begin{lemma}\label{zl}
As a polynomial in $\zeta$,
 $$S_n\big((2\zeta+1)x_1,\dots,(2\zeta+1)x_{k},x_{k+1},\dots,x_{2n+1}\big) $$
 is divisible by $(2\zeta+1)^{\delta_{k-1}}$, and
$$(\zeta+2)^{kn}S_n\left(\frac{\zeta x_1}{\zeta+2},\dots,\frac{\zeta x_{k}}{\zeta+2},x_{k+1},\dots,x_{2n+1}\right) $$
 is divisible by $(\zeta+2)^{\delta_{k-1}}$.
\end{lemma}

\begin{proof}
The second statement follows from 
Theorem \ref{tst} (note that we only need the very first step in the proof, not the  formulas involving Lie algebra characters). The first statement then follows using
 Lemma \ref{il}. 
\end{proof}

\subsection{Recursions}\label{prs}

We will now specialize the recursions for $S_n$ given in \S  \ref{rs}
to obtain recursions for $P_n$ and $p_n$.

If we let $a=2\zeta+1$ and $b=\zeta/(\zeta+2)$ in \eqref{abr}, using  
 \eqref{gsp} and the identities
$$G(2\zeta+1,x)=2(\zeta+1)^2x^2, $$
$$2\zeta+1-\frac{\zeta}{\zeta+2}=\frac{2(\zeta+1)^2}{\zeta+2}, $$
we obtain the following result after simplification.

\begin{proposition}\label{pyr}
The polynomials $p_n$ and $y_n$ satisfy
\begin{multline}\label{pyri}(\zeta+1)^2p_{n+1}(\zeta)p_{n-1}(\zeta)=\zeta^{n+1} \tilde p_n(\zeta)y_n(\zeta)\\
+(1+2\zeta)^{1+\chi(n \text{\emph{ even}})}\left(1+\frac \zeta 2\right)^{1+\chi(n \text{\emph{ odd}})}p_n(\zeta)^2.\end{multline}
Equivalently, \eqref{le} holds, where $Y=y_n$.
\end{proposition}

By induction, we obtain the following consequence, which is  a key result in our approach.

\begin{corollary}\label{pctc} The polynomial $p_n$ satisfies $p_n(0)=1$;
in particular, it does not vanish identically.
\end{corollary}

From this, another key result follows.

\begin{corollary}\label{doc}
The space $V_{n}(\underbrace{1,\dots,1}_{n},\underbrace{p,\dots,p}_{n})$
defined in \emph{\S \ref{stfs}} has dimension~$1$.
\end{corollary}

\begin{proof}
 By Lemma \ref{dcl} and \eqref{pss}, it is enough to show that the polynomial
$$S_n\Big(x,\underbrace{2\zeta+1,\dots,2\zeta+1}_{n},\underbrace{\frac\zeta
{\zeta+2},\dots,\frac\zeta{\zeta+2}}_{n}\Big)$$
does not vanish identically. By \eqref{pe}, this follows from Corollary \ref{pctc}.
\end{proof}

We remark that Corollary \ref{doc} can alternatively be deduced from
 Theorem \ref{tst}.

\begin{proposition}\label{prp}
The polynomials $P_n(x,\zeta)$ are determined 
from the starting values $P_0=1$ and $P_1=x+\zeta$
by the  recursion
\begin{equation}\label{ttr}A_n(x,\zeta)P_{n+1}(x,\zeta)=B_n(x,\zeta)P_n(x,\zeta)+C_n(x,\zeta)P_{n-1}(x,\zeta), \end{equation}
where
\begin{align*}A_n(x,\zeta)&=(1+2\zeta)^{\chi(n \text{\emph{ even}})}\left(1+\frac \zeta 2\right)^{\chi(n \text{\emph{ even}})}
(x-2\zeta-1)\left(x-\frac\zeta{\zeta+2}\right)\\
&\quad\times P_{n-1}(2\zeta+1,\zeta)P_{n-1}\left(\frac\zeta{\zeta+2},\zeta\right),\\
B_n(x,\zeta)&=\left(1+\frac \zeta 2\right)\left(
x^2(x-2\zeta-1)P_{n-1}\left(\frac\zeta{\zeta+2},\zeta\right)P_n(2\zeta+1,\zeta)
\right.\\
&\quad-\left.(2\zeta+1)^2\left(x-\frac\zeta{\zeta+2}\right)
P_{n-1}(2\zeta+1,\zeta)P_n\left(\frac\zeta{\zeta+2},\zeta\right)\right),\\
C_n(x,\zeta)&=(\zeta+1)^2x^2P_n(2\zeta+1,\zeta)P_n\left(\frac\zeta{\zeta+2},\zeta\right). \end{align*}
\end{proposition}

\begin{proof}
Letting  $a=2\zeta+1$ and $b=\zeta/(\zeta+2)$ in Corollary \ref{syr}
gives \eqref{ttr} after simplification. 
 The polynomials are determined by this recursion, provided that
 $A_n(x,\zeta)$ does not vanish identically. By \eqref{ppc}, this follows from Corollary \ref{pctc}.
\end{proof}

At this point, a remark on our approach is in order.
It is easy to see that if a system of polynomials satisfies \eqref{ttr}, then it can be used to construct  (via uniformization)
 functions $\Phi_{n}$ satisfying the properties of Proposition \ref{php}.
The reader may ask why we do not take \eqref{ttr} as our starting point, 
 thus avoiding introducing the multivariable polynomials $S_n$.
The  problem is that, to use \eqref{ttr} as a definition, we must know that $A_n(x,\zeta)$ does not vanish identically. Moreover, we need to
 know  that  $\Phi_{n}$ is uniquely determined by the 
 properties of Proposition~\ref{php}. Both these facts have been 
deduced from Corollary \ref{pctc}. If one could find a way to
deduce inductively from
 \eqref{ttr} that $P_n(2\zeta+1,\zeta)$ and 
 $P_n(\zeta/(\zeta+2),\zeta)$ never vanish identically, then the proofs of our main results could   be significantly shortened.
 
\subsection{Further properties}

Using
Proposition \ref{prp}, we can  obtain quite detailed information on the  polynomials $P_n$.

\begin{proposition}\label{pfp}
The polynomial $P_n(x,\zeta)$ can be written
\begin{equation}\label{pfd}P_n(x,\zeta)=\sum_{k=0}^n f^n_k(\zeta)\zeta^{n-k} x^k,
\end{equation}
where $f^n_k$ is a polynomial of degree $2\delta_n$, with leading coefficient
\begin{equation}\label{lcf}\frac 1{2^k}\binom{n+k}{n}. \end{equation}
In the notation \eqref{tn}, these polynomials satisfy
\begin{equation}\label{fs}f^n_k(\zeta)=\tilde f^n_{n-k}(\zeta). \end{equation}
For $k>[(n+1)/2]$,  $f^n_k(\zeta)$ is divisible by $(\zeta+2)^{k-\left[\frac {n+1}2\right]}$.
Moreover,
\begin{subequations}\label{ftt}
\begin{align}\label{flta}f^n_n(\zeta)&=\left(1+\frac \zeta 2\right)^{\left[\frac n2\right]}p_{n-1}(\zeta),  \\
\label{fltb} f_{n-1}^n(\zeta)&=\frac 12\left(1+\frac \zeta 2\right)^{\left[\frac n2\right]-1}(\zeta+n+1)\,p_{n-1}(\zeta). \end{align}
\end{subequations}
\end{proposition}

Further results follow  by
applying \eqref{fs} to the other statements.
In particular, from \eqref{flta} we obtain
\begin{equation}\label{pept}P_n(0,\zeta)=\zeta^n\left(\frac{2\zeta+1}2\right)^{\left[\frac {n}2\right]}\tilde p_{n-1}(\zeta). \end{equation}

\begin{proof} 
From the proof of Proposition \ref{efp}, and also from 
Proposition \ref{prp}, it is clear that $P_n(x,\zeta)$ is of degree at most $n$ in $x$, so we can write it as in  \eqref{pfd}, where  a priori $f_k^n(\zeta)\zeta^{n-k}$ is a polynomial. 

To show that $f_k^n$ is a polynomial, we consider the function
$$F_n(x) =
\lim_{\zeta\rightarrow 0}\frac {P_n(\zeta x,\zeta)}{\zeta^n}
=\sum_{k=0}^n\lim_{\zeta\rightarrow 0}f_k^n(\zeta)x^k. $$
We will prove by induction that the limit exists and equals
\begin{equation}\label{fn}F_n(x)=\sum_{k=0}^n\binom{n+k}{k}\frac{x^{n-k}}{2^k}
=
x^n\,{}_2F_1\left(\begin{matrix}-n,\,n+1\\-n\end{matrix};\frac 1{2x}\right), \end{equation}
in standard hypergeometric notation.
 To this end,
we divide \eqref{ttr} by $\zeta^{2n}$ and then let $\zeta\rightarrow 0$. We need that,  by \eqref{ppca} and Corollary \ref{pctc},
$$P_n(1,0)=p_n(0)=1,$$
and we also write
 $$\lim_{\zeta\rightarrow 0}\frac{P_n(\zeta/(\zeta+2),\zeta)}{\zeta^n}=F_n(1/2).$$
Then, the resulting identity simplifies to
\begin{multline*}\left(x-\frac 12\right)F_{n-1}(1/2)F_{n+1}(x)\\
=\left(x^2F_{n-1}(1/2)+\left(x-\frac 12\right)F_n(1/2)\right)F_n(x)
-x^2F_n(1/2)F_{n-1}(x).
 \end{multline*}
We must show that  this recursion is solved by \eqref{fn}. 
Equivalently, since
$$F_n(1/2)=\frac 1{2^n}\,{}_2F_1\left(\begin{matrix}-n,\,n+1\\-n\end{matrix};1\right)=\frac 1{2^{n+1}}\binom{2n+2}{n+1}, $$
we need to check that
$$\left(x-\frac 12\right)F_{n+1}(x)
=\left(x^2+\frac{2n+1}{n+1}\left(x-\frac 12\right)\right)F_n(x)
-\frac{2n+1}{n+1}\,x^2F_{n-1}(x),
$$
which can be done by a straight-forward computation.

The equation \eqref{fn} shows that $f_k^n$ is a polynomial, with
\begin{equation}\label{fev}f^n_{n-k}(0)=\frac 1{2^k}\binom{n+k}{n}. \end{equation}
 It follows from \eqref{pis}, using also \eqref{eli}, that 
$$f^n_{k}(\zeta)=\zeta^{2\delta_n}f^n_{n-k}(1/\zeta). $$
Together with \eqref{fev}, this
 proves that $f^n_k$ has degree   $2\delta_n$, as well as \eqref{lcf} and \eqref{fs}. 

The statement on divisibility by powers of $\zeta+2$
is equivalent to saying that
$$\lim_{\zeta\rightarrow -2}(\zeta+2)^{\left[\frac{n+1}2\right]}P_n\left(\frac{\zeta x}{\zeta+2},\zeta\right) $$
exists finitely. This is a consequence of \eqref{cpe} and 
Theorem \ref{tst}.

To prove \eqref{ftt}, we  pick out the coefficients of $x^{n+3}$ and $x^{n+2}$ on both sides of \eqref{ttr}.  On the right, neither $C_n$ nor the second of the two terms in $B_n$ contribute. Using 
\eqref{ppc} and simplifying, we find that
\begin{multline*}p_{n-1}(\zeta)\left(x^2-\left(\frac\zeta{\zeta+2}+2\zeta+1\right)x\right)
\left(f_{n+1}^{n+1}(\zeta)x^{n+1}+\zeta f_{n}^{n+1}(\zeta)x^{n}\right)\\
=\left(1+\frac \zeta 2\right)^{\chi(n \text{{ odd}})}p_n(\zeta)x^2(x-2\zeta-1)
\left(f_{n}^{n}(\zeta)x^{n}+\zeta f_{n-1}^{n}(\zeta)x^{n-1}\right)+\cdots, 
\end{multline*}
where the ellipsis denotes terms of order at most $n+1$ in $x$.
Picking out the top coefficient gives the recursion
$$p_{n-1}(\zeta)f_{n+1}^{n+1}(\zeta)=\left(1+\frac \zeta 2\right)^{\chi(n \text{{ odd}})}p_n(\zeta)f_{n}^{n}(\zeta), $$ 
which is  solved by \eqref{flta}.
Picking out the next coefficient then yields
$$p_{n-1}(\zeta)f_{n}^{n+1}(\zeta)=\left(1+\frac \zeta 2\right)^{\chi(n \text{{ odd}})}p_n(\zeta)f_{n-1}^{n}(\zeta)+\frac 12\left(1+\frac \zeta 2\right)^{\big[\frac{n-1}2\big]}p_{n-1}(\zeta)p_n(\zeta), $$ 
which is solved by \eqref{fltb}.
\end{proof}

For the formulation of the next result, it is convenient to introduce the polynomials
$$\phi_n(x)={}_2F_1\left(\begin{matrix}-n/2,\,-(n-1)/2\\n+3/2\end{matrix};x\right), $$
$$\psi_n(x)={}_3F_2\left(\begin{matrix}-n/2,\,-(n+1)/2,\,(n+5)/4\\n+3/2,\,(n+1)/4\end{matrix};x\right) $$
of degree $[n/2]$, $[(n+1)/2]$, respectively.

We will need the identities
\begin{equation}\label{ppa}(3x+1)\psi_n(x)=\frac{3(3n+1)(3n+4)}{(2n+1)(2n+3)}\,x\phi_{n+1}(x)+(x-1)^2\phi_{n-1}(x),\end{equation}
\begin{equation}\label{ppb}(3x+1)^2\phi_n(x)=\frac{3(3n+2)(3n+5)}{(2n+1)(2n+3)}\,x\psi_{n+1}(x)+(x-1)^2\psi_{n-1}(x),\end{equation}
\begin{equation}\label{phr}\frac{3(3n+2)(3n+4)}{(2n+1)(2n+3)}\,x\phi_{n+1}(x)
=(9x-1)\phi_n(x)+(x-1)^2\phi_{n-1}(x),\end{equation}
which are straight-forward to verify.

\begin{proposition} \label{pvp}
For special values of $\zeta$, 
$P_n(x,\zeta)$ may be expressed as 
\begin{align}
\label{peva}P_n(x,0)&=x^n,\\
\label{pevb}P_n(x,-1)&=(-1)^{\chi(n\equiv 2\, \operatorname{mod}\, 4)} 2^{\delta_{n-1}}(x-1)^n, \\
\label{pevc}P_n(x,1)&=2^{\delta_ {n-1}-n}3^{\left[\frac n2\right]}A_{n+1}(1+x)^n\phi_n\left(\frac{(3x-1)(x-3)}{3(1+x)^2}\right),\\
\label{pevd}P_n(x,-2)&=\begin{cases}\displaystyle(-3/4)^{\frac n2}A_{n+1}(1-x)^{\frac n2}\phi_n\left(\frac{x+3}{3(1-x)}\right),& n \text{ even},\\[5mm]
\displaystyle -2^{-n-1}3^{\frac{n-1}2}C_{n+1}(1-x)^{\frac{n+1}2}\psi_n\left(\frac{x+3}{3(1-x)}\right),& n \text{ odd}.\end{cases}
\end{align}
\end{proposition}

Applying \eqref{pis} to \eqref{pevd} one may also evaluate $P_n(x,-1/2)$.

\begin{proof}
The identity \eqref{peva} follows from \eqref{pfd} and \eqref{fev}.

If we let $\zeta=-1$ in \eqref{ttr}, we obtain after simplification
$$(-1)^{\chi(n \text{ even})}2^{\chi(n \text{ odd})}P_{n-1}(-1,-1)P_{n+1}(x,-1)
=(x-1)P_{n}(-1,-1)P_{n}(x,-1). $$
It is easy to check that this is solved by \eqref{pevb}, using
 \eqref{lac} to simplify the exponent of $2$. 

Similarly,  if we 
let $\zeta=-1$ in \eqref{ttr} and
substitute
\eqref{pevc}, we are reduced to the recursion
 \eqref{phr}. For this computation, it is useful to note that
$$\frac{A_nA_{n+2}}{A_{n+1}^2}=\frac{3(3n+2)(3n+4)}{4(2n+1)(2n+3)} $$
and that
$$y=\frac{(3x-1)(x-3)}{3(1+x)^2}\quad \Longrightarrow\quad
y-1=-\frac{16x}{3(x+1)^2}, \quad 9y-1= \frac{8(x^2-4x+1)}{(x+1)^2}.
 $$

To prove \eqref{pevd}, 
we  multiply \eqref{ttr} by $(1+\zeta/2)^{[\frac{n+1}2]}$ and then let $\zeta\rightarrow -2$. This gives
\begin{multline}\label{dr}
(-3)^{\chi(n \text{ even})}(x+3)X_{n-1}P_{n-1}(-3,-2)P_{n+1}(x,-2)\\
=X_n\left(-9P_{n-1}(-3,-2)P_n(x,-2)+x^2P_n(-3,-2)P_{n-1}(x,-2)\right),
\end{multline}
where
$$X_n=\lim_{\zeta\rightarrow -2}\left(1+\frac\zeta 2\right)^{\left[\frac{n+1}2\right]}P_n\left(\frac\zeta{\zeta+2},\zeta\right).$$
It follows from \eqref{cpe} and Corollary \ref{acc}  that
\begin{equation}\label{x}X_n=\begin{cases} (-1)^{\frac n2}2^{-1}C_{n+1}, & n \text{ even},\\
-2^{-1}A_{n+1}, & n \text{ odd}.\end{cases}
 \end{equation}
It is now straight-forward to check that  substituting
\eqref{pevd} in \eqref{dr} yields \eqref{ppa} when $n$ is odd and \eqref{ppb} when $n$ is even. For this computation, one needs  the identities
$$\frac{A_{n+2}C_n}{A_{n+1}C_{n+1}}=\frac{3(3n+1)(3n+4)}{4(2n+1)(2n+3)}, $$
$$\frac{A_nC_{n+2}}{A_{n+1}C_{n+1}}=\frac{3(3n+2)(3n+5)}{4(2n+1)(2n+3)}, $$
and also that, 
$$z=\frac{x+3}{3(1-x)}\quad\Longrightarrow\quad z-1=\frac{4x}{3(1-x)},\quad 3z+1=\frac 4{1-x}.$$
\end{proof}

\begin{proposition}\label{psv}
The polynomial $p_n$ has degree $n(n+1)/2$ and leading coefficient 
$2^{-\left[\frac {n+2}2\right]}\binom{2n+2}{n+1}.$
Moreover,  $p_n$ assumes the  special values
\begin{align}
\label{seva} p_n(-1)&=(-1)^{\chi(n\equiv 1\, \operatorname{mod}\, 4)} 2^{\delta_{n+1}}\\
\label{sevb} p_n(1)&=2^{\delta_{n+1}}A_{n+1}\\
\label{sevc} p_n(-2)&= \begin{cases}
A_{n+1}, & n \text{ \emph{even}},\\
(-1)^{\frac{n+1}2}C_{n+1}, & n \text{ \emph{odd}},
\end{cases} \\
\label{sevd}p_n(-1/2)&=\begin{cases}2^{-\frac 12(n^2+2n+2)}C_{n+1}, & n \text{ \emph{even}},\\
(-1)^\frac{n+1}2 2^{-\frac 12(n+1)^2}A_{n+1}, & n \text{ \emph{odd}}.\end{cases} 
\end{align}
\end{proposition}

\begin{proof} 
The first statement follows \eqref{lcf} and \eqref{flta}.
The evaluations \eqref{seva}, \eqref{sevb} and \eqref{sevc} follow from 
Proposition \ref{pvp}, using \eqref{ppca}. Finally, 
by \eqref{ppcb},
\eqref{sevd}
is equivalent to \eqref{x}.
\end{proof}

Writing $p_n(\zeta)=\sum_{k=0}^{n(n+1)/2}a_k\zeta^k$, we have simple formulas for $a_0$ (Corollary \ref{pctc}) and  $a_{n(n+1)/2}$.
It is easy to deduce from Proposition \ref{pyr} that
$$a_1=\begin{cases}\displaystyle \frac{7n(n+2)}8, & n \text{ {even}},\\[5mm]
\displaystyle \frac{7(n^2+2n+3)}8, & n \text{ {odd}}.
\end{cases} $$
 It also seems that
$$a_{\frac{n(n+1)}2-1}=\begin{cases}
\displaystyle \frac{n^2(7n+10)}{2^{(n+8)/2}(n+2)}\binom{2n+2}{n+1}, & n \text{{ even}},\\[5mm]
\displaystyle \frac{(n+1)(7n^2+3n-6)}{2^{(n+7)/2}(n+2)}\binom{2n+2}{n+1},& n \text{ {odd}}.
\end{cases}
$$
One can probably prove this  using Proposition \ref{prp}, though we have not worked out the details.  It does not seem that the other coefficients admit 
such simple expressions. 

Finally, we use the first part of Proposition \ref{psv}
to pick out the leading term on both sides of \eqref{pyri}.
This allows us to compute the leading term in the polynomials 
 $y_n$. In particular, we may conclude that 
$\deg Y=\deg A$ in \eqref{le}.

\begin{corollary}\label{ydc}
The polynomial $y_n$ has degree $\frac{n(n-1)}2+2 $ and leading coefficient
$$\frac {(2n+2)!(2n)!}{2^{n+\chi(n\text{\emph{ odd}})}(n+2)!(n+1)!^2n!}. $$
\end{corollary}

\subsection{Zeroes}
\label{zs}

In this Section we  prove Propositions \ref{zp} and \ref{zcp}.

\begin{proof}[Proof of \emph{Proposition \ref{zp}}]
We proceed by induction on $n$, treating the cases of odd and even $n$ separately. Assume that the statement holds for the zeroes of $p_{2n}$ and $p_{2n+1}$.
Let  $a_1,\dots,a_n$ denote the real zeroes of  $p_{2n}$, 
$b_1,\dots,b_{n+1}$  the real zeroes of  $\tilde p_{2n+1}$, and 
 $c_1,\dots,c_{n+1}$  the real zeroes of  $p_{2n+2}$.
We  assume that
\begin{equation}\label{abz}b_1<a_1<b_2<a_2<\dots<a_n<b_{n+1}<-2, \end{equation}
and need to prove
$$b_1<c_1<b_2<c_2<\dots<b_{n+1}<c_{n+1}<-2. $$
Since, by Proposition \ref{psv}, $p_{2n}$ has positive leading coefficient, it follows from \eqref{abz} 
that
$(-1)^{n+i+1}p_{2n}(b_i)>0. $
Moreover, by Proposition \ref{pyr}, if $\tilde p_{2n+1}(\zeta)=0$, then
$$(\zeta+1)^2p_{2n}(\zeta)p_{2n+2}(\zeta)=(1+2\zeta)\left(1+\frac\zeta 2\right)^2p_{2n+1}(\zeta)^2. $$
We conclude that
$(-1)^{n+i}p_{2n+2}(b_i)>0. $
Thus,  $p_{2n+2}$ has one zero between each consecutive 
pair of points $b_i,\,b_{i+1}$. Since, by \eqref{sevc},
$p_{2n+2}(-2)>0$, the remaining zero lies between $b_{n+1}$ and $-2$. 
The induction step for odd $n$ is similar, and we do not give the details.
\end{proof}

\begin{lemma}\label{sl}
Fix $\zeta$ in the interval $-2<\zeta<-1/2$ and
assume \emph{Conjecture~\ref{nvc}}. Then,  $p_n(\zeta)$ is negative if $n\equiv 1\ \operatorname{mod}\ 4$ and positive else. Further, $\tilde p_n(\zeta)$ is negative if $n\equiv 2\ \operatorname{mod}\ 4$ and positive else.
Moreover, $P_n(0,\zeta)$ is positive if  $n\equiv 0\ \operatorname{mod}\ 4$ and negative else. Finally, the leading coefficient of $P_n(x,\zeta)$ is 
negative if  $n\equiv 2\ \operatorname{mod}\ 4$ and positive else.
\end{lemma}

\begin{proof} Since we assume that $p_n$ does not vanish in the interval 
$-2<\zeta<-1/2$, $p_n(\zeta)$ has the same sign as $p_n(-1)$, so the first statement is obtained from \eqref{seva}. The sign of $\tilde p_n(\zeta)$  is determined as  a consequence. The remaining statements follow using \eqref{flta} and \eqref{pept}.
\end{proof}

\begin{proof}[Proof of \emph{Proposition \ref{zcp}}]
  For fixed $\zeta$, with $-2<\zeta<-1/2$ and $\zeta\neq -1$, let
$a_1,\dots,a_{n-1}$ denote the  zeroes of  $P_{n-1}$, 
$b_1,\dots,b_{n}$  the  zeroes of  $P_n$, and 
 $c_1,\dots,c_{n+1}$  the  zeroes of  $P_{n+1}$.
By induction, we  assume
\begin{equation}\label{xabz}0<b_1<a_1<b_2<a_2<\dots<a_{n-1}<b_{n}, \end{equation}
and  prove
$$0<c_1<b_1<c_2<\dots<b_{n}<c_{n+1}. $$

Using Lemma \ref{sl}, we deduce from \eqref{xabz} that
$$(-1)^{i+\chi(n\equiv 1\,\operatorname {mod}\,4)}P_{n-1}(b_i,\zeta)>0. $$
It follows from \eqref{ttr} that, if $P_n(x,\zeta)=0$, then
$$
P_{n+1}(x,\zeta)=\frac{2^{\chi(n \text{ even})}\zeta(\zeta+1)^2x^2\tilde p_{n}(\zeta)p_{n}(\zeta)}{(x-2\zeta-1)((\zeta+2)x-\zeta)\tilde p_{n-1}(\zeta)p_{n-1}(\zeta)}\,P_{n-1}(x,\zeta). $$
Again using Lemma \ref{sl},
$$(-1)^n\frac{\tilde p_{n}(\zeta)p_{n}(\zeta)}{\tilde p_{n-1}(\zeta)p_{n-1}(\zeta)}>0. $$
Combining these facts we find that
$$(-1)^{\chi(n\not\equiv 3\,\operatorname {mod}\,4)+i}P_{n+1}(b_i,\zeta)>0. $$
Thus, $P_{n+1}$ has a zero between any two consecutive $b_i$. Moreover, yet again using  Lemma \ref{sl},
$P_{n+1}(0,\zeta)P_{n+1}(b_1,\zeta)<0$, so there is an additional zero between $0$ and $b_1$. Finally,  $P_{n+1}(b_n,\zeta)$ and the leading coefficient of $P_{n+1}(x,\zeta)$ have opposite signs, so the final zero is to the right of $b_n$.
\end{proof}

\section{Return to three-colour model}
\label{rtcs}

\subsection{Uniformization of $\Phi_n$}
After the long detour in  \S \ref{sfs} and \S \ref{pps}, we are now ready to apply our results to the three-colour model. First, we express the function $\Phi_n$, introduced in \S \ref{phs},  in terms of the polynomials $P_{n-1}$. The following identities will be useful.

 \begin{lemma}\label{ztl}
With $\zeta$ as in \eqref{z},
\begin{align}\nonumber 2\zeta+1&=\frac{\tha(-p\om,\om;p^2)^2}{\tha(-\om,p\om;p^2)^2}, \\
\nonumber \zeta+2&=p\frac{\tha(-1,-\om;p^2)\tha(\om;p^2)^2}{\tha(-p,-p\om ;p^2)\tha(p\om;p^2)^2}, \\
\label{ztc}\zeta+1&=-\frac{\tha(p,-p\om;p^2)}{\tha(-p,p\om;p^2)}, \\
\label{ztd}\zeta-1&=\frac{\tha(p,p\om;p^2)\tha(\om;p^2)^2}{\tha(-p,-p\om;p^2)\tha(-\om;p^2)^2}. \end{align}
\end{lemma}

\begin{proof}
All four identities  follow from
 Lemma \ref{xvl}, in the last two cases by writing
 $\zeta+1=\xi(1)-\xi(-\om)$, $\zeta-1=\xi(-\om)-\xi(-1)$ and using
 \eqref{xid}.
\end{proof}

We note in passing that  Lemma \ref{ztl} can be used to prove
\eqref{omp}. It is enough to check that, after elementary simplification,
$$\frac{(2\zeta+1)(\zeta+2)(\zeta-1)^4}{\zeta(\zeta+1)^4}=p\omega\frac{\tha(\om;p^2)^{12}\tha(p\om;p^2)^4}{\tha(-\om,-p\om;p^2)^8}=3^6p\frac{(p^3;p^3)_\infty^{12}}{(p;p)_\infty^{12}}; $$
this implies \eqref{omp} in view of \eqref{tr}.

\begin{lemma}\label{ecl} One has
$$\zeta(\zeta+1)^4=2\eta^3, $$
where
$$\eta=-\frac{\tha(p;p^2)\tha(-p\om;p^2)^2}{\tha(p\om;p^2)\tha(-p;p^2)^2}. $$
\end{lemma}

\begin{proof}
This follows easily from \eqref{ztc}, using   \eqref{tq}.
\end{proof}

\begin{proposition}\label{phut} In the notation above,
\begin{equation}\label{phu}\Phi_n(t)=B_{n}\theta(t,\pm pt;p^2)\theta(t^{\pm},pt^{\pm},p\omega t^{\pm};p^2)^{n-1}P_{n-1}(\xi(t),\zeta), \end{equation}
where 
$$B_n=\begin{cases}\displaystyle
-\frac{p\om^{1-n^2-n}\tha(-\om;p^2)}{\tha(p,-p;p^2)\tha(p\om;p^2)^{2n-2}\eta^{\frac{n^2}4}(2\zeta+1)^{\frac {n-2}2}}, & n \text{\emph{ even}},\\[5mm]
\displaystyle\frac{\om^{-n^2-n}}{\tha(p;p^2)\tha(p\om;p^2)^{2n-2}\eta^{\frac{n^2-1}4}(2\zeta+1)^{\frac {n-1}2}}, & n \text{\emph{ odd}}.
\end{cases}
$$
\end{proposition}

\begin{proof}
Both sides of \eqref{phu} belong to the space
 $$V_{n-1}(\underbrace{1,\dots,1}_{n-1},\underbrace{p,\dots,p}_{n-1}),$$
which is one-dimensional by Corollary \ref{doc}.
Thus, \eqref{phu} holds for some constant $B_n$, which we compute using 
part (iv) of  Proposition \ref{php}. 
We have
$$\Phi_n(\om)=(-1)^{n+1}B_n\om^{1-n}\tha(-p\om;p^2)\tha(p;p^2)^{n-1}\tha(\om;p^2)^{2n-1}\tha(p\om;p^2)^{3n-2}P_{n-1}(0,\zeta), $$
$$\frac{\Phi_{n-1}(t)}{\tha(t;p)^{2n-3}}\Bigg|_{t=p}=B_{n-1}p^{2-n}\om^{-n-1}
\tha(-1;p^2)\tha(\om;p^2)^{2n-4}P_{n-2}\left(\frac\zeta{\zeta+2},\zeta\right). $$
By \eqref{ppcb} and \eqref{pept},
$$\frac{P_{n-1}(0,\zeta)}{P_{n-2}(\zeta/(\zeta+2),\zeta)}=\zeta\left(\frac{(\zeta+2)(2\zeta+1)}2\right)^{\left[\frac{n-1}2\right]}. $$
We conclude that
\begin{equation*}\begin{split} \frac{B_n}{B_{n-1}}&=(-1)^np^{3\left[\frac{n-2}2\right ]+3-n}\om^2\frac{\tha(-1;p^2)\tha(\om;p^2)^{2n-4}}{\tha(-p\om;p^2)\tha(p,p\om;p^2)^{n-1}}\frac{1}{\zeta}\left(\frac2{(\zeta+2)(2\zeta+1)}\right)^{\left[\frac{n-1}2\right]}\\
&=\begin{cases}\displaystyle
-\frac{p\om^{n+1}\tha(-\om;p^2)}{\tha(-p;p^2)\tha(p\om;p^2)^2\eta^{\frac n 2 }}, & n \text{ even},\\[5mm]
\displaystyle -\frac{p^{-1}\om^{n-1}\tha(-p;p^2)}{\tha(-\om;p^2)\tha(p\om;p^2)^2\eta^{\frac {n-1} 2 }(2\zeta+1)}, & n \text{ odd},
\end{cases}
\end{split}\end{equation*}
where we used Lemma \ref{ztl} to simplify the expression.
It is  clear that this recursion can be solved as indicated;
 the starting value
$B_1={\om}/{\tha(p;p^2)} $
follows from part (v) of  Proposition \ref{php}.
\end{proof}

\begin{corollary}\label{tcpca}
When $t_i$ are as in \eqref{ti},  
\begin{multline*}
Z_n^{3C}(t_0,t_1,t_2)=\frac{\tha(\la\om^2,\la\om^{n+1};p)^2}{\tha(p;p^2)\eta^{\frac{n^2}4}\tha(\la^3;p^3)^{n^2+2n+3}}\\
\times\left(\zeta^{\frac n2}\tilde p_{n-1}(\zeta)\tha(-p\om,-\om^n\la^2;p^2)
-\om^{1-n}\la p_{n-1}(\zeta)\tha(-\om,-p\om^n\la^2;p^2)
\right)
\end{multline*}
for $n$ even, while for odd $n$
\begin{multline*}
Z_n^{3C}(t_0,t_1,t_2)=\frac{\tha(\la\om^2,\la\om^{n+1};p)^2}{\tha(p;p^2)\eta^{\frac{n^2-1}4}\tha(\la^3;p^3)^{n^2+2n+3}}\\
\times\left(p_{n-1}(\zeta)\tha(-p,-\om^n\la^2;p^2)
-\om^{-n}\la\zeta^{\frac{n-1}2}\tilde p_{n-1}(\zeta)\tha(-1,-p\om^n\la^2;p^2)
\right).
\end{multline*}
\end{corollary}

\begin{proof}
In Corollary \ref{tcpc}, express $\Phi_n$ as in Theorem \ref{phu}. Using  \eqref{ppc}, we see that
$$\lim_{t\rightarrow 1}\frac{\Phi_n(t)}{\theta(t;p)^{2n-1}}
= D_{n} p_{n-1}(\zeta),$$
$$\lim_{t\rightarrow p}\frac{\Phi_n(t)}{\theta(t;p)^{2n-1}}
= E_{n} \tilde p_{n-1}(\zeta),$$
where 
$$D_n=(-1)^{n+1}\tha(-p;p^2)\tha(p\omega;p^2)^{2n-2}(2\zeta+1)^{\left[\frac{n-1}2\right]}B_n, $$
$$E_n=p^{1-n}\om^{1-n} \tha(-1;p^2)\tha(\omega;p^2)^{2n-2}\zeta^{n-1}(\zeta+2)^{-\left[\frac{n}2\right]}B_n.$$
To complete the proof, we use Lemma \ref{ztl} to write these
expressions as 
$$D_n=\begin{cases}\displaystyle p\om^{1-n-n^2}\frac{\tha(-\om;p^2)}{\tha(p;p^2)\eta^{\frac{n^2}4}}, & n \text{{ even}}\\[5mm]
\displaystyle \om^{-n-n^2}\frac{\tha(-p;p^2)}{\tha(p;p^2)\eta^{\frac{n^2-1}4}}, & n \text{{ odd}},\end{cases}$$
$$E_n=\begin{cases}\displaystyle -p^{2-\frac{3n}2}\om^{-n(n+1)}\frac{\tha(-p\om;p^2)\zeta^{\frac n2}}{\tha(p;p^2)\eta^{\frac{n^2}4}}, & n \text{{ even}}\\[5mm]
\displaystyle p^{\frac 32(1-n)}\om^{-n(n+1)}\frac{\tha(-1;p^2)\zeta^{\frac {n-1}2}}{\tha(p;p^2)\eta^{\frac{n^2-1}4}}, & n \text{{ odd}}.
\end{cases}
$$
\end{proof}

 Using \eqref{sevc} and \eqref{sevd}, one may check that
in the trigonometric case  $p=0$ (which implies $\eta=-1$ and 
$\zeta=-2$),
Corollary \ref{tcpca} reduces to \eqref{ttf}.

When $n=2$, we obtain the following  important consequence.

\begin{corollary}\label{tzc}
In the notation of \emph{Lemma \ref{fil}},
$$\tau(p)=\frac{\zeta^2+4\zeta+1}\eta. $$
\end{corollary}

By  Lemma \ref{ecl}, it follows that, when $t_i$ are given by \eqref{ti},
\begin{equation}\label{ttz}T(t_0,t_1,t_2)=\frac{2(\zeta^2+4\zeta+1)^3}{\zeta(\zeta+1)^4}. \end{equation}

\begin{proof}
Noting that 
$$Z_2^{3C}(t_0,t_1,t_2)= \frac{\tha(\la;p)^3+\tha(\la\om^2;p)^3}{\tha(\la;p)^9\tha(\la\om;p)^{12}\tha(\la\om^2;p)^9} $$
and that
$$\tau(p)=\frac{\tha(\la;p)^3+\tha(\la\om;p)^3+\tha(\la\om^2;p)^3}{\tha(\la^3;p^3)} =\lim_{\la\rightarrow\om^2}\frac{\tha(\la;p)^3+\tha(\la\om^2;p)^3}{\tha(\la^3;p^3)}, $$
we express $\tau(p)$ using the case $n=2$ of Corollary \ref{tcpca}. 
After simplification, we obtain 
$$\tau(p)=-\frac{\om^2\tha(-p,-\om;p^2)}{\tha(\om;p)^2\tha(p;p^2)\eta}\left(\zeta^2(\zeta+3)-(3\zeta+1)\right). $$
Factoring
$$\zeta^2(\zeta+3)-(3\zeta+1)=(\zeta-1)(\zeta^2+4\zeta+1), $$
applying \eqref{ztd} and  using \eqref{tqb},
 we arrive at the  stated result.
\end{proof}

\subsection{Return to the polynomials $q_n$ and $r_n$}
\label{qrss}

Combining Theorem \ref{st} and
  Corollary \ref{tcpca}, one easily recovers two of our main results: Theorem \ref{qrdp} and Theorem~\ref{pqrt}.

As a starting point, we note the theta function identities
\begin{subequations}\label{ctb}
\begin{multline}\tha(-\om^n\lambda^2;p^2)\\
=\frac{\om^2\tha(-\om;p^2)^2\tha(\la\om^{2n};p)^2-\om\tha(-1;p^2)\tha(\la\om^{2n+1},\la\om^{2n+2};p)}{\tha(\om;p)^2}, \end{multline}
\begin{multline}
\la\tha(-p\om^n\lambda^2;p^2)\\
=\om^{n+1}\frac{\tha(-p\om;p^2)^2\tha(\la\om^{2n};p)^2-\tha(-p;p^2)\tha(\la\om^{2n+1},\la\om^{2n+2};p)}{\tha(\om;p)^2}. 
\end{multline}
\end{subequations}
These can be obtained as special cases of \eqref{tad}, or  simply by noting that since they relate three functions in a two-dimensional space, it is enough to verify them at the points $\lambda=\om^n$ and $\lambda=\om^{n+2}$.

Using \eqref{ctb} in 
 Corollary \ref{tcpca}, we obtain an expression of the form \eqref{zfd}, where
for $n$ even,
\begin{equation*}\begin{split}X_n(p)&=-\frac{\om^2\tha(-\om;p)}{\tha(p;p^2)\tha(\om;p)^2\eta^{\frac{n^2}4}}\left(p_{n-1}(\zeta)-\zeta^{\frac n2}\tilde p_{n-1}(\zeta)\right)\\
 &=-\frac{p_{n-1}(\zeta)-\zeta^{\frac n2}\tilde p_{n-1}(\zeta)}{\eta^{\frac{n^2-4}4}(1-\zeta^2)},\end{split}\end{equation*}
\begin{equation*}\begin{split}Y_n(p)&=\frac{\om^2\tha(-\om,-p;p^2)}{\tha(p;p^2)\tha(\om;p)^2\eta^{\frac{n^2}4}}\left(p_{n-1}(\zeta)-\zeta^{\frac {n+2}2}\tilde p_{n-1}(\zeta)\right)\\
 &=\frac{p_{n-1}(\zeta)-\zeta^{\frac {n+2}2}\tilde p_{n-1}(\zeta)}{\eta^{\frac{n^2}4}(1-\zeta)},\end{split}\end{equation*}
where we used Lemma \ref{ztl} and \eqref{tq}.
Similarly, when $n$ is odd,
$$X_n(p)=\frac{p_{n-1}(\zeta)-\zeta^{\frac {n+1}2}\tilde p_{n-1}(\zeta)}{\eta^{\frac{n^2-1}4}(1-\zeta)}, $$
$$Y_n(p)=-2\frac{p_{n-1}(\zeta)-\zeta^{\frac {n-1}2}\tilde p_{n-1}(\zeta)}{\eta^{\frac{n^2-9}4}(1-\zeta)(1+\zeta)^3}. $$

We now compare these expressions
 with \eqref{aba}. For instance, when $n\equiv 0\ \operatorname{mod}\ 6$,
$$-X_n=\frac{p_{n-1}(\zeta)-\zeta^{\frac n2}\tilde p_{n-1}(\zeta)}{\eta^{\frac{n^2-4}4}(1-\zeta^2)}=\frac{Tq_n(T)}{\tau(p)}. $$
By Corollary \ref{tzc} and \eqref{ttz}, this can be written as
\begin{multline*}p_{n-1}(\zeta)-\zeta^{\frac n2}\tilde p_{n-1}(\zeta)=(1-\zeta)^2(\zeta^2+4\zeta+1)^{\frac{n^2-4}4}\\
\times\left(\frac{\zeta(\zeta+1)^4}{2(\zeta^2+4\zeta+1)^3)}\right)^{\frac{n^2}{12}-1}q_n\left(\frac{2(\zeta^2+4\zeta+1)^3)}{\zeta(\zeta+1)^4}\right). 
\end{multline*}
Since, by Corollary \ref{pctc}, $p_{n-1}(0)=1$, it follows that $q_n$ is monic of degree $n^2/12-1$, and that the final equation of Theorem \ref{pqrt} holds.
Repeating the same argument for all cases, one obtains
Theorem \ref{qrdp} and Theorem \ref{pqrt}.

Finally, we comment on the deduction of Corollary \ref{ecc} from Theorem \ref{qrdp}. This is a tedious  exercise, and we will only explain the case when $n\equiv 1\ \operatorname{mod}\ 6$ and we consider squares of colour $2$. In this case,
\begin{equation*}\begin{split}Z_n^{3C}(t_0,t_1,t_2)&=(t_0t_1t_2)^{\frac{n(n+2)}3}\left(\frac{t_0t_1}{t_2}\,q_{n}(T)-2\frac{t_0t_1+t_0t_2+t_1t_2}{t_2}\,r_{n}(T)\right)\\
&=(t_0t_1t_2)^{\frac{n(n+2)}3}\left(\frac{t_0t_1}{t_2}\,T^{\frac{n^2-1}{12}}-2\frac{t_0t_1+t_0t_2+t_1t_2}{t_2}\,T^{\frac{n^2-13}{12}}\right)+\cdots,\end{split}\end{equation*}
the ellipsis denoting lower terms in  $T$, hence also in  $t_2$, Writing 
$$ T=t_2\frac{(t_0+t_1)^3}{(t_0t_1)^2}+\cdots,$$
this can be simplified to
$$(t_0t_1)^{\frac{n^2+4n+7}6}t_2^{\frac{5n^2+8n-13}{12}}(t_0+t_1)^{\frac{n^2-9}4}(t_0^2+t_1^2)+\cdots.$$ 
Thus, the leading power of $t_2$ is $(5n^2+8n-13)/12$, and  the coefficient of
$$t_0^{\frac{n^2+4n+7}6+k}t_1^{\frac{5n^2+8n+11}{12}-k}t_2^{\frac{5n^2+8n-13}{12}} $$
is
$$\binom{(n^2-9)/4}{k}+\binom{(n^2-9)/4}{k-2}. $$
Repeating the same argument for each colour and each residue class of $n\ \operatorname{mod}\ 6$, one obtains Corollary \ref{ecc}.

\subsection{Proof of Proposition \ref{rip}}
\label{ics}

In this Section, we  prove that 
 $r_{2n+1}$ has integer coefficients. We need the following fact, which is generalized in Corollary \ref{pic}.

\begin{lemma}\label{pip}
The polynomial $2^{\big[\frac{n(n+2)}2\big]}p_n(\zeta/2)$ has integer coefficients.
\end{lemma}

For the proof,  we will use a determinant formula for $p_n$. 

\begin{lemma}\label{pdl}
The polynomials $p_n$ are given by
\begin{multline*}p_n(\zeta)=\frac{(-1)^{\binom{n+1}2}\zeta^{n^2}(\zeta+1)^{n^2}}{2^{n^2}\prod_{j=1}^n(j-1)!^2(1+2\zeta)^{\left[\frac{n^2}4\right]}\left(1+\frac\zeta 2\right)^{n^2+\left[\frac{(n-1)^2}4\right]}}\\
\times\det_{1\leq i,j\leq n}\left(\frac{\partial^{i+j-2}}{\partial x^{i-1}\partial y^{j-1}}\Bigg|_{x=2\zeta+1,\,y=\frac\zeta{\zeta+2}}\frac{F(x,y,2\zeta+1)}{G(x,y)}\right). \end{multline*}
\end{lemma}

\begin{proof} Let $x_i\rightarrow 2\zeta+1$, $y_i\rightarrow \zeta/(\zeta+2)$ and $z=2\zeta+1$ in \eqref{sn}.  The left-hand side can be expressed in terms of the polynomials $p_n$ using
\eqref{pe}. We now apply the well-known identity 
$$
\lim_{\substack{x_{1},\dots,x_n\rightarrow a,\\ y_{1},\dots,y_n\rightarrow b}}
\frac{\det_{1\leq i,j\leq n}(f(x_i,y_j))}{\prod_{1\leq i<j\leq n}(x_i-x_j)(y_i-y_j)}
=\frac{1}{\prod_{j=1}^{n}(j-1)!^2}
\det_{1\leq i,j\leq n}\left(\frac{\partial^{i+j-2}f(a,b)}{\partial x^{i-1}\partial y^{j-1}}\right).$$
 Simplifying the prefactor using \eqref{gsp}, we obtain the desired result.
\end{proof}

\begin{proof}[Proof of \emph{Lemma \ref{pip}}]
In Lemma \ref{pdl}, replace $x$ by $x+2\zeta+1$ and $y$ by $y+\zeta/(\zeta+2)$ inside the determinant, and then replace $\zeta$ by $\zeta/2$. Note that
\begin{equation}\label{mq}\frac{F(x+\zeta+1,y+\zeta/(\zeta+4),\zeta+1;\zeta/2)}{G(x+\zeta+1,y+\zeta/(\zeta+4);\zeta/2)}=\frac{(\zeta+2)(\zeta+4)A(x,y,\zeta)}{B(x,y,\zeta)},
 \end{equation}
where $A,\,B\in\mathbb Z[x,y,\zeta]$ and $B(0,0,\zeta)=\zeta^2(\zeta+2)^2$. It follows that the Taylor expansion of \eqref{mq} at $x=y=0$ has the form
$$\frac{(\zeta+4)}{\zeta^2(\zeta+2)}\sum_{k,l=0}^\infty\frac{C_{kl}(\zeta)}{(\zeta^2(\zeta+2)^2)^{k+l}}\,x^ky^l, $$
where $C_{kl}\in\mathbb Z[\zeta]$.
Using  \eqref{eli},
Lemma \ref{pdl} can then be written 
$$2^{\big[\frac{n(n+2)}2\big]}p_n(\zeta/2)=\frac{(-1)^{\binom{n+1}2}\det_{1\leq i,j\leq n}\left(C_{i-1,j-1}(\zeta)\right)}{\zeta^{n^2}(\zeta+2)^{n(n-1)}(\zeta+1)^{\left[\frac{n^2}4\right]}\left(\zeta+4 \right)^{n(n-1)+\left[\frac{(n-1)^2}4\right]}}.$$
The right-hand side is a
 quotient of two polynomials with integer coefficients. Since the denominator is monic, the left-hand side  has integer coefficients. 
\end{proof}

\begin{proof}[Proof of \emph{Proposition \ref{rip}}]
Let $n$ be even, and
write \eqref{pqre} in the form
\begin{multline}\label{pqrm}
\sum_{k=0}^{n(n+1)/2}2^{k-\frac{n(n+2)}2}b_k\zeta^k=
\sum_{k=0}^{[n(n+2)/12]}\frac{c_k}{2^k}\,\zeta^k(\zeta+1)^{4k}(\zeta^2+4\zeta+1)^{\frac{n(n+2)}4-3k}\\
-
\sum_{k=0}^{[(n-2)(n+4)/12]}\frac{d_k}{2^k}\,\zeta^{k+1}(\zeta+1)^{4k+3}(\zeta^2+4\zeta+1)^{\frac{(n-2)(n+4)}4-3k}.
 \end{multline}
Here, $b_k$ are coefficients of the polynomial $2^{\frac{n(n+2)}2}p_n(\zeta/2)$, which are integer by Lemma \ref{pip}. The numbers
$c_k$ are coefficients of $\tilde q_{n+1}$, which are also integers, see
 \S \ref{qrs}.  We need to show that $d_k$ are integers. To this end, let $0\leq m\leq [(n-2)(n+4)/12]$, pick out the coefficient of $\zeta^{\frac{n(n+2)}2-m}$
on both sides of \eqref{pqrm}, and multiply the result by $2^m$. This leads to a triangular system of the form
$$c_m-d_m+\sum_{j<m}(\mathbb Z c_j+\mathbb Z d_j)=\begin{cases}
 0, & 0\leq m< \frac n 2,\\
 b_{\frac{n(n+2)}2-m}, & \frac n2\leq m\leq \big[\frac{(n-2)(n+4)}{12}\big].
\end{cases} $$
By induction on      $m$,
 it follows  that each $d_m$  is an integer.
\end{proof}

Proposition \ref{rip} implies an integrality result for the coefficients of $p_n$. Although it seems to be far from sharp, it may still have some interest.

\begin{corollary}\label{pic}
Let $p_n(\zeta)=\sum_{k=0}^{n(n+1)/2}a_k\zeta^k$. Then, $2^{\mu_k}a_k\in\mathbb Z$, where
 $$\mu_k=\begin{cases} \min\left(\big[\frac{n(n+2)}{12}\big],k,\frac{n(n+2)}2-k\right), & n \text{\emph{ even}},\\
\min\left(\big[\frac{(n+1)^2}{12}\big],k,\frac{(n+1)^2}2-k\right), & n \text{\emph{ odd}}.
\end{cases} $$
\end{corollary}

 \begin{proof}
When $n$ is even, this follows easily from \eqref{pqrm}, using that 
$2^{k-\frac{n(n+2)}2}b_k=a_k$
and that $c_k$ and $d_k$ are integers. For odd $n$, the proof is similar.
\end{proof}

\section{Thermodynamic limit}
\label{tds}

In this Section, we give a formal derivation of our conjectured expression \eqref{fee} for the free energy. Essentially, we assume that the free energy can be expressed in terms of a  formal power series in the variable $\zeta$, and then 
 deduce  \eqref{fee} by formally taking $n\rightarrow\infty $ in
 Proposition \ref{pyr}. 
 To turn this into a rigorous proof would require strong analyticity assumptions that seem difficult to verify a priori. However, one can observe that \eqref{fee} gives the correct result not only as $\zeta\rightarrow 0$, but also at $\zeta=1$ (because of \eqref{ttf} and \eqref{aa}) and as $\zeta\rightarrow\infty$ (since it has the correct behaviour under $\zeta\mapsto\zeta^{-1}$). It also seems to agree  with numerical experiment.

We will assume that 
\begin{equation}\label{g}\lim_{n\rightarrow\infty}\frac{\log(p_n(\zeta))}{n^2}=\log g(\zeta) \end{equation}
exists in the algebra of formal power series in $\zeta$. 
We also assume that the convergence is regular in the sense that if we factor
\begin{equation}\label{pa}p_n(\zeta)=g(\zeta)^{n^2}\phi_n(\zeta), \end{equation}
then
\begin{equation}\label{gr}\lim_{n\rightarrow\infty} \frac{\phi_n(\zeta)^2}{\phi_{n+2}(\zeta)\phi_{n-2}(\zeta)}=1, \end{equation}
still  in the sense of formal power series.
As a weak motivation for this assumption, we note that, 
by Proposition  \ref{psv} and \eqref{aa}, 
 $\phi_n(1)\sim C n^\alpha \beta^n$, where $C$ is different for even and odd $n$. Thus, \eqref{gr} holds pointwise at $\zeta=1$, whereas the limit
$$\lim_{n\rightarrow\infty} \frac{\phi_n(\zeta)^2}{\phi_{n+1}(\zeta)\phi_{n-1}(\zeta)}\Bigg|_{\zeta=1} $$
 does not  exist.

Let us now apply Proposition \ref{pyr}, which we write in the form
$$(1+\zeta)^2p_{n+1}(\zeta)p_{n-1}(\zeta)=(1+2\zeta)^{1+\chi(n \text{{ even}})}\left(1+\frac \zeta 2\right)^{1+\chi(n \text{{ odd}})}p_n(\zeta)^2+\mathcal O(\zeta^{n+1}), $$
where $\mathcal O(\zeta^{n+1})$  is a polynomial divisible by $\zeta^{n+1}$.
Iterating this gives
$$(1+\zeta)^8p_{n+2}(\zeta)p_{n-2}(\zeta)=(1+2\zeta)^{6}\left(1+\frac \zeta 2\right)^{6}p_n(\zeta)^2+\mathcal O(\zeta^{n}). $$
Applying the factorization \eqref{pa} we obtain
$$ g(\zeta)^8=\frac{(1+2\zeta)^{6}\left(1+\frac \zeta 2\right)^{6}}{(1+\zeta)^8}\cdot\frac{\phi_n(\zeta)^2}{\phi_{n+2}(\zeta)\phi_{n-2}(\zeta)}+\frac{\mathcal O(\zeta^{n})}{(1+\zeta)^8g(\zeta)^{2n^2}\phi_{n+2}(\zeta)\phi_{n-2}(\zeta)}.$$ 
Letting  $n\rightarrow\infty $ and using \eqref{gr} we conclude that
$$g(\zeta)=\frac{(1+2\zeta)^{\frac 34}\left(1+\frac \zeta 2\right)^{\frac 34}}{1+\zeta}.$$
Replacing $\zeta$ by $\zeta^{-1}$ in \eqref{g} then gives
\begin{equation}\label{gt}\lim_{n\rightarrow\infty}\frac{\log(\tilde p_n(\zeta))}{n^2}=\log g(\zeta)\end{equation}
Finally, using \eqref{g} and \eqref{gt} 
in Corollary \ref{tcpca} gives
$$\lim_{n\rightarrow\infty}\frac{\log Z_n^{3C}(t_0,t_1,t_2)}{n^2}=\log\left(\frac{g(\zeta)}{\eta^{1/4}\tha(\lambda^3;p^3)}\right)
=\frac 13\log(t_0t_1t_2)+\log\left(\frac{g(\zeta)}{\eta^{1/4}}\right).
 $$
By Lemma \ref{ecl}, this may be written as in \eqref{fee}.

 \end{document}